\documentclass{amsart}

\usepackage[mathscr]{euscript}
\DeclareFontFamily{OT1}{rsfs}{}
\DeclareFontShape{OT1}{rsfs}{n}{it}{<-> rsfs10}{}
\DeclareMathAlphabet{\curly}{OT1}{rsfs}{n}{it}

\makeatletter
\newcommand{\eqnum}{\refstepcounter{equation}\textup{\tagform@{\theequation}}}
\makeatother

\newcommand\beq[1]{\begin{equation}\label{#1}}
\newcommand\eeq{\end{equation}}
\newcommand\beqa{\begin{eqnarray*}}
\newcommand\eeqa{\end{eqnarray*}}

\title[Hall-type algebras for categorical DT theories]{Hall-type algebras for 
categorical Donaldson-Thomas theories 
on local surfaces}
\date{}
\author{Yukinobu Toda}

\usepackage{amscd}
\usepackage{amsmath}
\usepackage{amssymb}
\usepackage{amsthm}
\usepackage{float}
\usepackage[dvips]{graphicx}
\usepackage{xypic}

\usepackage{array}
\usepackage{amscd}
\usepackage[all]{xy}

\usepackage{tabulary}
\usepackage{booktabs}

\usepackage{amssymb, paralist, xspace, url, amscd, euscript, mathrsfs, stmaryrd}

\DeclareFontFamily{U}{rsfs}{%
\skewchar\font127}
\DeclareFontShape{U}{rsfs}{m}{n}{%
<-6>rsfs5<6-8.5>rsfs7<8.5->rsfs10}{}
\DeclareSymbolFont{rsfs}{U}{rsfs}{m}{n}
\DeclareSymbolFontAlphabet
{\mathrsfs}{rsfs}
\DeclareRobustCommand*\rsfs{%
\@fontswitch\relax\mathrsfs}

\theoremstyle{plain}
\newtheorem{thm}{Theorem}[section]
\newtheorem{prop}[thm]{Proposition}
\newtheorem{lem}[thm]{Lemma}

\newtheorem{defi}[thm]{Definition}
\newtheorem{rmk}[thm]{Remark}
\newtheorem{cor}[thm]{Corollary}

\newtheorem{prop-defi}[thm]{Proposition-Definition}
\newtheorem{thm-defi}[thm]{Theorem-Definition}
\newtheorem{lem-defi}[thm]{Lemma-Definition}

\newtheorem{exam}[thm]{Example}

\newcommand{\aA}{\mathcal{A}}
\newcommand{\bB}{\mathcal{B}}
\newcommand{\cC}{\mathcal{C}}
\newcommand{\dD}{\mathcal{D}}
\newcommand{\eE}{\mathcal{E}}
\newcommand{\fF}{\mathcal{F}}

\newcommand{\hH}{\mathcal{H}}
\newcommand{\iI}{\mathcal{I}}

\newcommand{\lL}{\mathcal{L}}
\newcommand{\mM}{\mathcal{M}}
\newcommand{\nN}{\mathcal{N}}
\newcommand{\oO}{\mathcal{O}}
\newcommand{\pP}{\mathcal{P}}

\newcommand{\rR}{\mathcal{R}}
\newcommand{\sS}{\mathcal{S}}
\newcommand{\tT}{\mathcal{T}}
\newcommand{\uU}{\mathcal{U}}
\newcommand{\vV}{\mathcal{V}}
\newcommand{\wW}{\mathcal{W}}

\newcommand{\zZ}{\mathcal{Z}}

\newcommand{\fE}{\mathfrak{E}}
\newcommand{\ffF}{\mathfrak{F}}

\newcommand{\fH}{\mathfrak{H}}
\newcommand{\fI}{\mathfrak{I}}

\newcommand{\fM}{\mathfrak{M}}
\newcommand{\fN}{\mathfrak{N}}

\newcommand{\fP}{\mathfrak{P}}

\newcommand{\fS}{\mathfrak{S}}

\newcommand{\fU}{\mathfrak{U}}

\newcommand{\Supp}{\mathop{\rm Supp}\nolimits}
\newcommand{\Hom}{\mathop{\rm Hom}\nolimits}

\newcommand{\dotimes}{\stackrel{\textbf{L}}{\otimes}}

\newcommand{\dR}{\mathbf{R}}
\newcommand{\dL}{\mathbf{L}}

\newcommand{\Pic}{\mathop{\rm Pic}\nolimits}

\newcommand{\id}{\textrm{id}}

\newcommand{\ch}{\mathop{\rm ch}\nolimits}

\newcommand{\Ext}{\mathop{\rm Ext}\nolimits}
\newcommand{\Spec}{\mathop{\rm Spec}\nolimits}
\newcommand{\rank}{\mathop{\rm rank}\nolimits}
\newcommand{\Coh}{\mathop{\rm Coh}\nolimits}

\newcommand{\ev}{\mathop{\rm ev}\nolimits}
\newcommand{\pure}{\mathop{\rm pure}\nolimits}
\newcommand{\fin}{\mathop{\rm fin}\nolimits}
\newcommand{\pt}{\mathop{\rm pt}\nolimits}
\newcommand{\us}{\mathchar`-\rm{us}}
\newcommand{\sss}{\mathchar`-\rm{ss}}

\newcommand{\cneq}{\mathrel{\raise.095ex\hbox{:}\mkern-4.2mu=}}
\newcommand{\eqcn}{\mathrel{=\mkern-4.5mu\raise.095ex\hbox{:}}}

\newcommand{\Cok}{\mathop{\rm Cok}\nolimits}
\newcommand{\ext}{\mathop{\rm ext}\nolimits}

\newcommand{\perf}{\mathop{\rm perf}\nolimits}

\newcommand{\Imm}{\mathop{\rm Im}\nolimits}

 \newcommand{\RHom}{\mathop{\dR\mathrm{Hom}}\nolimits}
\newcommand{\Ker}{\mathop{\rm Ker}\nolimits}

\newcommand{\FM}{\mathop{\rm FM}\nolimits}

\newcommand{\vdim}{\mathop{\rm vdim}\nolimits}

\newcommand{\cl}{\mathop{\rm cl}\nolimits}

\newcommand{\Crit}{\mathop{\rm Crit}\nolimits}

\newcommand{\Dbc}{D^b_{\rm{coh}}}

\newcommand{\lgakko}{(\!(}
\newcommand{\rgakko}{)\!)}

\makeatletter
 \renewcommand{\theequation}{%
   \thesection.\arabic{equation}}
  \@addtoreset{equation}{section}
\makeatother

\begin{document}

\begin{abstract}
We show that the categorified cohomological Hall 
algebra structures on surfaces constructed by Porta-Sala descend to 
those on Donaldson-Thomas categories 
on local surfaces introduced in the author's previous paper.
A similar argument also shows that
Pandharipande-Thomas categories on local 
surfaces admit actions of categorified COHA 
for zero dimensional sheaves on surfaces.
We also construct 
annihilator actions
of its simple operators, 
and show that their commutator in the K-theory
satisfies the 
relation similar to the one of Weyl algebras.
This result may be regarded as a categorification
of Weyl algebra action on homologies of Hilbert schemes of 
points on locally planar curves due to Rennemo, which is 
relevant for Gopakumar-Vafa formula of generating 
series of PT invariants.  
\end{abstract}

\maketitle

\setcounter{tocdepth}{1}
\tableofcontents

\section{Introduction}
\subsection{Motivation and Background}
The Donaldson-Thomas (DT for short) invariants 
virtually count stable coherent sheaves on Calabi-Yau (CY for short) 3-folds~\cite{MR1818182}, 
and play important roles in
the recent study of curve counting theories and mathematical physics.
The original DT invariants are integer valued invariants, 
which coincide with weighted Euler characteristics
 of Behrend functions~\cite{Beh}
on moduli spaces of stable sheaves on CY 3-folds.
The Behrend functions  
are point-wise Euler characteristics of
vanishing cycles, so locally 
admit natural refinements such as motivic 
vanishing cycles, perverse sheaves of vanishing cycles. 
Based on this fact,  
Kontsevich-Soibelman~\cite{K-S, MR2851153} 
proposed several refinements of DT invariants,
e.g. motivic DT invariants, cohomological 
DT invariants, whose 
foundations are
now available in~\cite{BDM, MR3353002}. 
Among them, cohomological DT invariants  
are expected to carry an algebra structure,
which globalize critical cohomological Hall algebras
for quivers with super-potentials constructed in~\cite{MR2851153, MR3667216}. 

In the author's previous paper~\cite{TocatDT}, 
we proposed further refinement of DT invariants to 
triangulated (or dg) categories, called 
\textit{categorical DT theories}
(or \textit{DT categories} for short). 
The DT categories should be constructed as gluing 
of locally defined categories of matrix factorizations, 
but their general construction is still beyond our 
scope (see
 also see~\cite[(J)]{Jslide}, \cite[Section~6.1]{MR3728637}). 
In~\cite{TocatDT} we  
constructed $\mathbb{C}^{\ast}$-equivariant DT categories in the special case of 
CY 3-folds, called \textit{local surfaces}, i.e. 
the total spaces of canonical line bundles on 
surfaces. 
In this case,
they are defined to be
the Verdier quotients of 
derived categories of coherent sheaves on derived 
moduli stacks of coherent sheaves on surfaces, by 
the subcategory of objects whose
singular supports (in the sense of Arinkin-Gaitsgory~\cite{MR3300415})
are contained in the unstable loci. 
Via Koszul duality, the DT category is locally equivalent to the 
category of $\mathbb{C}^{\ast}$-equivariant matrix factorizations. 
The DT categories are expected to recover
the cohomological 
DT invariants on local surfaces by taking the periodic cyclic 
homologies, so their Euler characteristics should recover the original DT invariants. 

The purpose of this paper is to 
construct Hall-type algebra structures on 
DT categories for local surfaces.  
Our construction is 
induced by the categorification of cohomological Hall algebra 
(COHA for short) structures
on surfaces constructed by Porta-Sala~\cite{PoSa}.
A new point in this paper is that 
their product structure is compatible with 
singular supports in a certain sense, so 
that it
descends to the product structure on 
DT categories. 
By taking the 
associated product on 
K-theory, we obtain a
globalization of 
K-theoretic Hall algebras for quivers with super-potentials
constructed by P{\u{a}}durairu~\cite{Tudor}.
Moreover by taking  
periodic cyclic homologies, 
we expect that our construction gives a 
critical 
COHA restricted 
to the semistable locus, giving a globalization of 
critical COHA defined for quivers with 
super-potentials~\cite{MR2851153, MR3667216}. 

One of our motivations of this study is
to construct some algebra actions on 
Pandharipande-Thomas (PT for short) categories for local surfaces, which 
are relevant for the Gopakumar-Vafa (GV for short) formula
of the generating series of PT invariants.
By a similar argument as above, 
we also show
that PT categories for local surfaces admit 
actions of
DT categories of zero dimensional sheaves. 
We also construct 
annihilator actions
of its simple operators, 
and show that their commutator
in the K-theory 
satisfies the 
relation similar to the one of Weyl algebras.
This result may be regarded as a categorification  
of Weyl algebra action on homologies of Hilbert schemes of 
points on locally planar curves due to Rennemo,
where he derived GV formula of 
the generating series of 
Euler characteristics of these Hilbert schemes. 
Therefore 
 we expect that this work may be relevant for 
categorical understanding of 
conjectural PT/GV formula, which 
will be pursued in a future work. 

\subsection{Categorical DT theory for local surfaces}
Let $S$ be a smooth projective surface over $\mathbb{C}$. 
We consider the total space of its canonical line bundle, 
which is a non-compact CY 3-fold
\begin{align}\label{intro:pi}
\pi \colon 
X\cneq \mathrm{Tot}_S(\omega_S) \to S.
\end{align}
Let $N(S)$ be the numerical Grothendieck group of $S$ 
and take $v \in N(S)$. 
For a choice of a stability condition $\sigma$ on 
the abelian category of compactly supported coherent sheaves on $X$
(e.g. Gieseker stability condition), 
we have the 
moduli stack
$\mM_{X}^{\sigma\mathchar`-\rm{ss}}(v)$
of $\sigma$-semistable sheaves 
$E$ on $X$
with $[\pi_{\ast}E]=v$, 
together with the commutative diagram
\begin{align}\label{intro:stack}
\xymatrix{
\mM_{X}^{\sigma\mathchar`-\rm{ss}}(v) 
\ar@<-0.3ex>@{^{(}->}[r]
\ar[d]_-{\pi_{\ast}} &
t_0(\Omega_{\fM_S(v)}[-1]) \ar[d] \ar@<-0.3ex>@{^{(}->}[r] 
& \Omega_{\fM_S(v)}[-1] \ar[d] \\
\mM_S(v) \ar@{=}[r] & \mM_S(v) \ar@<-0.3ex>@{^{(}->}[r] & \fM_S(v) . 
} 
\end{align}
Here $\fM_S(v)$ is the derived 
moduli stack of coherent sheaves $F$ on $S$ with $[F]=v$, 
$t_0(-)$ means the classical truncation with 
$\mM_S(v)=t_0(\fM_S(v))$, 
and $\Omega_{\fM_S(v)}[-1]$ is the $(-1)$-shifted cotangent stack 
over $\fM_S(v)$.
The top left horizontal arrow is an open immersion and 
the right horizontal arrows are closed immersions. 

In~\cite{TocatDT}, 
the $\mathbb{C}^{\ast}$-equivariant 
categorical DT theory (or simply DT category)
associated with the moduli stack $\mM_X^{\sigma \sss}(v)$
is defined to be the Verdier quotient
(see Definition~\ref{catDT:stable})
\begin{align}\label{intro:DTcat}
\mathcal{DT}^{\mathbb{C}^{\ast}}(\mM_{X}^{\sigma\mathchar`-\rm{ss}}(v))
\cneq \Dbc(\fM_S(v)^{\fin})/\cC_{\zZ_{\sigma \us}(v)^{\fin}}.
\end{align}
Here $\fM_S(v)^{\fin}$ is a  
derived open substack of 
$\fM_S(v)$ which is of finite type and 
contains the image of the map $\pi_{\ast}$ in 
 (\ref{intro:pi}).
Furthermore the subcategory
\begin{align*}
\cC_{\zZ_{\sigma \us}(v)^{\fin}}
\subset \Dbc(\fM_S(v)^{\fin})
\end{align*} 
consists of objects
$E \in \Dbc(\fM_S(v)^{\fin})$ satisfying that 
\begin{align*}
\mathrm{Supp}^{\rm{sg}}(E) \subset 
\zZ_{\sigma \us}(v)^{\fin} \cneq 
t_0(\Omega_{\fM_S(v)^{\fin}}[-1]) \setminus 
\mM_{X}^{\sigma\mathchar`-\rm{ss}}(v). 
\end{align*}
Here $\mathrm{Supp}^{\rm{sg}}(E)$ is the singular 
support of $E$ introduced by Arinkin-Gaitsgory~\cite{MR3300415}, 
following an 
earlier work by Benson-Iyengar-Krause~\cite{MR2489634}. 
Via Koszul duality, the category (\ref{intro:DTcat}) 
is locally equivalent to the category of $\mathbb{C}^{\ast}$-equivariant 
matrix factorizations of functions, whose critical loci 
locally describe $\mM_{X}^{\sigma\mathchar`-\rm{ss}}(v)$. 
If there is no strictly $\sigma$-semistable sheaves, 
the $\mathbb{C}^{\ast}$-rigidified version of the 
triangulated category (\ref{intro:DTcat})
is expected to recover the 
cohomological Donaldson-Thomas invariants associated with (\ref{intro:stack})
by taking its periodic cyclic homology. 
A key point of the construction (\ref{intro:DTcat}) is that, 
we capture (un)stable loci on moduli stacks of sheaves on 3-folds 
via derived structures of those on surfaces using 
singular supports. In other words, 
we define the category (\ref{intro:DTcat})
as if we have the `dimension reduction' 
for DT categories, similarly to the dimension 
reduction for critical COHA for some quivers with 
super-potentials coming from preprojective algebras
proved in~\cite{MR3727563, YangZhao}. 

\subsection{Categorified COHA}
Our first result is to show the existence of Hall-type 
algebra structure on the DT categories (\ref{intro:DTcat}).
Let us recall Porta-Sala's construction~\cite{PoSa} of 
categorified COHA for surfaces. 
We take $v_{\bullet}=(v_1, v_2, v_3) \in N(S)^{\times 3}$ with $v_2=v_1+v_3$.
Then the functor
\begin{align}\label{intro:catCOHA}
\Dbc(\fM_S(v_1)) \times \Dbc(\fM_S(v_3)) \to \Dbc(\fM_S(v_2))
\end{align}
is constructed
in~\cite{PoSa}
by using the following Hall-type diagram\footnote{More precisely 
Porta-Sala's construction work in a dg-categorical 
setting, where 
not only the stacks of exact sequences 
but also higher parts of the Waldhausen construction
are required in order to control the higher associativity. 
Porta and Sala pointed out that the same would apply to 
our situation, see Remark~\ref{rmk:higher}.
}
\begin{align}\label{intro:dia:extS}
\xymatrix{
\fM_S^{\ext}(v_{\bullet}) \ar[r]^-{\ev_2} \ar[d]_-{(\ev_1, \ev_3)} & \fM_S(v_2)
 \\
\fM_S(v_1) \times \fM_S(v_3). & 
}
\end{align}
Here $\fM_S^{\ext}(v_{\bullet})$ is the derived moduli
stack of short exact sequences of coherent sheaves on $S$
\begin{align*}
0 \to F_1 \to F_2 \to F_3 \to 0, \ [F_i]=v_i, \ \ev_i(F_{\bullet})=F_i. 
\end{align*}
The functor (\ref{intro:catCOHA}) 
is defined by pull-back/push-forward of the diagram (\ref{intro:dia:extS}), 
and regarded as a categorification of 
two dimensional cohomological Hall algebra by Kapranov-Vasserot~\cite{KaVa2}. 

We show that 
the product functor (\ref{intro:catCOHA}) 
descends to the functor on DT categories. 
In order to state the statement, 
we fix a polynomial $\overline{\chi} \in \mathbb{Q}[m]$
and denote by $N(S)_{\overline{\chi}} \subset N(S)$
the subgroup 
of $v \in N(S)$ whose reduced Hilbert polynomial 
is equal to $\overline{\chi}$, or $v=0$.  
Our first result is as follows. 
\begin{thm}\emph{(Theorem~\ref{thm:COHA})}\label{intro:thmA}
For $(v_1, v_2, v_3) \in N(S)_{\overline{\chi}}$ with 
$v_2=v_1+v_3$, the 
functor (\ref{intro:catCOHA})
descends to the functor
\begin{align*}
\mathcal{DT}^{\mathbb{C}^{\ast}}(\mM_{X}^{\sigma\mathchar`-\rm{ss}}(v_1))
\times \mathcal{DT}^{\mathbb{C}^{\ast}}(\mM_{X}^{\sigma\mathchar`-\rm{ss}}(v_3))\to \mathcal{DT}^{\mathbb{C}^{\ast}}(\mM_{X}^{\sigma\mathchar`-\rm{ss}}(v_2)). 
\end{align*}
In particular, the direct sum of the K-theory
\begin{align}\label{intro:K}
\bigoplus_{v \in N(S)_{\overline{\chi}}}
K(\mathcal{DT}^{\mathbb{C}^{\ast}}(\mM_{X}^{\sigma\mathchar`-\rm{ss}}(v)))
\end{align}
has a structure of an associative algebra. 
\end{thm}

We note that if any $\sigma$-semistable sheaf on 
$X$ is push-forward to a $\sigma$-semistable sheaf on $S$, 
then we have 
\begin{align*}
\bigoplus_{v \in N(S)_{\overline{\chi}}}
K(\mathcal{DT}^{\mathbb{C}^{\ast}}(\mM_{X}^{\sigma\mathchar`-\rm{ss}}(v)))
=\bigoplus_{v \in N(S)_{\overline{\chi}}}
K(\mM_S^{\sigma\mathchar`-\rm{ss}}(v))
\end{align*}
and the algebra structure on it 
is essentially the same one 
in~\cite{PoSa}. 
For example as we will see in (\ref{intro:K0}), this happens 
when $\overline{\chi}\equiv 1$. 
However this is not the case in general, and 
our construction of the algebra (\ref{intro:K}) is new 
when the push-forward does not preserve
the $\sigma$-semistability. 

\subsection{Categorical PT theories on local surfaces}
By definition, a PT stable pair on $X$
consists of 
a pair
\begin{align}\label{intro:PTpair}
(F, s), \ s \colon \oO_X \to F
\end{align}
where $F$ is a compactly supported 
pure one dimensional coherent sheaf on $X$ and 
$s$ is surjective in dimension one. 
For $\beta \in \mathrm{NS}(S)$ and 
$n \in \mathbb{Z}$, we have the
moduli space of Pandharipande-Thomas stable pairs
\begin{align*}
P_n(X, \beta)
\end{align*}
which parametrizes 
pairs (\ref{intro:PTpair}) satisfying 
$\pi_{\ast}[F]=\beta$ and $\chi(F)=n$. 
The DT type invariants defined from $P_n(X, \beta)$, 
called \textit{PT invariants}, are defined in~\cite{MR2545686}
and play an important role in 
the recent study of curve counting invariants 
on CY 3-folds (see~\cite{MR2813335, MR2669709, MR2892766, MR3221298} 
for example).

The moduli space $P_n(X, \beta)$
fits into the commutative diagram
\begin{align*}
\xymatrix{
P_n(X, \beta) 
\ar@<-0.3ex>@{^{(}->}[r]
\ar[d]_-{\pi_{\ast}} & t_0(\Omega_{\fM_S^{\dag}(\beta, n)}[-1]) \ar[d] 
\ar@<-0.3ex>@{^{(}->}[r] & \Omega_{\fM_S^{\dag}(\beta, n)}[-1] \ar[d] 
\\
\mM_S^{\dag}(\beta, n) \ar@{=}[r] & \mM_S^{\dag}(\beta, n)
\ar@<-0.3ex>@{^{(}->}[r] & \fM_S^{\dag}(\beta, n). 
} 
\end{align*}
Here $\fM_S^{\dag}(\beta, n)$ is the derived moduli 
stack of 
pairs $(F, \xi)$, where $F$ is a one dimensional coherent 
sheaf on $S$ satisfying $[F]=\beta$, $\chi(F)=n$, 
and $\xi \colon \oO_S \to F$ is a morphism. 
Similarly to (\ref{intro:DTcat}),  
we also introduced the notion of $\mathbb{C}^{\ast}$-equivariant
categorical PT theory in~\cite{TocatDT} 
by the following 
(see Definition~\ref{def:catPT} for details)
\begin{align}\label{intro:catPT}
\dD \tT^{\mathbb{C}^{\ast}}(P_n(X, \beta)) \cneq 
D^b_{\rm{coh}}(\mathfrak{M}_S^{\dag}(\beta, n)^{\fin})/ 
\cC_{\zZ_{P\us}(\beta, n)^{\fin}}.
\end{align} 

On the other hand
if we take $\overline{\chi}$ to be the constant polynomial $1$,
then $N(S)_{\overline{\chi}}$ consists of $m[\mathrm{pt}]$ where 
$[\mathrm{pt}]$ is the numerical class of $\oO_x$ for $x \in S$. 
Then (\ref{intro:K}) is  
\begin{align}\label{intro:K0}
\bigoplus_{m\ge 0}
K(\mathcal{DT}^{\mathbb{C}^{\ast}}(\mM_{X}^{\sigma\mathchar`-\rm{ss}}(m[\mathrm{pt}])))
=\bigoplus_{m\ge 0} K(\mM_S(m[\mathrm{pt}])
\end{align} 
where $\mM_S(m[\mathrm{pt}])$ is the moduli stack of 
zero dimensional sheaves on $S$ with length $m$. 
The above algebra is nothing but the 
K-theoretic Hall algebra for zero dimensional sheaves on the surface $S$
constructed by Zhao~\cite{Zhao}, and is related to 
the shuffle algebra as proved in~\textit{loc.~cit.~}. 
We show that the above algebra acts on PT categories. 
\begin{thm}\emph{(Theorem~\ref{thm:PTaction})}\label{intro:thmB}
There exist functors
\begin{align}\label{intro:PTaction}
\dD \tT^{\mathbb{C}^{\ast}}(P_n(X, \beta))
\times \mathcal{DT}^{\mathbb{C}^{\ast}}(\mM_{X}^{\sigma\mathchar`-\rm{ss}}(m[\mathrm{pt}]))
\to 
\dD \tT^{\mathbb{C}^{\ast}}(P_{n+m}(X, \beta))
\end{align}
which induce the right action of the K-theoretic Hall algebra 
of zero dimensional sheaves (\ref{intro:K0})
to the following direct sum
\begin{align}\label{intro:catKPT}
\bigoplus_{n\in \mathbb{Z}}
K(\dD \tT^{\mathbb{C}^{\ast}}(P_n(X, \beta))). 
\end{align}
\end{thm}
In~\cite{TocatDT}, 
we also defined MNOP categories 
associated with moduli spaces of 
one or zero dimensional subschemes in $X$. 
In Section~\ref{sec:other}, we also show that 
the algebra (\ref{intro:K0}) acts on 
MNOP categories from the \textit{left} (not right). 
Since the PT moduli spaces and MNOP moduli spaces are related by 
wall-crossing, the fact that 
the direction of the action changes 
may be an incarnation of wall-crossing phenomena in terms of 
actions of Hall-type algebras. 

We will also consider the algebra (\ref{intro:K})
for one dimensional semistable sheaves on 
$X$, and show that it acts on the left/right on 
DT type categories associated with stable 
D0-D2-D6 bound states. 
Similarly to above, 
the direction of the action
of DT type categories of one dimensional semistable sheaves also 
changes
when we crosses the wall of weak stability 
conditions on the category of D0-D2-D6 bound states.
Since the wall-crossing in D0-D2-D6 bound states is relevant 
for GV formula of generating series of PT invariants (see~\cite{MR2892766}), 
the above observation may be relevant for categorical understanding of 
GV formula. This direction of research will be pursued in future.

\subsection{Commutator relations of Hecke actions}
Since we have $\mM_S([\pt])=S \times B\mathbb{C}^{\ast}$, 
we have the decomposition
\begin{align*}
K(\mM_S([\pt]))=\bigoplus_{k \in \mathbb{Z}}K(S)_k
\end{align*}
where $K(S)_k$ is the $\mathbb{C}^{\ast}$-weight $k$-part, 
which is isomorphic to $K(S)$.  
Therefore the action (\ref{intro:PTaction}) on the K-theory 
for $m=1$ and weight $k$-part induces the 
creation operators
\begin{align*}
\mu_{\eE, k}^+ \colon 
K(\dD \tT^{\mathbb{C}^{\ast}}(P_n(X, \beta)))
\to K(\dD \tT^{\mathbb{C}^{\ast}}(P_{n+1}(X, \beta)))
\end{align*}
for each $\eE \in K(S)$. 
We also construct 
annihilator operators
\begin{align*}
\mu_{\eE, k}^- \colon 
K(\dD \tT^{\mathbb{C}^{\ast}}(P_{n+1}(X, \beta)))
\to K(\dD \tT^{\mathbb{C}^{\ast}}(P_{n}(X, \beta)))
\end{align*}
and form the 
following maps
of degree $\pm 1$
\begin{align*}
\mu_{\eE}^{\pm}(z)
\cneq \sum_{k\in\mathbb{Z}} \frac{\mu_{\eE, k}^{\pm}}{z^{k}} \colon \bigoplus_{n \in \mathbb{Z}}
K(\dD \tT^{\mathbb{C}^{\ast}}(P_n(X, \beta)))
\to \bigoplus_{n\in \mathbb{Z}}
K(\dD \tT^{\mathbb{C}^{\ast}}(P_{n}(X, \beta)))\{z\}. 
\end{align*}
We then compute the 
commutator relation of the above operators.
For a fixed $\beta$, we set
\begin{align*}
\fM_S^{\dag}(\beta) \cneq \coprod_{n \in \mathbb{Z}} \fM_S^{\dag}(\beta, n)
\end{align*}
and we denote by 
\begin{align*}
(\oO_{S \times \fM_S^{\dag}(\beta)} \to \ffF(\beta))
\in \mathrm{Perf}\left(
S \times
\fM_S^{\dag}(\beta) \right)
\end{align*}
the universal pair. 
Let $p_{\fM} \colon S \times \fM_S^{\dag}(\beta) \to \fM_S^{\dag}(\beta)$
be the projection. We have the following result. 

\begin{thm}\emph{(Theorem~\ref{thm:relation})}\label{intro:thmC}
For a K-group element $(-)$ 
represented by perfect complexes, 
we have the following commutator relation
\begin{align}\label{intro:commu}
[\mu_{\eE_1}^+(z), \mu_{\eE_2}^-(w)](-)=(-) \otimes 
p_{\fM\ast}\left(
(\eE_1 \otimes \eE_2 \otimes \omega_S) \boxtimes 
\frac{h^+(z)-h^-(w)}{q_S-1}
\delta\left(\frac{w}{z}\right) \right). 
\end{align}
Here $\delta(x)=\sum_{k \in \mathbb{Z}}x^k$, 
$h^{\pm}(z)$ is the expansion of the following 
rational function at $z=\infty$ and $z=0$
\begin{align*}
h(z)=\left(1-\frac{1}{z} \right)\wedge^{\bullet} \left(\frac{(q_S^{-1}-1)\ffF(\beta)}{z}\right).
\end{align*}
Also $q_S$ is the pull-back of the 
class $[\omega_S] \in K(S)$. 
\end{thm}
The proof for Theorem~\ref{intro:thmC}
much relies on arguments of Negut~\cite{Negut}, 
where he studied Hecke operators for K-groups of 
moduli spaces of stable sheaves on surfaces,
and computed several commutator relations. 
Contrary to the case of~\cite{Negut}, our 
moduli stacks may be singular so we 
can prove the identity (\ref{intro:commu}) only for elements 
represented by a perfect complex.
Also it seems likely that, following the arguments of~\cite{Negut},
we can compute more commutator relations 
such as $[\mu_{\eE_1}^{+}(z), \mu_{\eE_2}^+(w)]$, 
$[\mu_{\eE_1}^{-}(z), \mu_{\eE_2}^+(-)]$.
We will not pursue these computations in this paper. 

As a corollary of Theorem~\ref{intro:thmC},
we have the following relation for $k=0$
(see Corollary~\ref{cor:com:k=0})
 \begin{align}\label{intro:relation}
[\mu_{\eE_1, 0}^{+}, \mu_{\eE_2, 0}^-](-) = 
(-)\otimes p_{\fM\ast}((\eE_1 \otimes \eE_2 \otimes \omega_S)
\boxtimes \ffF(\beta)^{\vee}).
\end{align}
The 
existence of operators $\mu_{\eE, 0}^{\pm}$ satisfying 
the 
relation (\ref{intro:relation})
may be regarded as 
a categorification of Weyl algebra action on homologies of Hilbert schemes of 
points $C^{[n]}$ on locally 
planar curves $C$ constructed by Rennemo~\cite{MR3807309} (see Remark~\ref{rmk:planar}), 
which is an 
analogy of Grojnowski and Nakajima's 
Heisenberg action 
for homologies of Hilbert schemes of points on 
surfaces~\cite{MR1386846, MR1441880}. 
The moduli space of stable pairs $P_n(X, \beta)$ is 
much more generalized version of $C^{[n]}$, so we expect 
a similar Weyl algebra 
action on the vanishing cycle cohomology of $P_n(X, \beta)$. 
The 
operators $\mu_{\eE}^{\pm}$ together with the 
relation (\ref{intro:relation}) gives a further categorification 
of such expected Weyl algebra action.  
Using the Weyl-algebra action, the GV form of the generating 
series of $\chi(C^{[n]})$ is derived in~\cite{MR3807309}, 
see~\cite[Section~1.4]{MR3807309} for details. 
We expect that the result of Theorem~\ref{intro:thmC} is relevant 
for the categorification of GV form of PT categories, 
which will be pursued in a future work. 

\subsection{Notation and convention}
In this paper, all the schemes or derived stacks are defined 
over $\mathbb{C}$. For a scheme or a derived stack $A$ and a 
quasi-coherent sheaf $\fF$ on it, we denote by $S(\fF)
=\mathrm{Sym}_{A}^{\bullet}(\fF)$
its symmetric product. 
For a derived stack $\fM$, we always denote by $t_0(\fM)$
its classical truncation. 
For a triangulated category $\dD$ and a set of objects
$\sS \subset \dD$, 
we denote by $\langle \sS \rangle_{\mathrm{ex}}$ the extension 
closure, i.e. the smallest extension closed subcategory 
which contains $\sS$.

\subsection{Acknowledgements}
The author is grateful to Francesco Sala 
and Mauro Porta
for asking a
question
 whether DT categories admit 
Hall-type algebra structures
when the previous paper~\cite{TocatDT} was posted on 
arXiv, and 
also several useful comments for the first 
draft of this paper. 
The author is also grateful to Andrei Negut for 
useful comments, and Tasuki Kinjo
for valuable discussions. 
The author is supported by World Premier International Research Center
Initiative (WPI initiative), MEXT, Japan, and Grant-in Aid for Scientific
Research grant (No.~19H01779) from MEXT, Japan.

\section{Singular supports of coherent sheaves and Fourier-Mukai transforms}
The notion of singular supports for (ind) coherent sheaves on 
quasi-smooth derived stacks was 
developed by Arinkin-Gaitsgory~\cite{MR3300415}, in order to 
formulate a categorical 
geometric Langlands conjecture. 
In the author's previous paper~\cite{TocatDT}, we used 
singular supports to capture (un)stable sheaves 
on 3-folds from the derived geometry on moduli 
stacks of sheaves the surface. 
In this section, we review the theory of singular supports
and see how they interact with Fourier-Mukai transforms. 
\subsection{Singular supports of coherent sheaves}
Let $A$ be an affine $\mathbb{C}$-scheme and 
$V \to A$ a vector  bundle on it. 
For a section $s \colon A \to V$, 
we consider the affine derived scheme $\mathfrak{U}$
given by the derived zero locus of $s$
\begin{align}\label{frak:U}
\mathfrak{U}=\Spec \rR(V \to A, s)
\end{align}
where $\rR(V \to A, s)$ is the Koszul complex
\begin{align*}
\rR(V\to A, s) \cneq \left( \cdots \to \bigwedge^2 V^{\vee} 
\stackrel{s}{\to} V^{\vee} \stackrel{s}{\to} \oO_A  \right). 
\end{align*}
The classical truncation of $\mathfrak{U}$
is the closed subscheme of $A$
given by  
\begin{align*}
\uU \cneq t_0(\mathfrak{U})=(s=0) \subset A.
\end{align*}

On the other hand, let $w \colon V^{\vee} \to \mathbb{C}$
be the function defined by 
\begin{align*}
w(x, v)=\langle s(x), v\rangle, \ 
x \in A, \ v \in V^{\vee}|_{x}.
\end{align*}
Then its critical locus is the 
classical truncation of $(-1)$-shifted cotangent 
scheme over $\mathfrak{U}$
(or called dual obstruction cone, see~\cite{MR3607000})
\begin{align*}
\mathrm{Crit}(w)=
t_0(\Omega_{\mathfrak{U}}[-1])=
\Spec S(\hH^1(\mathbb{T}_{\mathfrak{U}})). 
\end{align*}
Below we take the fiberwise weight two 
$\mathbb{C}^{\ast}$-action on
the total space of the bundle 
$V^{\vee} \to A$, so that $w$ is of weight two. 

Let $\mathrm{HH}^{\ast}(\mathfrak{U})$ be the 
Hochschild cohomology
\begin{align*}
\mathrm{HH}^{\ast}(\mathfrak{U})
 \cneq \Hom_{\mathfrak{U} \times \mathfrak{U}}^{\ast}
(\Delta_{\ast}\oO_{\mathfrak{U}}, \Delta_{\ast}\oO_{\mathfrak{U}}).
\end{align*}
Here $\Delta \colon \mathfrak{U} \to \mathfrak{U} \times \mathfrak{U}$ is the diagonal. 
Then it is shown in~\cite[Section~4]{MR3300415}
that there 
exists a
canonical map 
$\hH^1(\mathbb{T}_{\mathfrak{U}}) \to \mathrm{HH}^2(\mathfrak{U})$.
So for $F \in \Dbc(\mathfrak{U})$, we have 
the map of 
graded rings
\begin{align}\label{ssuport:map}
\oO_{\mathrm{Crit}(w)}=
S(\hH^1(\mathbb{T}_{\mathfrak{U}})) \to \mathrm{HH}^{2\ast}(\mathfrak{U}) 
\to 
\Hom^{2\ast}(F, F).  
\end{align}
Here the last arrow is defined 
by taking Fourier-Mukai transforms associated with 
morphisms $\Delta_{\ast}\oO_{\mathfrak{U}} \to \Delta_{\ast}\oO_{\mathfrak{U}}[2\ast]$. 
The above map 
 defines the 
$\mathbb{C}^{\ast}$-equivariant 
$\oO_{\mathrm{Crit}(w)}$-module 
structure on $\Hom^{2\ast}(\fF, \fF)$, 
which is finitely generated  by~\cite[Theorem~4.1.8]{MR3300415}. 
The \textit{singular support} of $F$
\begin{align}\label{def:sing:supp}
\Supp^{\rm{sg}}(F) \subset \mathrm{Crit}(w)
\end{align}
is defined to be the support of $\Hom^{2\ast}(F, F)$ as a graded
$\oO_{\mathrm{Crit}(w)}$-module. 
Note that the singular support is a conical 
(i.e. $\mathbb{C}^{\ast}$-invariant)
 closed subscheme of 
$\mathrm{Crit}(w)$. 
For a conical
 closed subset $Z \subset \mathrm{Crit}(w)$, we denote by 
\begin{align*}
\cC_Z \subset \Dbc(\mathfrak{U})
\end{align*}
the triangulated subcategory of objects $F \in \Dbc(\mathfrak{U})$ whose 
singular supports are contained in $Z$. 
Via Koszul duality, 
we have the following relation with 
the category of matrix factorizations
(see~\cite[Corollary~2.11]{TocatDT})
\begin{align}\label{equiv:MF}
\Dbc(\fU)/\cC_Z \stackrel{\sim}{\to}
\mathrm{MF}^{\mathbb{C}^{\ast}}(V^{\vee} \setminus Z, w). 
\end{align}
Here the right hand side is the 
triangulated category of $\mathbb{C}^{\ast}$-equivariant 
matrix factorizations for 
$w \colon V^{\vee}\setminus Z \to \mathbb{C}$. 
\subsection{Quasi-smooth derived stacks}
Below, we denote by $\mathfrak{M}$ a
derived Artin stack over $\mathbb{C}$.
This means that 
$\mathfrak{M}$ is a contravariant
$\infty$-functor from 
the $\infty$-category of 
affine derived schemes over $\mathbb{C}$ to 
the $\infty$-category of 
simplicial sets 
\begin{align*}
\mathfrak{M} \colon 
dAff^{op} \to SSets
\end{align*}
satisfying some conditions (see~\cite[Section~3.2]{MR3285853} for details). 
Here $dAff^{\rm{op}}$ is defined to be the
$\infty$-category of 
commutative simplicial $\mathbb{C}$-algebras, 
which is equivalent to the $\infty$-category of 
commutative differential graded 
$\mathbb{C}$-algebras with non-positive degrees. 
All derived stacks considered in this paper are locally of finite
presentation. 
The classical truncation of $\mathfrak{M}$ is denoted by 
\begin{align*}
\mM \cneq t_0(\mathfrak{M}) \colon 
Aff^{op} \hookrightarrow 
dAff^{op} \to SSets
\end{align*}
where the first arrow is a natural functor 
from the category of affine schemes
to affine derived schemes. 

Following~\cite{MR3285853}, we define the 
dg-category of quasi-coherent sheaves on $\mathfrak{M}$
as 
\begin{align}\label{limit:L}
L_{\rm{qcoh}}(\mathfrak{M}) \cneq
\lim_{\mathfrak{U} \to \mathfrak{M}} L_{\rm{qcoh}}(\mathfrak{U}).
\end{align}
Here $\fU=\Spec A$ is an 
affine derived scheme 
for a cdga $A$.
The category
$L_{\rm{qcoh}}(\mathfrak{U})$ 
is defined to be the 
dg-category of dg-modules over $A$ 
localized
by quasi-isomorphisms
(see~\cite[Section~2.4]{MR2762557}), 
so that its 
homotopy category is equivalent to 
the derived category $D_{\rm{qcoh}}(\mathfrak{U})$
of dg-modules over $A$. 
The limit in (\ref{limit:L})
is taken 
for all the diagrams
\begin{align}\label{dia:smooth}
\xymatrix{
\mathfrak{U} \ar[rd]_{\alpha} \ar[rr]^-{f} & & \mathfrak{U}' 
\ar[ld]^-{\alpha'} \\
& \mathfrak{M}, & 
}
\end{align}
where $f$ is a 0-representable smooth 
morphism, and 
$\alpha' \circ f$ is equivalent to $\alpha$. 
The homotopy category of $L_{\rm{qcoh}}(\mathfrak{M})$
is denoted by 
$D_{\rm{qcoh}}(\mathfrak{M})$. We have the 
triangulated subcategory
\begin{align*}
D^b_{\rm{coh}}(\mathfrak{M}) \subset D_{\rm{qcoh}}(\mathfrak{M})
\end{align*}
consisting of objects which have bounded coherent 
cohomologies. 
We note that there 
is a bounded t-structure on $D_{\rm{coh}}^b(\mathfrak{M})$
whose heart coincides with 
$\Coh(\mM)$. 

A morphism of derived stacks $f \colon \fM \to \fN$ is 
called \textit{quasi-smooth} if 
$\mathbb{L}_f$ is perfect 
such that for any point $x \to \mM$
the restriction $\mathbb{L}_f|_{x}$ is 
of cohomological 
amplitude $[-1, 1]$.
Here 
$\mathbb{L}_f$ is the $f$-relative cotangent complex. 
A derived stack 
$\mathfrak{M}$ over $\mathbb{C}$
is called \textit{quasi-smooth}
if $\fM \to \Spec \mathbb{C}$ is quasi-smooth. 
By~\cite[Theorem~2.8]{MR3352237}, 
the quasi-smoothness of $\mathfrak{M}$ is equivalent to 
that $\mM$ 
is a 1-stack, 
and 
any point of $\mathfrak{M}$ lies 
in the image of a $0$-representable 
smooth morphism 
\begin{align}\label{map:alpha}
\alpha \colon \mathfrak{U} \to \mathfrak{M}
\end{align}
where $\mathfrak{U}$ is an affine derived scheme 
of the form (\ref{frak:U}). 
Furthermore following~\cite[Definition~1.1.8]{MR3037900}, 
a derived stack $\mathfrak{M}$ is called 
\textit{QCA (quasi-compact and with affine automorphism groups)}
if the following conditions hold:
\begin{enumerate}
\item $\mathfrak{M}$ is quasi-compact;
\item The automorphism groups of its geometric points are affine;
\item The classical inertia stack $I_{\mM} \cneq \Delta
 \times_{\mM \times \mM} \Delta$
is of finite presentation over $\mM$. 
\end{enumerate}
Below our derived stack $\fM$ always satisfies (ii), (iii), 
but we often encount derived stacks which are not of finite type 
and in this case (i) may not be satisfied. 
In such a case we may take a derived open substack 
$\fM^{\fin} \subset \fM$ of finite type, 
and then $\fM^{\fin}$ is QCA. 

For a quasi-smooth derived stack 
$\fM$ and an object 
$\fE \in \mathrm{Perf}(\fM)$, we 
set
\begin{align*}
\rho \colon \mathbb{V}(\fE) \cneq 
\Spec_{\fM}S(\fE) \to \fM. 
\end{align*}
If $\fE|_{x}$ is of cohomological amplitude $[-1, 1]$
for any point $x \to \mM$, 
then $\mathbb{V}(\fE)$ is quasi-smooth and $\rho$ is a 
quasi-smooth morphism. 
We also set
\begin{align*}
\rho' \colon 
\mathbb{P}(\fE) \cneq 
(\mathbb{V}(\fE) \setminus 0_{\fM})/\mathbb{C}^{\ast}
\to \fM.
\end{align*}
Here $\mathbb{C}^{\ast}$ acts by weight $m$ on 
$S^m(\fE)$, and $0_{\fM}$ is the zero section of 
$\rho$. 
Under the above situation, 
$\mathbb{P}(\fE)$ is quasi-smooth and 
$\rho'$ is a quasi-smooth morphism. 
If furthermore 
$\fE|_{x}$ is of cohomological amplitude $[-1, 0]$
for any point $x \to \mM$, 
then $\rho'$ is a proper morphism, i.e. 
$\rho'$ is a representable morphism such that any 
pull-back by $U \to \fM$ for a scheme $U$ is 
a proper morphism of schemes.

\subsection{Singular supports for quasi-smooth derived stacks}
Let $\Omega_{\fM}[-1]$ be the $(-1)$-shifted cotangent stack over $\fM$, i.e. 
\begin{align*}
\Omega_{\fM}[-1]=
\mathbb{V}(\mathbb{T}_{\fM}[1]). 
\end{align*}
Let $\fM_1$, $\fM_2$ be
quasi-smooth derived stacks 
with truncations $\mM_i=t_0(\fM_i)$. 
Let $f \colon \fM_1 \to \fM_2$ be a morphism. 
Then 
the morphism 
$f^{\ast}\mathbb{L}_{\fM_2} \to \mathbb{L}_{\fM_1}$
induces the diagram
\begin{align}\label{diagram:induced}
\xymatrix{
t_0(\Omega_{\fM_1}[-1]) \ar[d] & f^{\ast}t_0(\Omega_{\fM_2}[-1]) \ar[d] \ar[l]_-{f^{\diamondsuit}} 
\ar[r]^-{f^{\spadesuit}} 
\ar@{}[rd]|\square
& 
t_0(\Omega_{\fM_2}[-1]) \ar[d] \\
\mM_1 & \mM_1 \ar@{=}[l] \ar[r]_-{f} & \mM_2. 
}
\end{align}
It is easy to see that $f$ is quasi-smooth if and only
if $f^{\diamondsuit}$ is a closed immersion, 
$f$ is smooth if and only if $f^{\diamondsuit}$ is an isomorphism. 
Let $\Omega_f[-2]$ is the
 $(-2)$-shifted conormal stack
\begin{align*}
\Omega_{f}[-2] \cneq 
\mathbb{V}(\mathbb{T}_f[2]).
\end{align*}
Here $\mathbb{T}_f$ is the dual of $\mathbb{L}_f$. 
We have the following lemma. 
\begin{lem}\label{lem:conormal}
We have the isomorphism over $\mM_1$
\begin{align}\label{conormal:shift}
t_0(\Omega_{f}[-2]) \stackrel{\cong}{\to}
(f^{\diamondsuit})^{-1}(0_{\mM_1}).
\end{align}
Here $0_{\mM_1}$ is the zero section of the 
left horizontal arrow in (\ref{diagram:induced}). 
\end{lem}
\begin{proof}
By the definition of $\Omega_f[-2]$, we have
\begin{align*}
t_0(\Omega_f[-2])=\Spec_{\mM_1}S(\hH^2(\mathbb{T}_f)). 
\end{align*}
Therefore for each $p \in \mM_1$, the fiber of 
$t_0(\Omega_f[-2]) \to \mM_1$ at $p$ is 
$\hH^{-2}(\mathbb{L}_f|_{p})$. 
On the other hand, we 
have the distinguished triangle 
\begin{align}\label{distinct:L}
f^{\ast}\mathbb{L}_{\fM_2} \to \mathbb{L}_{\fM_1} \to 
\mathbb{L}_{f}. 
\end{align}
By restricting it to $p$ and the associated long 
exact sequence of cohomologies, 
we have the exact sequence
\begin{align*}
0 \to \hH^{-2}(\mathbb{L}_f|_{p}) \to \hH^{-1}(\mathbb{L}_{\fM_2}|_{f(p)}) 
\to \hH^{-1}(\mathbb{L}_{\fM_1}|_{p}). 
\end{align*}
Therefore the fibers of both sides of (\ref{conormal:shift}) are naturally 
identified. 
\end{proof}

The smooth morphism (\ref{map:alpha})
induces the diagram
\begin{align}\label{map:t0omega}
\xymatrix{
t_0(\Omega_{\mathfrak{U}}[-1]) &
\ar[l]_-{\alpha^{\diamondsuit}}^-{\cong} 
t_0(\alpha^{\ast}\Omega_{\mathfrak{M}}[-1])
\ar[r]^-{\alpha^{\spadesuit}} &
t_0(\Omega_{\mathfrak{M}}[-1]).
}
\end{align}
A closed substack 
of $t_0(\Omega_{\fM}[-1])$
is called \textit{conical} if it is closed under the
fiberwise $\mathbb{C}^{\ast}$-action on 
$t_0(\Omega_{\fM}[-1])$. 
For a conical closed substack 
$\zZ \subset t_0(\Omega_{\fM}[-1])$, 
we have the conical closed subscheme
\begin{align*}
\alpha^{\ast}\zZ \cneq 
\alpha^{\diamondsuit}(\alpha^{\spadesuit})^{-1}(\zZ) \subset 
t_0(\Omega_{\mathfrak{U}}[-1])=\Crit(w). 
\end{align*}
We define 
\begin{align*}
\cC_{\zZ} \subset \Dbc(\fM)
\end{align*}
to be
the triangulated subcategory consisting of objects 
whose singular supports are contained in $\zZ$,
i.e. those of objects $E \in \Dbc(\fM)$ such that 
for any map $\alpha$ as in (\ref{map:alpha}), we have 
\begin{align*}
\Supp^{\rm{sg}}(\alpha^{\ast}E) \subset \alpha^{\ast}\zZ. 
\end{align*}
 
Below we sometimes take a derived open substack 
$\fM^{\fin} \subset \fM$ and work on 
$\fM^{\fin}$. In this case, 
for a conical closed substack $\zZ \subset t_0(\Omega_{\fM}[-1])$, 
we set
\begin{align*}
\zZ^{\fin} \cneq \zZ \times_{\mM} \mM^{\fin} 
\subset t_0(\Omega_{\fM^{\fin}}[-1])
\end{align*}
where $\mM^{\fin}$ is the classical truncation of $\fM^{\fin}$. 
Then we have the subcategory 
$\cC_{\zZ^{\fin}} \subset \Dbc(\fM^{\fin})$, and the 
 quotient category
\begin{align*}
\Dbc(\fM^{\fin})/\cC_{\zZ^{\fin}}
\end{align*}
is our model for the definition of DT category in~\cite{TocatDT}. 
By the equivalence (\ref{equiv:MF}), the above quotient category 
may be regarded as a gluing of matrix factorizations. 
\subsection{Functoriality of singular supports}
Let $\fM_1$, $\fM_2$ be quasi-smooth 
derived stacks, and take a morphism
\begin{align*}
f \colon \fM_1 \to \fM_2
\end{align*}
First suppose that 
$f$ is a  quasi-smooth morphism. 
Then we have the pull-back functor
(see~\cite[Section~4.2]{PoSa})
\begin{align}\label{f:back}
f^{\ast} \colon \Dbc(\fM_2) \to \Dbc(\fM_1). 
\end{align} 
In this case, the morphism
$f^{\diamondsuit}$ in the diagram (\ref{diagram:induced}) 
is a closed immersion. 
Therefore for any 
conical closed substack $\zZ_2 \subset t_0(\Omega_{\fM_2}[-1])$, 
we have the conical closed substack
\begin{align}\label{f:diamond}
f^{\diamondsuit}(f^{\spadesuit})^{-1}(\zZ_2) \subset
t_0(\Omega_{\fM_1}[-1]). 
\end{align}

\begin{lem}\label{lem:back}
Suppose that 
a conical closed substack 
$\zZ_1 \subset t_0(\Omega_{\fM_1}[-1])$
contains (\ref{f:diamond}). 
Then the functor (\ref{f:back}) sends $\cC_{\zZ_2}$ to $\cC_{\zZ_1}$. 
\end{lem}
\begin{proof}
The $f^{!}$-version is proved in~\cite[Lemma~8.4.2]{MR3300415}, i.e. 
the functor $f^{!}$ sends $\cC_{\zZ_2}$ to $\cC_{\zZ_1}$. 
As $f$ is quasi-smooth, we have 
$f^{!}(-)=f^{\ast}(-) \otimes \omega_{f}[\vdim f]$ where $\omega_f$ is a 
$f$-relative canonical line bundle of $f$ and $\vdim f$ is the virtual dimension of $f$ (see~\cite[(3.12)]{MR3037900}). 
Since $\otimes \omega_f[\vdim]$ 
sends $\cC_{\zZ_1}$ to $\cC_{\zZ_1}$, we obtain the lemma. 
\end{proof}

Next suppose that $f$ is a proper morphism.
Then we have the push-forward functor (see~\cite[Section~4.2]{PoSa})
\begin{align}\label{f:push}
f_{\ast} \colon \Dbc(\fM_1) \to \Dbc(\fM_2). 
\end{align}
In this case, the morphism 
$f^{\spadesuit}$ in the diagram (\ref{diagram:induced})
is a proper morphism of schemes. 
Therefore for any conical closed
substack $\zZ_1 \subset t_0(\Omega_{\fM_1}[-1])$, 
we have the 
conical closed substack 
\begin{align}\label{f:spade}
f^{\spadesuit} (f^{\diamondsuit})^{-1}(\zZ_1) \subset 
t_0(\Omega_{\fM_2}[-1]). 
\end{align}
The following lemma is proved in~\cite[Lemma~8.4.5]{MR3300415}. 
\begin{lem}\label{lem:fpush}
Suppose that 
a conical closed substack 
$\zZ_2 \subset t_0(\Omega_{\fM_2}[-1])$
contains (\ref{f:spade}). 
Then the functor (\ref{f:push}) sends 
$\cC_{\zZ_1}$ to $\cC_{\zZ_2}$. 
\end{lem}

Let $\fN$ be another quasi-smooth derived stack 
with 
a diagram
\begin{align}\label{diagram:N}
\fM_1 \stackrel{f_1}{\leftarrow} \fN \stackrel{f_2}{\to} \fM_2. 
\end{align}
Suppose that $f_1$ is quasi-smooth and $f_2$ is a 
proper morphism. 
Then for any $\pP \in \mathrm{Perf}(\fN)$, 
we have the functor
\begin{align}\label{FM:P}
\FM_{\pP}(-) =f_{2\ast}(f_1^{\ast}(-) \otimes \pP) \colon 
\Dbc(\fM_1) \to \Dbc(\fM_2). 
\end{align}
Let $f$ be the morphism
\begin{align*}
f=(f_1, f_2) \colon \fN \to \fM_1 \times \fM_2. 
\end{align*}
Then we have the diagram
\begin{align}\label{diagram:fN}
\xymatrix{
& & t_0(\Omega_f[-2])
\ar@{}[dd]|\square
 \ar[ld]_-{g_1} \ar[rd]^-{g_2} \ar@/_30pt/[lldd]_-{h_1} 
\ar@/^30pt/[rrdd]^-{h_2} & & \\
& f_1^{\ast}t_0(\Omega_{\fM_1}[-1]) \ar[ld]_-{f^{\spadesuit}_1} 
\ar[rd]^-{f^{\diamondsuit}_1} & & f_2^{\ast}t_0(\Omega_{\fM_2}[-1]) \ar[ld]_-{f^{\diamondsuit}_2}
 \ar[rd]^-{f^{\spadesuit}_2} & \\
t_0(\Omega_{\fM_1}[-1]) & & t_0(\Omega_{\fN}[-1]). & & t_0(\Omega_{\fM_2}[-1]).
}
\end{align}
Here the middle square is Cartesian by the isomorphism 
(\ref{conormal:shift}). 
Note that $g_2$ is a closed immersion as $f^{\diamondsuit}_1$ is,
and $f^{\spadesuit}_2$ is proper as $f_2$ is. 
Therefore $h_2$ is also proper. 
We obtain the following commutative diagram
\begin{align}\label{dia:evf}
\xymatrix{
t_0(\Omega_{\fM_1}[-1]) \ar[d] & t_0(\Omega_{f}[-2]) 
\ar[r]^-{h_2} \ar[l]_-{h_1} \ar[d] & 
t_0(\Omega_{\fM_2}[-1]) \ar[d] \\
\mM_1 & \nN \ar[r]^-{f_2} \ar[l]_-{f_1} & \mM_2. 
}
\end{align}
Here $\mM_i=t_0(\fM_i)$ and $\nN=t_0(\fN)$. 
For 
a conical closed substack $\zZ_1 \subset t_0(\Omega_{\fM_1}[-1])$, we 
have the following conical closed substack
\begin{align}\label{Z:h}
h_2 (h_1)^{-1}(\zZ_1) \subset t_0(\Omega_{\fM_2}[-1]). 
\end{align}

\begin{prop}\label{prop:FM}
Suppose that 
a conical closed substack 
$\zZ_2 \subset t_0(\Omega_{\fM_2}[-1])$
contains (\ref{Z:h}). 
Then the functor (\ref{FM:P})
sends $\cC_{\zZ_1}$ to $\cC_{\zZ_2}$. 
\end{prop}
\begin{proof}
Note that since $\pP$ is perfect, 
the functor $\otimes \pP$ 
takes $\cC_{\wW}$ to $\cC_{\wW}$
for any conical closed substack 
$\wW \subset t_0(\Omega_{\fN}[-1])$. 
Therefore by Lemma~\ref{lem:back} and Lemma~\ref{lem:fpush}, 
the functor (\ref{FM:P}) sends $\cC_{\zZ_1}$ to $\cC_{\zZ_2'}$
where 
$\zZ_2'$ is 
\begin{align*}
\zZ_2'=f^{\spadesuit}_2(f^{\diamondsuit}_2)^{-1} f^{\diamondsuit}_1 
(f^{\spadesuit}_1)^{-1}(\zZ_1) 
=f^{\spadesuit}_2 g_2 (g_1)^{-1} (f^{\spadesuit}_1)^{-1}(\zZ_1)=h_2(h_1)^{-1}(\zZ_1). 
\end{align*}
By the assumption we have $\zZ_2' \subset \zZ_2$, 
hence $\cC_{\zZ_2'} \subset \cC_{\zZ_2}$. Therefore the proposition holds. 
\end{proof}

In the situation above, let 
$\fM_i^{\circ} \subset \fM_i$ for $i=1, 2$ be 
derived open substacks satisfying 
\begin{align*}
\fN^{\circ} \cneq f_2^{-1}(\fM_2^{\circ}) \subset f_1^{-1}(\fM_1^{\circ}). 
\end{align*}
Then a diagram (\ref{diagram:N})
restricts to the diagram
\begin{align}\label{diagram:N2}
\fM_1^{\circ} \stackrel{f_1^{\circ}}{\leftarrow}
 \fN^{\circ} \stackrel{f_2^{\circ}}{\to} \fM_2^{\circ}. 
\end{align}
Note that $f^{\circ}_1$ is quasi-smooth and 
$f^{\circ}_2$ is proper. 
Therefore for $\pP^{\circ} \cneq \pP|_{\fN^{\circ}}$, we
have the functor
\begin{align*}
\mathrm{FM}_{\pP^{\circ}}(-) 
=f_{2\ast}^{\circ}(f_1^{\circ \ast}(-) \otimes \pP^{\circ})
 \colon \Dbc(\fM_1^{\circ}) \to \Dbc(\fM_2^{\circ}). 
\end{align*}
For conical closed substacks 
$\zZ_i \subset t_0(\Omega_{\fM_i}[-1])$, 
we set 
$\zZ_i^{\circ} \cneq \zZ_i \times_{\mM_i} \mM_i^{\circ}$
where $\mM_i^{\circ}=t_0(\fM_i^{\circ})$. 
\begin{lem}\label{lem:restrict}
Under the same assumption of Proposition~\ref{prop:FM}, 
the functor $\mathrm{FM}_{\pP^{\circ}}$ sends 
$\cC_{\zZ_1^{\circ}}$ to $\cC_{\zZ_2^{\circ}}$. 
\end{lem}
\begin{proof}
Similarly to (\ref{dia:evf}), the diagram (\ref{diagram:N2})
 induces the diagram 
\begin{align}\label{dia:evf2}
\xymatrix{
t_0(\Omega_{\fM_1^{\circ}}[-1]) \ar[d] & t_0(\Omega_{f^{\circ}}[-2]) 
\ar[r]^-{h_2^{\circ}} \ar[l]_-{h_1^{\circ}} \ar[d] & 
t_0(\Omega_{\fM_2^{\circ}}[-1]) \ar[d] \\
\mM_1^{\circ} & \nN^{\circ} \ar[r]^-{f_2^{\circ}}
 \ar[l]_-{f_1^{\circ}} & \mM_2^{\circ}. 
}
\end{align}
Here 
$\nN^{\circ}=t_0(\fN^{\circ})$ and 
$f^{\circ}=(f_1^{\circ}, f_2^{\circ}) \colon 
\nN^{\circ} \to \fM_1^{\circ} \times \fM_2^{\circ}$. 
As we have the Cartesian square
\begin{align*}
\xymatrix{
\fN^{\circ} 
\ar[r]^-{f^{\circ}} \ar[d] \ar@{}[rd]|\square
& \fM_1^{\circ} \times \fM_2^{\circ} \ar[d] \\
\fN \ar[r]^-{f} & \fM_1 \times \fM_2. 
}
\end{align*}
we have 
$t_0(\Omega_{f^{\circ}}[-2])=t_0(\Omega_f[-2]) \times_{\nN} \nN^{\circ}$.
It follows that the right square of (\ref{dia:evf2})
is obtained from the right square of (\ref{dia:evf})
via $(-) \times_{\mM_2}\mM_2^{\circ}$.
Therefore we have
\begin{align*}
h_2^{\circ}(h_1^{\circ})^{-1}(\zZ_1^{\circ}) 
=h_2^{\circ}(h_1^{-1}(\zZ_1) \times_{\nN} \nN^{\circ})
=h_2(h_1^{-1}(\zZ_1) \times_{\nN} \nN^{\circ})\times_{\mM_2}\mM_2^{\circ}
\subset \zZ_2 \times_{\mM_2} \mM_2^{\circ}=\zZ_2^{\circ}. 
\end{align*}
The lemma now follows from Proposition~\ref{prop:FM}. 
\end{proof}

\section{Hall-type algebras for DT categories}
In this section, 
we prove Theorem~\ref{intro:thmA}.
We first recall the relevant 
constructions of derived moduli stacks of 
sheaves and their extensions, 
and then recall the definition of  
DT categories introduced in~\cite{TocatDT}.
They are defined in terms of
derived categories of coherent sheaves on derived 
moduli spaces of sheaves on surfaces, together with 
the notion of singular supports. 
We then show that the categorified COHA structure 
by Porta-Sala~\cite{PoSa}
for derived moduli spaces of sheaves on surfaces
is compatible with singular supports, 
so that it descends to the algebra structure on 
DT categories. 

\subsection{Derived moduli stacks of coherent sheaves on surfaces}
Let $S$ be a smooth projective surface over $\mathbb{C}$. 
We consider the 
derived Artin stack
\begin{align}\label{dstack:MS}
\mathfrak{M}_S \colon dAff^{op} \to SSets
\end{align}
which sends an affine derived scheme 
$T$ to the $\infty$-groupoid of 
perfect complexes on $T \times S$, 
whose restriction to $t_0(T) \times S$ is a 
flat family of coherent sheaves on $S$ over $t_0(T)$. 
Note that we have an open immersion
\begin{align*}
\mathfrak{M}_S \subset \mathfrak{Perf}_S
\end{align*}
where $\mathfrak{Perf}_S$ is
the derived moduli stack of perfect complexes on $S$ 
constructed in~\cite{MR2493386}. 
Since
any object in $\Coh(S)$ is perfect as $S$ is smooth,
the derived Artin stack $\mathfrak{M}_S$ is 
the derived  
moduli stack of objects in 
$\Coh(S)$. 
The classical truncation of 
$\mathfrak{M}_S$ is denoted by 
$\mM_S=t_0(\mathfrak{M}_S)$. 

Let $N(S)$ be the numerical Grothendieck group of 
$S$
\begin{align}\label{def:N(S)}
N(S) \cneq K(S)/\equiv
\end{align}
where $F_1, F_2 \in K(S)$
satisfies $F_1 \equiv F_2$ if $\ch(F_1)=\ch(F_2)$. 
Then $N(S)$ is a finitely generated free abelian 
group. 
We have the decompositions into open and closed substacks
\begin{align*}
\mathfrak{M}_S=\coprod_{v \in N(S)} \mathfrak{M}_S(v), \ 
\mM_S=\coprod_{v \in N(S)} \mM_S(v)
\end{align*}
where each component 
corresponds to sheaves $F$ on $S$
with $[F]=v$. 
We denote by 
\begin{align}\label{F:universal}
\mathfrak{F}(v) \in \mathrm{Perf}(S \times \mathfrak{M}_S(v)). 
\end{align}
the universal object on the component $\fM_S(v)$. 

We also define the derived moduli stack 
of exact sequences of coherent sheaves on $S$, 
following~\cite[Section~3]{PoSa}.
It is given by the derived Artin stack 
\begin{align}\notag
\mathfrak{M}_S^{\rm{ext}} \colon dAff^{op} \to SSets
\end{align}
which sends an affine derived scheme 
$T$ to the $\infty$-groupoid of fiber 
sequences of 
perfect complexes on $T \times S$, 
\begin{align}\label{seq:F}
\ffF_1 \to \ffF_2 \to \ffF_3
\end{align}
whose restrictions to $t_0(T) \times S$ are 
flat families of exact sequences of 
coherent sheaves on $S$ over $t_0(T)$. 
The classical truncation of $\fM_S^{\ext}$ is denoted by 
$\mM_S^{\ext} \cneq t_0(\fM_S^{\ext})$. 
We have the decompositions into open and closed substacks
\begin{align*}
\fM_S^{\ext}=\coprod_{v_{\bullet}=(v_1, v_2, v_3)} 
\fM_S^{\ext}(v_{\bullet}), \ 
\mM_S^{\ext}=\coprod_{v_{\bullet}=(v_1, v_2, v_3)} \mM_S^{\ext}(v_{\bullet})
\end{align*}
where each component corresponds to 
exact sequences $0\to F_1 \to F_2 \to F_3 \to 0$
with $[F_i]=v_i$. 

By sending a sequence (\ref{seq:F}) to $\ffF_i$, 
we have the evaluation morphisms
\begin{align*}
\mathrm{ev}_i \colon \mathfrak{M}_S^{\rm{ext}}(v_{\bullet})
\to \mathfrak{M}_S(v_i), \ 1\le i\le 3. 
\end{align*}
Below we use the following diagram
\begin{align}\label{dia:extS}
\xymatrix{
\fM_S^{\ext}(v_{\bullet}) \ar[r]^-{\ev_2} \ar[d]_-{(\ev_1, \ev_3)} & \fM_S(v_2)
 \\
\fM_S(v_1) \times \fM_S(v_3). & 
}
\end{align}
Then the vertical map in (\ref{dia:extS})
is described as the relative spectrum (see~\cite[Proposition~3.8]{PoSa})
\begin{align}\label{MSext:spec}
\mathfrak{M}_S^{\ext}(v_{\bullet})
=\mathbb{V}(\hH om_{p_{\mathfrak{M} \times \mathfrak{M}}}
(\mathfrak{F}(v_3), \mathfrak{F}(v_1)[1])^{\vee}).
\end{align}
Here we have denoted by $\ffF(v_i)$
the pull-back of the universal object (\ref{F:universal})
by the projection
\begin{align*}
S \times \fM_S(v_1) \times \fM_S(v_3) \to S \times \fM_S(v_i)
\end{align*}
and $p_{\fM \times \fM}$ is the projection to 
$\fM_S(v_1) \times \fM_S(v_3)$. 
Since for $(F_1, F_3) \in \mM_S(v_1) \times \mM_S(v_3)$
the cohomological amplitude of 
$\RHom(F_3, F_1)[1]$
is $[-1, 1]$, the morphism $(\ev_1, \ev_3)$ is quasi-smooth. 
In particular, we have the well-defined functor
\begin{align*}
(\ev_1, \ev_3)^{\ast} \colon 
\Dbc(\mathfrak{M}_S(v_1) \times \mathfrak{M}_S(v_3)) \to 
\Dbc(\mathfrak{M}_S^{\ext}(v_{\bullet})). 
\end{align*}

On the other hand, the 
horizontal morphism in (\ref{dia:extS})
is a proper morphism. 
Indeed for a point $x \to \mathfrak{M}_S(v_2)$
corresponding to coherent sheaf 
$F_2$ on $S$, the 
fiber product 
\begin{align*}
\mathfrak{M}_S^{\ext}(v_{\bullet})
 \times_{\ev_2, \mathfrak{M}_S(v_2)} x
\end{align*} 
is the derived Quot scheme 
parameterizing 
quotients $F_2 \twoheadrightarrow F_3$
with $[F_3]=v_3$. 
Therefore we have the well-defined functor
\begin{align*}
\ev_{2\ast} \colon \Dbc(\mathfrak{M}_S^{\ext}(v_{\bullet})) \to 
\Dbc(\mathfrak{M}_S(v_2)). 
\end{align*}

As we mentioned in the introduction, 
the following composition 
\begin{align}\label{cat:COHA}
\Dbc(\mathfrak{M}_S(v_1)) \times 
\Dbc(\mathfrak{M}_S(v_3)) 
\stackrel{\boxtimes}{\to}
\Dbc(\mathfrak{M}_S(v_1) \times \mathfrak{M}_S(v_3))
\stackrel{\ev_{2\ast}(\ev_1, \ev_3)^{\ast}}{\longrightarrow}
\Dbc(\mathfrak{M}_S(v_2))
\end{align}
was considered in~\cite{PoSa}, as a 
categorification of 
cohomological Hall algebra on surfaces~\cite{KaVa2}. 

\subsection{Moduli stacks of coherent sheaves on local surfaces}
We next consider similar (but underived)
moduli stacks on the 3-fold $X$, defined by
\begin{align*}
X \cneq \mathrm{Tot}_S(\omega_S) \stackrel{\pi}{\to} S. 
\end{align*}
Here $\pi$ is the projection. 
Note that $X$ is a non-compact CY 3-fold, 
called \textit{local surface}. 
We denote by 
$\Coh_{\rm{cpt}}(X) \subset \Coh(X)$
the subcategory of compactly supported coherent 
sheaves on $X$. 
We consider the 
classical Artin stack
\begin{align*}
\mM_X : 
Aff^{\rm{op}} \to Groupoid
\end{align*}
whose $T$-valued points for $T \in Aff$
form the groupoid of $T$-flat families of 
objects in 
$\Coh_{\rm{cpt}}(X)$. 
We have the decomposition
into open and closed substacks
\begin{align*}
\mM_X=\coprod_{v \in N(S)} \mM_X(v)
\end{align*}
where each component corresponds to 
objects $E \in \Coh_{\rm{cpt}}(X)$ with 
$[\pi_{\ast}E]=v$. 

We have the natural push-forward morphism
\begin{align*}
\pi_{\ast} \colon \mM_X(v) \to \mM_S(v), \ 
E \mapsto \pi_{\ast}E
\end{align*}
which realizes $\mM_X(v)$ as 
the dual obstruction cone over $\mM_S(v)$, i.e. 
$\mM_X(v)$ is
the classical truncation of the 
$(-1)$-shifted cotangent stack
of $\fM_S(v)$ (see~\cite[Lemma~5.1]{TocatDT})
\begin{align}\label{shifted:cot}
\mM_X(v)=t_0(\Omega_{\mathfrak{M}_S(v)}[-1]). 
\end{align}

Similarly we consider the 
classical Artin stack
of short exact sequences of 
compactly supported coherent sheaves on $X$.
It is given by the 2-functor
\begin{align*}
\mM_X ^{\ext} \colon  
Aff^{\rm{op}} \to Groupoid
\end{align*}
whose $T$-valued points for $T \in Aff$
form the groupoid of 
exact sequences of coherent sheaves on $X \times T$, 
\begin{align*}
0 \to \eE_1 \to \eE_3 \to \eE_3 \to 0, \ 
\eE_i \in \mM_X(T).
\end{align*}
We have the decomposition into open and closed substacks
\begin{align*}
\mM_X^{\ext}=\coprod_{v_{\bullet}=(v_1, v_2, v_3)}
\mM_X^{\ext}(v_{\bullet})
\end{align*}
where each component corresponds to 
exact sequences $0 \to E_1 \to E_2 \to E_3 \to 0$
with $[\pi_{\ast}E_i]=v_i$. 
We have the evaluation morphisms
\begin{align*}
\ev_{i}^X \colon \mM_X^{\ext}(v_{\bullet}) \to \mM_X(v_i), \ 
E_{\bullet} \mapsto E_i
\end{align*}
and obtain the diagram
\begin{align}\label{dia:extX}
\xymatrix{
\mM_X^{\ext}(v_{\bullet})
 \ar[r]^-{\ev_2^X} \ar[d]_-{(\ev_1^X, \ev_3^X)} & \mM_X(v_2) \\
\mM_X(v_1) \times \mM_X(v_3). &
}
\end{align}
We also have the morphism given by the push-forward along 
$\pi$
\begin{align}\label{def:push:XS}
\pi_{\ast} \colon \mM_X^{\ext}(v_{\bullet})
 \to \mM_S^{\ext}(v_{\bullet}), \ 
E_{\bullet} \mapsto \pi_{\ast}E_{\bullet}
\end{align}
Here note that $\pi_{\ast}$ preserves the 
exact sequences of coherent sheaves 
as $\pi$ is affine. 

\subsection{The relation of $\mM_X^{\ext}$ and $\mM_S^{\ext}$}
Here we relate $\mM_X^{\ext}$ and 
$\fM_S^{\ext}$ as a $(-2)$-shifted 
conormal stack. 
Let $\ev$ be the morphism
\begin{align*}
\ev=(\ev_1, \ev_2, \ev_3)
 \colon \fM_S^{\ext}(v_{\bullet}) \to 
\fM_S(v_1) \times \fM_S(v_2) \times \fM_S(v_3). 
\end{align*}
We have the $(-2)$-shifted conormal stack 
and its truncation
\begin{align*}
\Omega_{\ev}[-2] \to \fM_S^{\ext}(v_{\bullet}), \ 
t_0(\Omega_{\ev}[-2]) \to \mM_S^{\ext}(v_{\bullet}). 
\end{align*}
\begin{prop}\label{prop:isomM}
There is an isomorphism over $\mM_S^{\ext}(v_{\bullet})$
\begin{align}\label{isom:ext}
\mM_X^{\ext}(v_{\bullet}) \stackrel{\cong}{\to} t_0(\Omega_{\ev}[-2]). 
\end{align}
\end{prop}
\begin{proof}
Let us take a point 
$p \in \mM_S^{\ext}(v_{\bullet})$ represented by 
an exact sequence 
$0 \to F_1 \to F_2 \to F_3 \to 0$
on $S$. The fiber of the morphism (\ref{def:push:XS})
at this point is identified with the stack of 
morphisms $\phi_i \colon F_i \to F_i \otimes \omega_S$
for $i=1, 2, 3$ which fit into the 
commutative diagram
\begin{align}\label{diagram:F}
\xymatrix{
0 \ar[r] & F_1 \ar[r]^-{i} \ar[d]_-{\phi_1} & 
F_2 \ar[r]^-{j} \ar[d]_-{\phi_2} & F_3 \ar[r] \ar[d]_-{\phi_3} & 0 \\
0 \ar[r] & F_1 \otimes \omega_S \ar[r]^-{i}  & 
F_2 \otimes \omega_S \ar[r]^-{j}  & F_3 \otimes \omega_S \ar[r]  & 0
}
\end{align}
Indeed such $\phi_i$ determines an $\oO_X$-module structure on 
$F_i$, and 
the diagram (\ref{diagram:F})
is an exact sequence of $\oO_X$-modules. 
From the diagram (\ref{diagram:F}), 
$\phi_1$, $\phi_3$ are uniquely determined by $\phi_2$. 
Conversely given $\phi_2$, 
it induces $\phi_1$, $\phi_3$ in the diagram (\ref{diagram:F})
if and only if the 
composition 
\begin{align*}
F_1 \stackrel{i}{\to} F_2 \stackrel{\phi_2}{\to}
F_2 \otimes \omega_S \stackrel{j}{\to} F_3 \otimes \omega_S
\end{align*}
is zero. 
Therefore the fiber of the morphism (\ref{def:push:XS})
is identified with 
\begin{align}\label{fiber:p}
\pi_{\ast}^{-1}(p)
=\Ker(\Hom(F_2, F_2 \otimes \omega_S) 
\stackrel{j \circ (-) \circ i}{\longrightarrow} 
\Hom(F_1, F_3 \otimes \omega_S)).
\end{align}

On the other hand, we have the commutative diagram
\begin{align*}
\xymatrix{
(\ev_1, \ev_3)^{\ast}\mathbb{L}_{\fM_S(v_1) \times \fM_S(v_3)}
 \ar@{=}[r] \ar[d] & 
(\ev_1, \ev_3)^{\ast}\mathbb{L}_{\fM_S(v_1) \times \fM_S(v_3)} \ar[d] &  \\
\ev^{\ast}\mathbb{L}_{\fM_S(v_1) \times \fM_S(v_2) \times \fM_S(v_3)} 
\ar[r] \ar[d] & 
\mathbb{L}_{\fM_S^{\ext}(v_{\bullet})} \ar[r] \ar[d] & \mathbb{L}_{\ev} \\
\ev_2^{\ast}\mathbb{L}_{\fM_S(v_2)} & (\ev_1, \ev_3)^{\ast}
(\hH om_{p_{\fM \times \fM}}(\ffF_3, \ffF_1)[1])^{\vee} &
}
\end{align*}
Here each horizontal and vertical sequences are distinguished 
triangle. The middle horizontal sequence is 
obtained by the description (\ref{MSext:spec}).
By taking the cone, 
we have the distinguished triangle
\begin{align*}
\ev_2^{\ast}\mathbb{L}_{\fM_S(v_2)} 
\to (\ev_1, \ev_3)^{\ast}
(\hH om_{p_{\fM \times \fM}}(\ffF_3, \ffF_1)[1])^{\vee} \to 
\mathbb{L}_{\ev}. 
\end{align*} 
By restricting it to $p$ and 
the associated exact sequence of cohomologies, we see that
\begin{align*}
\hH^{-2}(\mathbb{L}_{\ev}|_{p}) &=
\Ker(\Ext_S^2(F_2, F_2)^{\vee} \to \Ext_S^2(F_3, F_1)^{\vee}) \\
&=\Ker(\Hom(F_2, F_2 \otimes \omega_S) \stackrel{\eta}{\to} 
\Hom(F_1, F_3 \otimes \omega_S)). 
\end{align*}
Here the second identity is due to Serre duality, and 
one can check that the map $\eta$
is identified with the map in 
(\ref{fiber:p}). 
Since the left hand side is 
the fiber of $t_0(\Omega_{\ev}[-2]) \to \mM_S^{\ext}$ at $p$, 
we have the 
equivalence of $\mathbb{C}$-valued points 
\begin{align*}
\mM_X^{\ext}(v_{\bullet})(\mathbb{C}) \stackrel{\sim}{\to} 
t_0(\Omega_{\ev}[-2])(\mathbb{C}).
\end{align*}
It is straightforward to generalize the above arguments for 
$T$-valued point of $\mM_S^{\ext}(v_{\bullet})$ for 
a $\mathbb{C}$-scheme $T$, 
and therefore we obtain the isomorphism 
 (\ref{isom:ext}). 
\end{proof}

Let us take derived open substacks
$\fM_S(v_i)^{\fin} \subset \fM_S(v_i)$ 
of finite type satisfying 
\begin{align}\label{cond:ev}
\fM_S^{\ext}(v_{\bullet})^{\fin} \cneq 
\ev_2^{-1}(\fM_S(v_2)^{\fin}) 
\subset (\ev_1, \ev_3)^{-1}(\fM_S(v_1)^{\fin} \times \fM_S(v_3)^{\fin}). 
\end{align}
Then the diagram (\ref{dia:extS}) restricts to the diagram
\begin{align}\notag
\xymatrix{
\fM_S^{\ext}(v_{\bullet})^{\fin} \ar[r]^-{\ev_2}
 \ar[d]_-{(\ev_1, \ev_3)} & \fM_S(v_2)^{\fin} \\
\fM_S(v_1)^{\fin} \times \fM_S(v_3)^{\fin}. & 
}
\end{align}
Note that the vertical arrow is quasi-smooth and 
the horizontal arrow is proper. 
Therefore we have the induced functor
\begin{align}\label{cat:COHA1.5}
\ev_{2\ast}(\ev_1, \ev_3)^{\ast}
 \colon \Dbc(\fM_S(v_1)^{\fin}) \times \Dbc(\fM_S(v_3)^{\fin})
\to \Dbc(\fM_S(v_2)^{\fin}). 
\end{align}
We have the following corollary of Proposition~\ref{prop:isomM}. 

\begin{cor}\label{cor:induce}
For conical closed substacks 
$\zZ_i \subset \mM_X(v_i)$, 
suppose that 
\begin{align}\label{assum:Z}
(\ev_1^X, \ev_3^X)^{-1}((\zZ_1 \times \mM_X) \cup 
(\mM_X \times \zZ_3)) \subset
(\ev_2^X)^{-1}(\zZ_2).
\end{align}
Then the functor (\ref{cat:COHA1.5}) descends to the functor
\begin{align}\label{cat:COHA2}
\Dbc(\fM_S(v_1)^{\fin})/\cC_{\zZ_1^{\fin}} \times 
\Dbc(\fM_S(v_3)^{\fin})/\cC_{\zZ_3^{\fin}} \to 
\Dbc(\fM_S(v_2)^{\fin})/\cC_{\zZ_2^{\fin}}. 
\end{align}
\end{cor}
\begin{proof}
By (\ref{diagram:fN}), the diagram (\ref{dia:extS}) induces the diagram
\begin{align*}
\xymatrix{
t_0(\Omega_{\ev}[-1]) \ar[r] \ar[d] & t_0(\Omega_{\fM_S(v_2)}[-1]) \\
t_0(\Omega_{\fM_S(v_1)}[-1]) \times
t_0(\Omega_{\fM_S(v_3)}[-1]). &
}
\end{align*}
Under the isomorphisms (\ref{shifted:cot}), (\ref{isom:ext}),
one can check that 
the above diagram is identified with
the diagram (\ref{dia:extX}).
Therefore by Proposition~\ref{prop:FM}
and Lemma~\ref{lem:restrict},  
the assumption (\ref{assum:Z}) implies that 
the functor (\ref{cat:COHA1.5})
restricts to the functors
 \begin{align*}
\cC_{\zZ_1^{\fin}} \times \Dbc(\fM_S(v_3)^{\fin}) \to 
\cC_{\zZ_2^{\fin}}, \ 
\Dbc(\fM_S(v_1)^{\fin}) \times \cC_{\zZ_3^{\fin}} 
\to \cC_{\zZ_2^{\fin}}.
\end{align*}
Therefore the functor (\ref{cat:COHA1.5}) descends 
to the functor (\ref{cat:COHA2}). 
\end{proof}

\subsection{Categorical DT theory of stable sheaves on local surfaces}\label{subsec:catDT:sigma}
Let us take an element
\begin{align}\label{sigma:BH}
\sigma=B+iH \in \mathrm{NS}(S)_{\mathbb{C}}
\end{align}
such that $H$ is an ample class. 
For an object $E \in \Coh_{\rm{cpt}}(X)$, 
its $B$-twisted reduced Hilbert polynomial is defined by
\begin{align}\label{reduced:Hilb}
\overline{\chi}_{\sigma}(E, m) \cneq 
\frac{\chi(\pi_{\ast}E \otimes \oO_S(-B+mH))}{c} \in \mathbb{Q}[m]. 
\end{align}
Here $c$ is the coefficient of the highest degree term of 
the polynomial $\chi(\pi_{\ast}E \otimes \oO_S(-B+mH))$ in $m$. 
Note that the polynomial
$\overline{\chi}_{\sigma}(E, m)$ 
is determined by the numerical class of $E$, so 
it extends to the map
\begin{align*}
\overline{\chi}_{\sigma}(-, m) \colon 
N(S) \to \mathbb{Q}[m]
\end{align*}
such that $\overline{\chi}_{\sigma}(E, m)=\overline{\chi}_{\sigma}([E], m)$. 

By definition, 
an 
object $E \in \Coh_{\rm{cpt}}(X)$ is 
$\sigma$-\textit{(semi)stable} if it is a pure dimensional 
sheaf, and
for any subsheaf $0 \neq E' \subsetneq E$ we have
\begin{align*}
\overline{\chi}_{\sigma}(E', m) <
 (\le) \overline{\chi}_{\sigma}(E, m), \ m\gg 0.
\end{align*}
We have the open substack and 
the conical closed substack
\begin{align}\label{open:ss}
\mM_{X}^{\sigma\mathchar`-\rm{ss}}(v)
\subset \mM_X(v), \ 
\zZ_{\sigma\mathchar`-\rm{us}}(v)
\subset \mM_X(v)
\end{align}
where $\mM_{X}^{\sigma\mathchar`-\rm{ss}}(v)$
corresponds to $\sigma$-semistable sheaves
and $\zZ_{\sigma\mathchar`-\rm{us}}(v)$
is its complement. 
The moduli stack
$\mM_{X}^{\sigma\mathchar`-\rm{ss}}(v)$
is of finite 
type, while $\mathfrak{M}_S(v)$ is not 
of finite type in general, so not necessary QCA. 
Therefore we take a derived open substack
$\mathfrak{M}_{S}(v)^{\fin} \subset \mathfrak{M}_S(v)$
of finite type 
satisfying the condition  
\begin{align}\label{take:open}
\mM_{X}^{\sigma\mathchar`-\rm{ss}}(v) \subset 
\pi_{\ast}^{-1}(\mM_{S}(v)^{\fin}) 
=t_0(\Omega_{\mathfrak{M}_{S}(v)^{\fin}}[-1]). 
\end{align}
The $\mathbb{C}^{\ast}$-equivariant 
categorical DT theory for $\mM_{X}^{\sigma\mathchar`-\rm{ss}}(v)$
is defined as follows: 
\begin{defi}\label{catDT:stable}
The $\mathbb{C}^{\ast}$-equivariant 
categorical DT theory for 
the moduli stack $\mM_{X}^{\sigma\mathchar`-\rm{ss}}(v)$ 
is defined by 
\begin{align*}
&\mathcal{DT}^{\mathbb{C}^{\ast}}(\mM_{X}^{\sigma\mathchar`-\rm{ss}}(v)) \cneq 
\Dbc(\fM_S(v)^{\fin})/\cC_{\zZ_{\sigma\mathchar`-\rm{us}}(v)^{\fin}}. 
\end{align*}
\end{defi}

\begin{rmk}\label{rmk:open:independent}
By~\cite[Lemma~3.10]{TocatDT},  
the categorical DT theories in Definition~\ref{catDT:stable}
are independent of 
a choice of a finite type open substack 
$\mathfrak{M}_S(v)^{\fin}$ 
of $\mathfrak{M}_S(v)$ 
satisfying (\ref{take:open}), up to equivalence. 
We use the QCA condition for the above independence. 
\end{rmk}

\subsection{Categorical COHA for categorical DT theories}
For each polynomial 
$\overline{\chi} \in \mathbb{Q}[m]$, we set
\begin{align*}
N(S)_{\overline{\chi}} \cneq 
\{ v\in N(S) : \overline{\chi}_{\sigma}(v, m)=\overline{\chi}\} \cup \{0\}. 
\end{align*}
The following is the main result in this section. 
\begin{thm}\label{thm:COHA}
For $v_{\bullet}=(v_1, v_2, v_3) \in N(S)_{\overline{\chi}}^{\times 3}$
with $v_2=v_1+v_3$, 
the functor (\ref{cat:COHA}) descends to the functor
\begin{align}\label{COHA:DT}
\mathcal{DT}^{\mathbb{C}^{\ast}}(\mM_{X}^{\sigma\mathchar`-\rm{ss}}(v_1))
\times 
\mathcal{DT}^{\mathbb{C}^{\ast}}(\mM_{X}^{\sigma\mathchar`-\rm{ss}}(v_2))
\to \mathcal{DT}^{\mathbb{C}^{\ast}}(\mM_{X}^{\sigma\mathchar`-\rm{ss}}(v_3)). 
\end{align}
\end{thm}
\begin{proof}
By Corollary~\ref{cor:induce},
it is enough to show that 
\begin{align*}
(\ev_1^{X}, \ev_3^{X})^{-1}((\zZ_{\sigma\mathchar`-\rm{us}}(v_1) \times \mM_X(v_3)) 
\cup 
(\mM_X(v_1) \times \zZ_{\sigma\mathchar`-\rm{us}}(v_3))) \subset
(\ev_2^{X})^{-1}(\zZ_{\sigma\mathchar`-\rm{us}}(v_2)).
\end{align*}
The above inclusion follows from
Lemma~\ref{lem:stability}
below. Therefore 
we obtain the induced functor (\ref{COHA:DT}). 
\end{proof}

We have used the following lemma, whose proof is obvious 
from the definition of stability condition. 
\begin{lem}\label{lem:stability}
For $v_{\bullet} \in N(S)_{\overline{\chi}}^{\times 3}$
and a point  
of $\mM_X^{\ext}(v_{\bullet})$
corresponding to an exact sequence
on $X$
\begin{align*}
0 \to E_1 \to E_2 \to E_3 \to 0
\end{align*}
the object
$E_2$ is $\sigma$-semistable 
if and only if both of $E_1$, $E_3$ are 
$\sigma$-semistable. 
\end{lem}

By taking the associated morphisms on K-groups, 
we obtain the following corollary. 
\begin{cor}\label{cor:K}
The functors (\ref{COHA:DT}) determine the 
associative algebra structure on 
\begin{align}\label{KHA:DT}
\bigoplus_{v \in N(S)_{\overline{\chi}}}
K(\mathcal{DT}^{\mathbb{C}^{\ast}}(\mM_{X}^{\sigma\mathchar`-\rm{ss}}(v))).
\end{align}
\end{cor}
\begin{proof}
By the construction, the functor (\ref{COHA:DT})
fits into the commutative diagram
\begin{align*}
\xymatrix{
\Dbc(\fM_S(v_1)) \times \Dbc(\fM_S(v_3)) \ar[r] \ar[d] & 
\Dbc(\fM_S(v_2)) \ar[d] \\
\mathcal{DT}^{\mathbb{C}^{\ast}}(\mM_{X}^{\sigma\mathchar`-\rm{ss}}(v_1))
\times \mathcal{DT}^{\mathbb{C}^{\ast}}(\mM_{X}^{\sigma\mathchar`-\rm{ss}}(v_3))\ar[r] & 
\mathcal{DT}^{\mathbb{C}^{\ast}}(\mM_{X}^{\sigma\mathchar`-\rm{ss}}(v_2)).
}
\end{align*}
Here the top horizontal arrow is given by (\ref{cat:COHA}) and 
the vertical arrows are compositions of restrictions to 
$\fM_S(v_i)^{\fin}$ and quotient functors. 
The vertical arrows are surjective on K-theory, 
so we have the surjective map of algebras
\begin{align*}
\bigoplus_{v \in N(S)_{\overline{\chi}}}
K(\fM_S(v)) \to 
\bigoplus_{v \in N(S)_{\overline{\chi}}}
K(\mathcal{DT}^{\mathbb{C}^{\ast}}(\mM_{X}^{\sigma\mathchar`-\rm{ss}}(v))).
\end{align*}
The left hand side is associative by~\cite{PoSa}, hence (\ref{KHA:DT}) is 
also an associative algebra. 
\end{proof}
\begin{rmk}\label{rmk:indep}
As we mentioned in Remark~\ref{rmk:open:independent}, 
the DT categories $\mathcal{DT}^{\mathbb{C}^{\ast}}(\mM_X^{\sigma \sss}(v))$
are independent of choices of $\fM_S(v)^{\fin}$. 
It is straightforward to check that, under 
the above equivalences, the functor (\ref{COHA:DT}) is also 
independent of choices of $\fM_S(v_i)^{\fin}$
satisfying (\ref{cond:ev}). 
The same also applies for later constructions in Theorem~\ref{thm:PTaction}, 
Theorem~\ref{thm:Iaction}, Theorem~\ref{thm:Pt1} and Theorem~\ref{thm:Pt2}. 
\end{rmk}

\begin{rmk}\label{rmk:higher}
Porta and Sala pointed out to the author that 
the result of Theorem~\ref{thm:COHA} 
would apply to a dg-categorical setting,
so that the 
dg-enhancements 
$\mathcal{DT}_{\rm{dg}}^{\mathbb{C}^{\ast}}(\mM_X^{\sigma \sss}(v))$
of
$\mathcal{DT}^{\mathbb{C}^{\ast}}(\mM_X^{\sigma \sss}(v))$ 
form an $\mathbb{E}_1$-algebra. 
Similarly to~\cite[Section~4]{PoSa}, 
we need to involve 
 higher parts of 
the Waldhausen constructions
for the dg-enhancements of $\Coh(S)$
in order to control the higher associativity. 
\end{rmk}

As we mentioned in the introduction, 
if we have 
\begin{align}\label{stack:S}
\mM_X^{\sigma \sss}(v)=
t_0(\Omega_{\fM_S^{\sigma \sss}(v)[-1]})
\end{align}
then we have (see~\cite[Lemma~5.7]{TocatDT})
\begin{align}\label{algebra:surface}
\bigoplus_{v \in N(S)_{\overline{\chi}}}
K(\mathcal{DT}^{\mathbb{C}^{\ast}}(\mM_{X}^{\sigma\mathchar`-\rm{ss}}(v)))
=\bigoplus_{v \in N(S)_{\overline{\chi}}}
K(\mM_S^{\sigma\mathchar`-\rm{ss}}(v))
\end{align}
and the algebra structure on it 
is essentially the same one 
in~\cite{PoSa}.
Of course in general the condition (\ref{stack:S})
does not hold, and in this case the algebra structure of (\ref{KHA:DT})
is more difficult to describe (for example, see~\cite[Example~5.8]{TocatDT}). 
Here we give some examples of the algebra (\ref{algebra:surface})
when (\ref{stack:S}) holds. 
\begin{exam}\label{exam:zero}
(i) 
If we take $\overline{\chi} \equiv 1$, 
then we have
\begin{align*}
N(S)_{\overline{\chi}}
=\mathbb{Z} \cdot [\pt], \ 
[\pt] \cneq [\oO_x]
\end{align*}
where $x \in S$. 
Since
the stack $\mM_X^{\sigma \sss}(m[\mathrm{pt}])$
coincides with the 
stack of zero dimensional sheaves on $X$
with length $m$, the 
condition (\ref{stack:S}) is satisfied.
Then the algebra (\ref{algebra:surface})
\begin{align*}
\bigoplus_{v \in N(S)_{\overline{\chi}}}
K(\mathcal{DT}^{\mathbb{C}^{\ast}}(\mM_{X}^{\sigma\mathchar`-\rm{ss}}(v)))
=
\bigoplus_{m\ge 0}K(\mM_S(m[\pt]))
\end{align*}
is nothing but the K-theoretic Hall algebra of zero dimensional 
sheaves constructed by Zhao~\cite{Zhao}, which admits 
a morphism to the shuffle algebra. 

(ii)
Let $C=\mathbb{P}^1 \subset S$ be a $(-1)$-curve.
Suppose that 
we have 
\begin{align*}
N(S)_{\overline{\chi}}
=\mathbb{Z} \cdot [\oO_C(k)], \ k \in \mathbb{Z}
\end{align*}
for a fixed $k$. 
Since 
$\mM_X^{\sigma \sss}(m[\oO_C(k)])$
consists of the direct sum $\oO_C(k)^{\oplus m}$, 
the condition (\ref{stack:S}) is satisfied.
The algebra (\ref{algebra:surface}) is 
\begin{align}\label{K:simple}
\bigoplus_{v \in N(S)_{\overline{\chi}}}
K(\mathcal{DT}^{\mathbb{C}^{\ast}}(\mM_{X}^{\sigma\mathchar`-\rm{ss}}(v)))
=\bigoplus_{m\ge 0} K(\mathrm{BGL}_m(\mathbb{C}))
=\mathrm{Sym}^{\bullet}_{\mathbb{Z}}(\mathbb{Z}[z^{\pm 1}]). 
\end{align}
For $f \in S^n(\mathbb{Z}[z^{\pm 1}])$ and 
$g \in S^m(\mathbb{Z}[z^{\pm 1}])$, 
its product is given by (see~\cite[Proposition~9.1]{Tudor})
\begin{align*}
f \cdot g(z_1, \ldots, z_n, z_{n+1}, \ldots, z_{n+m})
=\mathrm{Sym}\left( 
\frac{f(z_1, \ldots, z_n) g(z_{n+1}, \ldots, z_{n+m})}{\prod_{i=1}^n
\prod_{j=n+1}^{n+m}(1-z_i z_j^{-1})}  \right). 
\end{align*}
Here $\mathrm{Sym}$ means the symmetrization. 
\end{exam}

\section{An action of zero dimensional DT categories to PT categories}
\label{sec:catPT}
In this section, 
we prove Theorem~\ref{intro:thmB}. 
We introduce the moduli stacks of pairs and their extensions, 
and recall the 
definition of PT categories
on local surfaces defined in~\cite{TocatDT}.
We then construct an 
action of DT categories of zero dimensional sheaves on them. 
By taking the associated action on the K-theory, we obtain a 
representation of K-theoretic Hall algebra of zero dimensional 
sheaves in Example~\ref{exam:zero} (i). 
\subsection{Moduli stacks of pairs}
For a smooth projective surface $S$, let 
$\mathfrak{M}_S$ 
be the 
derived moduli stack of coherent sheaves on $S$
considered in (\ref{dstack:MS}), 
and $\mathfrak{F}$ the universal 
object (\ref{F:universal}). 
We 
define the derived stack $\mathfrak{M}_S^{\dag}$ by 
\begin{align*}
\rho^{\dag} \colon 
\mathfrak{M}_S^{\dag} \cneq 
\mathbb{V}((p_{\mathfrak{M}\ast} \mathfrak{F})^{\vee})
\to \mathfrak{M}_S. 
\end{align*}
Here $p_{\mathfrak{M}} \colon S \times \mathfrak{M}_S \to \mathfrak{M}_S$
is the projection. 
For $T \in dAff$, the $T$-valued points
of $\fM_S^{\dag}$ form the $\infty$-groupoid of pairs
\begin{align*}
(\ffF, \xi), \ 
\xi \colon \oO_{S \times T} \to \ffF
\end{align*}
where $\ffF$ is a $T$-valued point of $\fM_S$. 

The classical truncation of $\mM_S^{\dag}$ is 
a 1-stack 
\begin{align}\label{MS:dag}
\mM_S^{\dag}\cneq t_0(\mathfrak{M}_S^{\dag})
=\Spec_{\mM_S}(S(\hH^0((p_{\mM\ast} \fF)^{\vee})). 
\end{align}
We have the universal pair
on $\mM_S^{\dag}$
\begin{align*}
\iI^{\bullet}=(\oO_{S \times \mM_S^{\dag}} \to \fF).
\end{align*}
Then we have the following 
description of the cotangent 
complex of $\fM_S^{\dag}$
\begin{align}\label{perf:obs2}
&\mathbb{L}_{\fM_S^{\dag}}|_{\mM_S^{\dag}}=\left(
\hH om_{p_{\mM^{\dag}}}(\iI^{\bullet}, \fF)\right)^{\vee}. 
\end{align}
Here $p_{\mM^{\dag}} \colon S \times \mM_S^{\dag} \to \mM_S^{\dag}$
is the projection. 
Also we have the decompositions into
open and closed substacks
\begin{align*}
\mathfrak{M}_S^{\dag}=\coprod_{v \in N(S)} \mathfrak{M}_S^{\dag}(v), \ 
\mM_S^{\dag}=\coprod_{v \in N(S)} \mM_S^{\dag}(v),
\end{align*}
where each component corresponds to pairs $(F, \xi)$
such that $[F]=v$. 

Let $\Coh_{\le 1}(S) \subset \Coh(S)$ be the 
subcategory of sheaves $F$ with $\dim \Supp(F) \le 1$. 
We define the subgroup 
$N_{\le 1}(S) \subset N(S)$ 
to be
\begin{align*}
N_{\le 1}(S) \cneq \Imm (K(\Coh_{\le 1}(S)) \to N(S)). 
\end{align*}
Note that we have an isomorphism
\begin{align}\label{isom:N}
N_{\le 1}(S) \stackrel{\cong}{\to} 
\mathrm{NS}(S) \oplus \mathbb{Z}, \ 
F \mapsto (l(F), \chi(F))
\end{align}
where $l(F)$ is the fundamental one cycle of $F$. 
Below we identify an element $v \in N_{\le 1}(S)$
with $(\beta, n) \in \mathrm{NS}(S) \oplus \mathbb{Z}$
by the above isomorphism. 

For $v \in N_{\le 1}(S)$, the 
derived stack $\fM_S^{\dag}(v)$ is quasi-smooth 
(see~\cite[Lemma~6.1]{TocatDT}). 
We have the $(-1)$-shifted cotangent stack, and 
its classical truncation
\begin{align}\label{shift:dag}
\Omega_{\fM_S^{\dag}(v)}[-1] \to 
\fM_S^{\dag}(v), \  
t_0(\Omega_{\fM_S^{\dag}(v)}[-1]) \stackrel{\eta}{\to}
 \mM_S^{\dag}(v). 
\end{align}
From (\ref{perf:obs2}),
the fiber of the morphism
$\eta$ 
at the pair $(F, \xi)$
is 
\begin{align}\label{fiber:eta}
\eta^{-1}((F, \xi))=\Hom_S(I^{\bullet}, F[1])^{\vee}=
\Hom_S(F \otimes \omega_S^{-1}, I^{\bullet}[1]). 
\end{align}
Here $I^{\bullet}$ is the two 
term complex $(\oO_S \stackrel{\xi}{\to} F)$
such that $\oO_S$ is located in degree zero. 

On the other hand, let 
$\bB_S$ be the category of diagrams
\begin{align}\label{dia:BS2}
\xymatrix{
0 \ar[r] & \vV
 \ar[r] \ar@{.>}[rd]_-{\xi} & \uU \ar[r]
\ar[d]^-{\phi} & F \otimes \omega_S^{-1} 
\ar[r] & 0 \\
&   & F  & 
}
\end{align}
for $\vV \in \langle \oO_S\rangle_{\rm{ex}}$ and 
$F \in \Coh_{\le 1}(S)$. 
Here the top sequence is an exact sequence
of coherent sheaves on $S$.
Then 
$\bB_S$ is an abelian category, 
and we denote by $\bB_S^{\le 1} \subset \bB_S$
the subcategory of diagrams (\ref{dia:BS2}) with $\rank(\vV) \le 1$. 
The following result was proved in~\cite{TocatDT}. 

\begin{thm}\emph{(\cite[Theorem~6.3]{TocatDT})}\label{thm:diaB}
The stack 
$t_0(\Omega_{\fM_S^{\dag}(v)}[-1])$ is 
isomorphic to the stack of 
diagrams 
(\ref{dia:BS2}) 
with $\vV=\oO_S$
and $[F]=v$. 
Under the above isomorphism, the 
map $\eta$ in (\ref{shift:dag}) sends the diagram (\ref{dia:BS2})
to the 
pair $(F, \xi)$. 
\end{thm}

The correspondence in Theorem~\ref{thm:diaB}
is explained in the following way. 
Over the pair $(F, \xi)$, 
the diagram (\ref{dia:BS2})
with $\vV=\oO_S$ and $[F]=v$
determines a point in the fiber (\ref{fiber:eta})
by 
\begin{align}\label{map:eta+}
F \otimes \omega_S^{-1}[-1] \stackrel{\sim}{\leftarrow} (\oO_S \to \uU)
\stackrel{(\id, \phi)}{\longrightarrow} (\oO_S \stackrel{\xi}{\to} F)=I^{\bullet}. 
\end{align}
Conversely given a morphism 
$\vartheta \colon F \otimes \omega_S^{-1} \to I^{\bullet}[1]$, 
then we associate the commutative diagram
\begin{align}\label{dia:OIF}
\xymatrix{
& \oO_S\ar[r]\ar@{=}[d] & \uU\ar[r] \ar@{.>}[d]^-{\phi}& F \otimes \omega_S^{-1} \ar[r] \ar[d]^-{\vartheta} & \oO_S[1] \ar@{=}[d] \\
I^{\bullet} \ar[r]& \oO_S\ar[r]^-{\xi} &
F \ar[r] & I^{\bullet}[1] \ar[r] & \oO_S[1]. 
}
\end{align}
Here horizontal sequences are distinguished triangles. 
Therefore there exists a morphism $\phi \colon \uU \to F$ which 
makes the above diagram commutative, and 
can be shown to be unique (see~\cite[Lemma~9.2]{TocatDT}).  

\begin{rmk}\label{rmk:tilt}
Here we note that the above morphism $\vartheta$ is regarded as a 
morphism in 
some abelian category. 
Indeed let 
$\tT, \fF \subset \Coh(S)$
be subcategories such that 
$\tT$ consists of torsion sheaves and $\fF$
consists of torsion free sheaves. 
Then $(\tT, \fF)$ forms a torsion pair, 
and we have the associated tilting~\cite{HRS}
\begin{align}\label{tilt:sharp}
\Coh^{\sharp}(S) \cneq \langle \fF[1], \tT\rangle \subset \Dbc(S). 
\end{align}
As a general result of tilting, 
$\Coh^{\sharp}(S)$ is the heart of a t-structure on $\Dbc(S)$, 
hence an abelian category. 
Then both of $F \otimes \omega_S^{-1}$ and 
$I^{\bullet}[1]$ are objects in $\Coh^{\sharp}(S)$
and $\vartheta$ is a morphism in $\Coh^{\sharp}(S)$. 
\end{rmk}

\subsection{Derived moduli stacks of extensions of pairs}
In this subsection, we introduce the derived moduli stack
of extensions of pairs, and extend the 
categorified COHA structure to the
module structure over it on the derived category 
of coherent sheaves on derived moduli stacks of pairs. 
We define the derived stack 
$\fM_S^{\ext, \dag}$ by the 
Cartesian square
\begin{align*}
\xymatrix{
\fM_S^{\ext, \dag} \ar@{}[rd]|\square
\ar[r] \ar[d]_-{(\ev_1^{\dag}, \ev_3^{\dag})} 
& \fM_S^{\ext} \ar[d]^-{(\ev_1, \ev_3)} \\
\fM_S^{\dag} \times \fM_S \ar[r]^-{(\rho, \id)} & 
\fM_S \times \fM_S. 
}
\end{align*}
For $T \in dAff$, the $T$-valued points of $\fM_S^{\dag}$ form the 
$\infty$-groupoid of diagrams
\begin{align}\label{diagram:dag}
\xymatrix{
\oO_{S \times T} \ar[d]_-{\xi} &  &  \\
\ffF_1 \ar[r] & \ffF_2 \ar[r] & \ffF_3.
}
\end{align}
Here the bottom sequence is a $T$-valued point of 
$\fM_S^{\ext}$. 
Let us take $v_{\bullet}=(v_1, v_2, v_3) \in N_{\le 1}(S)^{\times 3}$. 
We have the open and closed derived substack
\begin{align*}
\fM_S^{\ext, \dag}(v_{\bullet}) \subset \fM_S^{\ext, \dag}
\end{align*}
corresponding to the diagram
(\ref{diagram:dag}) such that 
the bottom sequence is a $T$-valued point of $\fM_S^{\ext}(v_{\bullet})$. 
\begin{lem}
The derived stack $\fM_S^{\ext, \dag}(v_{\bullet})$ is quasi-smooth. 
\end{lem}
\begin{proof}
The lemma follows since $(\ev_1^{\dag}, \ev_3^{\dag})$ is quasi-smooth 
and both of $\fM_S^{\dag}(v_1)$, $\fM_S(v_3)$ are quasi-smooth. 
\end{proof}

We have the diagram
\begin{align}\label{dia:MSdag}
\xymatrix{
\fM_S^{\ext, \dag}(v_{\bullet}) \ar[r]^-{\ev_2^{\dag}} \ar[d]_-{(\ev_1^{\dag}, \ev_3^{\dag})}
 & \fM_S^{\dag}(v_2) \\
\fM_S^{\dag}(v_1) \times \fM_S(v_3). 
}
\end{align}
Here $\ev_2^{\dag}$ 
is obtained by sending 
a diagram (\ref{diagram:dag})
to the composition 
$\oO_{S \times T} \stackrel{\xi}{\to} \ffF_1 \to \ffF_2$. 
Note that the vertical arrow is quasi-smooth. 
As for the horizontal arrow, we have the following. 
\begin{lem}\label{lem:evproper}
The morphism $\ev_2^{\dag}$ in the diagram (\ref{dia:MSdag})
 is proper. 
\end{lem}
\begin{proof}
For a point of $\fM_S^{\dag}(v_2)$
corresponding to a pair $(F_2, \xi_2)$, 
the fiber of $\ev_2^{\dag}$ at this point 
corresponds to diagrams
\begin{align*}
\oO_S \stackrel{\xi_2}{\to}F_2 \twoheadrightarrow F_3, \ 
[F_3]=v_3,
\end{align*}
whose composition is zero. 
The classical moduli space of such diagrams 
is a closed subscheme of 
the Quot scheme parameterizing quotients 
$F_2 \twoheadrightarrow F_3$ with $[F_3]=v_3$, hence 
it is a proper scheme. 
Therefore $\ev_2^{\dag}$ is proper. 
\end{proof}
By Lemma~\ref{lem:evproper}, 
the diagram (\ref{dia:MSdag}) 
induces the functor
\begin{align}\label{FM:extP}
\ev_{2\ast}^{\dag}(\ev_1^{\dag}, \ev_3^{\dag})^{\ast}
 \colon \Dbc(\fM_S^{\dag}(v_1)) \times \Dbc(\fM_S(v_3))
\to \Dbc(\fM_S^{\dag}(v_2)). 
\end{align}

\subsection{Moduli stacks of D0-D2-D6 bound states}
\label{subsec:moduliD026}
In this subsection, we recall the notion of D0-D2-D6 bound 
states and their moduli stacks. 
Let $\overline{X}$ be the projective 
compactification of $X$
\begin{align*}
X \subset \overline{X} \cneq \mathbb{P}_S(\omega_S \oplus \oO_S)
=X \cup S_{\infty}.
\end{align*}
Here $S_{\infty}$ is the divisor at the infinity. 
The category of D0-D2-D6 bound states on the 
non-compact CY 3-fold $X=\mathrm{Tot}_S(\omega_S)$ 
is defined by the extension closure
in $D^b_{\rm{coh}}(\overline{X})$
\begin{align*}
\aA_{X} \cneq 
\langle \oO_{\overline{X}}, \Coh_{\le 1}(X)[-1]
\rangle_{\rm{ex}} \subset \Dbc(\overline{X}). 
\end{align*}
Here $\Coh_{\le 1}(X)$
is the subcategory of objects 
in $\Coh_{\mathrm{cpt}}(X)$ whose supports have 
dimensions less than or equal to one. 
We regard $\Coh_{\le 1}(X)$ 
 as a subcategory of $\Coh(\overline{X})$
by the push-forward of the open 
immersion $X \subset \overline{X}$. 
The arguments in~\cite[Lemma~3.5, Proposition~3.6]{MR2669709} 
show
that 
$\aA_{X}$ is an abelian subcategory
of $\Dbc(\overline{X})$. 
We denote by $\aA_X^{\le 1} \subset \aA_X$
the subcategory of objects $E \in \aA_X$ with $\rank(E) \le 1$. 

There is a group homomorphism
\begin{align*}
\cl \colon K(\aA_X) \to \mathbb{Z} \oplus N_{\le 1}(S)
\end{align*}
characterized by the condition 
that $\cl(\oO_X)=(1, 0)$
and $\cl(F)=(0, [\pi_{\ast}F])$
for $F \in \Coh_{\le 1}(X)$. 
We define the (classical) moduli stack of rank one 
objects in $\aA_X$
to be the 2-functor
\begin{align*}
\mM_X^{\dag} \colon Aff^{op} \to 
Groupoid
\end{align*}
whose $T$-valued points for $T \in Aff$
form the groupoid
of data
\begin{align}\label{data:A}
\eE \in D^b_{\rm{coh}}(\overline{X} \times T), \ 
\lambda \colon 
\eE \dotimes \oO_{S_{\infty} \times T} \stackrel{\cong}{\to}
\oO_{S_{\infty} \times T}
\end{align}
such that for any closed point $x \in T$, 
we have 
\begin{align*}
\eE_x \cneq \dL i_x^{\ast} \eE \in \aA_X, \ 
i_x \colon \overline{X} \times \{x\} \hookrightarrow \overline{X} \times T.
\end{align*} 
The isomorphisms of the groupoid $\mM_X^{\dag}(T)$ are 
given by isomorphisms of objects $\eE_T$ which commute with 
the trivializations at the infinity. 
We have the decomposition of $\mM_X^{\dag}$ 
into open and closed substacks
\begin{align*}
\mM_X^{\dag}=\coprod_{v \in N_{\le 1}(S)} \mM_X^{\dag}(v)
\end{align*}
where $\mM_X^{\dag}(v)$ corresponds to 
$E \in \aA_X$ with $\cl(E)=(1, -v)$. 
The following result is proved in~\cite{TocatDT}: 

\begin{thm}\label{thm:D026}
\emph{(\cite[Theorem~6.3]{TocatDT})}
There is an equivalence of categories
\begin{align}\label{equiv:AB}
\aA_X^{\le 1} \stackrel{\sim}{\to} \bB_S^{\le 1}. 
\end{align}
Moreover 
the above equivalence together with the
isomorphism in Theorem~\ref{thm:D026} induce the isomorphism 
of stacks over $\mM_S^{\dag}(v)$
\begin{align}\label{isom:dag} 
\mM_X^{\dag}(v) \stackrel{\cong}{\to} 
t_0(\Omega_{\mathfrak{M}_S^{\dag}(v)}[-1]).  
\end{align}
\end{thm}

\subsection{Moduli stacks of extensions in $\aA_X$}
In this subsection, we consider the classical 
moduli stack of exact sequences in $\aA_X$
of the form $0 \to E_1 \to E_2 \to E_3[-1] \to 0$
where $E_1$, $E_2$ are rank one objects in $\aA_X$
and $E_3 \in \Coh_{\le 1}(X)$. 
More precisely, we 
take $v_{\bullet}=(v_1, v_2, v_3) \in N_{\le 1}(S)^{\times 3}$
and define the classical 
stack
\begin{align*}
\mM_X^{\ext, \dag}(v_{\bullet}) \colon Aff^{op} \to Groupoid
\end{align*}
by sending $T \in Aff$ to the groupoid 
of distinguished triangles
$\eE_1 \stackrel{i}{\to} \eE_2 \to \eE_3[-1]$
together with commutative diagrams
\begin{align}\notag
\xymatrix{
\eE_1 \dotimes \oO_{S_{\infty} \times T} \ar[r]^-{i} 
\ar[d]_-{\lambda_1}^-{\cong} & 
\eE_2 \dotimes \oO_{S_{\infty} \times T} \ar[d]_-{\lambda_2}^-{\cong} \\
\oO_{S_{\infty} \times T} \ar@{=}[r] & \oO_{S_{\infty} \times T}. 
}
\end{align}
Here $(\eE_i, \lambda_i)$ for $i=1, 2$ are $T$-valued points of 
$\mM_X^{\dag}(v_i)$
and $\eE_3$ is a $T$-valued point of $\mM_X(v_3)$. 
We also have the evaluation morphisms
\begin{align}\label{dia:Xext:dag}
\xymatrix{
\mM_X^{\ext, \dag}(v_{\bullet}) 
\ar[r]^-{\ev_2^{X, \dag}} \ar[d]_-{(\ev_1^{X, \dag}, \ev_3^{X, \dag})} 
& \mM_X^{\dag}(v_2) \\
\mM_X^{\dag}(v_1) \times \mM_X(v_3) & 
}
\end{align}
where $\ev_i^{X, \dag}$ sends
 $\eE_{\bullet}$ to $\eE_i$. 
We have the following description of 
each point of $\mM_X^{\ext, \dag}(v_{\bullet})$. 

\begin{lem}\label{prop:equiv}
Giving a point of $\mM_X^{\ext, \dag}(v_{\bullet})$
is equivalent to giving 
a point of $\mM_S^{\ext, \dag}(v_{\bullet})$ 
\begin{align}\label{dia:F}
\xymatrix{
& \oO_S \ar@{=}[r] \ar[d] & \oO_S \ar[d] & &\\
0 \ar[r] & F_1 \ar[r]^-{i} & F_2 \ar[r]^-{j} & F_3 \ar[r] & 0. 
}
\end{align}
together with 
a morphism of distinguished triangles
\begin{align}\label{dia:I}
\xymatrix{
F_1 \ar[d]_-{\vartheta_1} \otimes \omega_S^{-1} \ar[r]^-{i} & F_2 \otimes 
\omega_S^{-1} \ar[r]^-{j} \ar[d]_-{\vartheta_2} 
& F_3 \otimes \omega_S^{-1} \ar[d]_-{\vartheta_3}  \\
 I_{1}^{\bullet}[1]  \ar[r] 
 & I_{2}^{\bullet}[1] \ar[r]^-{k} & F_3.
}
\end{align}
Here 
$I_{i}^{\bullet}=(\oO_S \to F_i)$
and the bottom sequence of (\ref{dia:I})
 is given by the diagram (\ref{dia:F}). 
\end{lem}
\begin{proof}
By the equivalence (\ref{equiv:AB}), 
the stack $\mM_X^{\ext, \dag}(v_{\bullet})$
is isomorphic to the stack of exact sequences in $\bB_S^{\le 1}$
of the form
\begin{align*}
0 &\to 
\left(\xymatrix{
0 \ar[r] & \oO_S  \ar@{.>}[rd]
 \ar[r]  & \uU_1 \ar[r]
\ar[d] & F_1 \otimes \omega_S^{-1} 
\ar[r] & 0 \\
&   & F  & 
}\right)
\to 
\left(\xymatrix{
0 \ar[r] & \oO_S  \ar@{.>}[rd]
 \ar[r]  & \uU_2 \ar[r]
\ar[d] & F_2 \otimes \omega_S^{-1} 
\ar[r] & 0 \\
&   & F_2  & 
}\right) \\
&\to 
\left(\xymatrix{
0 \ar[r] & 0  \ar@{.>}[rd]
 \ar[r]  & \uU_3 \ar[r]
\ar[d] & F_3 \otimes \omega_S^{-1} 
\ar[r] & 0 \\
&   & F_3  & 
}\right)
 \to 0
\end{align*}
such that $[F_i]=v_i$. 
By taking the associated exact sequence for dotted arrows, 
we obtain the diagram (\ref{dia:F}).
From the correspondence of objects in $\bB_S^{\le 1}$ 
explained after Theorem~\ref{thm:diaB}, the above 
diagram immediately gives the diagram (\ref{dia:I}). 

Conversely suppose that diagrams (\ref{dia:F}), (\ref{dia:I})
are given. 
Then by composing the diagram (\ref{dia:I}) 
with $I^{\bullet}_i[1] \to \oO_S[1]$
and taking cones, 
we obtain the commutative diagram
\begin{align*}
\xymatrix{
\uU_1 \ar[r] \ar[d] & \uU_2 \ar[r] \ar[d] & \uU_3 \ar[d] \\
F_1 \otimes \omega_S^{-1} \ar[r] \ar[d] & 
F_2 \otimes \omega_S^{-1} \ar[r] \ar[d] &
F_3 \otimes \omega_S^{-1} \\
\oO_S[1] \ar@{=}[r] & \oO_S[1]. 
}
\end{align*}
Here each horizontal and vertical sequences are distinguished triangles. 
The above diagram together with (\ref{dia:F}) give 
an exact sequence in $\bB_S^{\le 1}$, hence an exact 
sequence in $\aA_X^{\le 1}$ by the equivalence (\ref{equiv:AB}). 
\end{proof}

For $v_{\bullet}=(v_1, v_2, v_3) \in N_{\le 1}(S)^{\times 3}$, 
we have the morphism
\begin{align}\label{mor:pi*}
\pi_{\ast}^{\dag} \colon \mM_X^{\ext, \dag}(v_{\bullet}) \to \mM_S^{\ext, \dag}(v_{\bullet})
\end{align}
by sending a point in $\mM_X^{\ext, \dag}(v_{\bullet})$ 
to the diagram (\ref{dia:F}). 
On the other hand, 
let $\ev^{\dag}$ be the morphism
\begin{align*}
\ev^{\dag}=(\ev_1^{\dag}, \ev_2^{\dag}, \ev_3^{\dag}) \colon 
\fM_S^{\ext, \dag}(v_{\bullet})
 \to \fM_S^{\dag}(v_1) \times \fM_S^{\dag}(v_2) \times \fM_S(v_3). 
\end{align*}
We have the following proposition. 
\begin{prop}\label{prop:ev+}
We have an isomorphism over $\mM_S^{\ext, \dag}(v_{\bullet})$
\begin{align}\label{isom:ev+}
\mM_X^{\ext, \dag}(v_{\bullet}) \stackrel{\cong}{\to}
 t_0(\Omega_{\ev^{\dag}}[-2]). 
\end{align}
\end{prop}
\begin{proof}
Let $p \in \mM_S^{\ext}(v_{\bullet})$ be a point 
corresponding to the diagram (\ref{dia:F}). 
By Lemma~\ref{prop:equiv}, the fiber of 
the morphism (\ref{mor:pi*}) 
at $p$ is given by the diagram (\ref{dia:I}). 
By Remark~\ref{rmk:tilt}, 
the diagram (\ref{dia:I}) is regarded 
as a morphism of exact sequences in
the abelian category $\Coh^{\sharp}(S)$. 
Therefore similarly to the proof of Proposition~\ref{prop:isomM}, we have 
\begin{align}\label{fiber:pidag}
(\pi_{\ast}^{\dag})^{-1}(p)=
\Ker(\Hom(F_2 \otimes \omega_S^{-1}, I_2^{\bullet}[1])
\stackrel{k \circ (-) \circ i}{\longrightarrow} 
\Hom(F_1 \otimes \omega_S^{-1}, F_3)). 
\end{align}
On the other hand, we have the commutative diagram
\begin{align*}
\xymatrix{
(\ev_1^{\dag}, \ev_3^{\dag})^{\ast}
\mathbb{L}_{\fM_S^{\dag}(v_1) \times \fM_S(v_3)} \ar@{=}[r] \ar[d] & 
(\ev_1^{\dag}, \ev_3^{\dag})^{\ast}
\mathbb{L}_{\fM_S^{\dag}(v_1) \times \fM_S(v_3)} \ar[d] &  \\
(\ev^{\dag})^{\ast}\mathbb{L}_{\fM_S^{\dag}(v_1) \times \fM_S^{\dag}(v_2) 
\times
\fM_S(v_3)} \ar[r] \ar[d] & 
\mathbb{L}_{\fM_S^{\ext, \dag}(v_{\bullet})} 
\ar[r] \ar[d] & \mathbb{L}_{\ev^{\dag}} \\
(\ev_2^{\dag})^{\ast}\mathbb{L}_{\fM_S^{\dag}(v_2)} & 
(\ev_1^{\dag}, \ev_3^{\dag})^{\ast}
(\hH om_{p_{\fM^{\dag} \times \fM}}(\ffF_3, \ffF_1)[1])^{\vee}. &
}
\end{align*}
Here each horizontal and vertical sequences are distinguished 
triangle. 
By taking the cone, 
we obtain the distinguished triangle
\begin{align*}
(\ev_2^{\dag})^{\ast}\mathbb{L}_{\fM_S^{\dag}(v_2)} \to 
(\ev_1^{\dag}, \ev_3^{\dag})^{\ast}
(\hH om_{p_{\fM^{\dag} \times \fM}}(\ffF_3, \ffF_1)[1])^{\vee}
\to \mathbb{L}_{\ev^{\dag}}. 
\end{align*} 
By restricting it to $p$ and 
the associated exact sequence of cohomologies, we see that
\begin{align*}
\hH^{-2}(\mathbb{L}_{\ev^{\dag}}|_{p})=
\Ker(\Hom(F_2\otimes \omega_S^{-1}, I_2^{\bullet}[1]) \to 
\Hom(F_1, F_3 \otimes \omega_S)),
\end{align*}
which is identified with (\ref{fiber:pidag}). 
Therefore similarly to Proposition~\ref{prop:isomM}, we have the 
isomorphism (\ref{isom:ev+}). 
\end{proof}

Let us take finite type derived open substacks
$\fM_S^{\dag}(v_i)^{\fin} \subset \fM_S^{\dag}(v_i)$ 
for $i=1, 2$ and 
$\fM_S(v_3)^{\fin} \subset \fM_S(v_3)$
satisfying 
\begin{align}\label{cond:ev1.5}
\fM_S^{\ext, \dag}(v_{\bullet})^{\fin} \cneq 
(\ev_2^{\dag})^{-1}(\fM^{\dag}_S(v_2)^{\fin}) 
\subset (\ev^{\dag}_1, \ev^{\dag}_3)^{-1}(\fM^{\dag}_S(v_1)^{\fin} 
\times \fM_S(v_3)^{\fin}). 
\end{align}
Then the diagram (\ref{dia:MSdag}) restricts to the diagram
\begin{align}\label{dia:circ1.5}
\xymatrix{
\fM_S^{\ext, \dag}(v_{\bullet})^{\fin} \ar[r]^-{\ev^{\dag}_2}
 \ar[d]_-{(\ev_1^{\dag}, \ev_3^{\dag})} & \fM_S^{\dag}(v_2)^{\fin} \\
\fM_S^{\dag}(v_1)^{\fin} \times \fM_S(v_3)^{\fin}. & 
}
\end{align}
The vertical arrow is quasi-smooth and
the horizontal arrow is proper. 
Therefore we have the induced functor
\begin{align}\label{cat:COHA2.5}
\ev^{\dag}_{2\ast}(\ev^{\dag}_1, \ev^{\dag}_3)^{\ast}
 \colon \Dbc(\fM_S^{\dag}(v_1)^{\fin}) \times \Dbc(\fM_S(v_3)^{\fin})
\to \Dbc(\fM_S^{\dag}(v_2)^{\fin}). 
\end{align}
Similarly to Corollary~\ref{cor:induce}, we also have 
the following corollary of Proposition~\ref{prop:ev+}. 
\begin{cor}\label{cor:induce2}
For conical closed substacks 
$\zZ_i \subset \mM_X^{\dag}(v_i)$
for $i=1, 2$ 
and $\zZ_3 \subset \mM_X(v_3)$
suppose that the following condition holds
\begin{align}\notag
(\ev_1^{X, \dag}, \ev_3^{X, \dag})^{-1}((\zZ_1 \times \mM_X(v_3)) \cup 
(\mM_X^{\dag}(v_1) \times \zZ_3)) \subset
(\ev_2^{X, \dag})^{-1}(\zZ_2).
\end{align}
Then the functor (\ref{cat:COHA2.5}) descends to the functor
\begin{align}\notag
\Dbc(\fM_S^{\dag}(v_1)^{\fin})/\cC_{\zZ_1^{\fin}} \times 
\Dbc(\fM_S(v_3)^{\fin})/\cC_{\zZ_3^{\fin}}
 \to \Dbc(\fM_S^{\dag}(v_2)^{\fin})/\cC_{\zZ_2^{\fin}}. 
\end{align}
\end{cor}

\subsection{Categorical PT theory}
Here we recall the definition of PT categories 
and give a proof of Theorem~\ref{intro:thmB}. 
By definition, a PT stable pair consists of a pair~\cite{MR2545686}
\begin{align}\label{PTpair}
(F, s), \  F \in \Coh_{\le 1}(X), \ s \colon \oO_X \to F
\end{align}
such that $F$ is pure one dimensional and $s$ is 
surjective in dimension one. 
For $(\beta, n) \in N_{\le 1}(S)$, 
we denote by 
\begin{align}\notag
P_n(X, \beta)
\end{align}
the moduli space 
of PT stable pairs (\ref{PTpair})
satisfying $[\pi_{\ast}F]=(\beta, n)$. 
The moduli space of stable pairs $P_n(X, \beta)$ is 
known to be a quasi-projective scheme. 

We have the open immersion
\begin{align*}
P_n(X, \beta) \subset \mM_X^{\dag}(\beta, n)
\end{align*}
sending a pair $(F, s)$ to the
two term complex
$(\oO_{\overline{X}} \stackrel{s}{\to} F)$. 
We define the following 
conical closed substack
\begin{align*}
\zZ_{P\us}(\beta, n) \cneq \mM_X^{\dag}(\beta, n) \setminus P_n(X, \beta)
\subset \mM_X^{\dag}(\beta, n). 
\end{align*}
Since $P_n(X, \beta)$ is a quasi-projective scheme, 
there is a 
derived open substack 
$\mathfrak{M}_S^{\dag}(\beta, n)^{\fin}
\subset \mathfrak{M}_S^{\dag}(\beta, n)$ 
of finite type
such that 
\begin{align}\label{PT:qcompact}
P_n(X, \beta) \subset 
t_0(\Omega_{\mathfrak{M}_S^{\dag}(\beta, n)^{\fin}}[-1])
\subset \mM_X^{\dag}(\beta, n).  
\end{align}
Here we have used the isomorphism (\ref{isom:dag}). 
\begin{defi}\label{def:catPT}\emph{(\cite[Definition~6.6]{TocatDT})}
The $\mathbb{C}^{\ast}$-equivariant 
categorical PT theory is defined by
\begin{align*}
\dD \tT^{\mathbb{C}^{\ast}}(P_n(X, \beta)) \cneq 
D^b_{\rm{coh}}(\mathfrak{M}_S^{\dag}(\beta, n)^{\fin})/ 
\cC_{\zZ_{P\us}(\beta, n)^{\fin}}.
\end{align*}
\end{defi}
Similarly to Remark~\ref{rmk:open:independent}, 
the above definition is independent of 
a choice of $\fM_S^{\dag}(\beta, n)^{\fin}$
satisfying (\ref{PT:qcompact}). 
Let us take $v_{\bullet} \in N_{\le 1}(S)^{\times 3}$ to be 
\begin{align*}
v_1=(\beta, n), \ v_2=(\beta, n+m), \ 
v_3=(0, m)=m[\mathrm{pt}].
\end{align*}
We also take 
derived open substacks 
of finite type 
\begin{align*}
\fM_S^{\dag}(v_1)^{\fin} \subset \fM_S^{\dag}(v_1), \ 
\fM_S^{\dag}(v_2)^{\fin} \subset \fM_S^{\dag}(v_2), \ 
\fM_S^{\fin}(v_3)=\fM_S(m[\mathrm{pt}])
\end{align*}
satisfying (\ref{PT:qcompact}) for $v_1$, $v_2$ and 
the condition
(\ref{cond:ev1.5}).
Here we note that 
$\fM_S(m[\mathrm{pt}])$ is the derived moduli 
stack of zero dimensional sheaves of 
length $m$, so it is of finite type. 
Then the diagram ({\ref{dia:circ1.5}) induces the functor
\begin{align}\label{funct:circ}
\ev_{2\ast}^{\dag}(\ev_1^{\dag}, \ev_3^{\dag})^{\ast}
 \colon \Dbc(\fM_S^{\dag}(v_1)^{\fin}) \times 
\Dbc(\fM_S(m[\mathrm{pt}])) \to 
\Dbc(\fM_S^{\dag}(v_2)^{\fin}).
\end{align}

\begin{thm}\label{thm:PTaction}
The functor (\ref{funct:circ})
descends to the functor
\begin{align}\label{descend:PT}
\dD \tT^{\mathbb{C}^{\ast}}(P_n(X, \beta))
\times \dD \tT^{\mathbb{C}^{\ast}}(\mM_X^{\sigma \sss}(m[\mathrm{pt}]))
\to \dD \tT^{\mathbb{C}^{\ast}}(P_{n+m}(X, \beta)). 
\end{align}
\end{thm}
\begin{proof}
By Lemma~\ref{lem:PT} below, we have 
the following inclusion
\begin{align*}
(\ev_1^{X, \dag}, \ev_3^{X, \dag})^{-1}(\zZ_{P\us}(v_1)
 \times \mM_X(v_3))
\subset (\ev_2^{X, \dag})^{-1}(\zZ_{P\us}(v_2)). 
\end{align*}
Therefore the result follows from 
Corollary~\ref{cor:induce2}. 
\end{proof}

Here we have used the following lemma. 
\begin{lem}\label{lem:PT}
For an exact sequence $0 \to E_1 \to E_2 \to E_3 \to 0$
in $\aA_X$, 
suppose that $E_2=(\oO_{\overline{X}} \stackrel{s_2}{\to} F_2)$ for a PT stable 
pair $(F_2, s_2)$ and $E_3=Q[-1]$ for a zero dimensional sheaf $Q$. 
Then $E_1=(\oO_{\overline{X}} \stackrel{s_1}{\to} F_1)$
for a PT stable pair $(F_1, s_1)$. 
\end{lem}
\begin{proof}
By taking the long exact sequence of cohomologies, 
we have the surjection
\begin{align}\label{surj:s}
\hH^1(E_2)=\Cok(s_2) \twoheadrightarrow \hH^1(E_3)=Q. 
\end{align}
On the other hand, we have the following commutative diagram
\begin{align*}
\xymatrix{
F_1[-1] \ar[d] & E_1 \ar[d] & \\
F_2[-1] \ar[r] \ar[d]_-{\alpha[-1]} & E_2 \ar[r] \ar[d] & \oO_{\overline{X}}  \\
Q[-1] \ar@{=}[r] & Q[-1].  & 
}
\end{align*}
Here vertical sequences are distinguished triangles. 
The map $\alpha \colon F_2 \to Q$ is the composition of 
$F_2 \twoheadrightarrow \Cok(s_2)$ 
with the morphism (\ref{surj:s}), so it is surjective. 
Therefore $F_1$ is a subsheaf of $F_2$, which is a pure one dimensional 
sheaf. By taking the cones, 
we obtain the 
distinguished triangle 
$F_1[-1] \to E_1 \to \oO_{\overline{X}}$, 
hence $E_1=(\oO_{\overline{X}} \stackrel{s_1}{\to} F_1)$
for a pair $(F_1, s_1)$. 
Since $(F_1, s_1)$ is isomorphic to $(F_2, s_2)$ outside the 
support of $Q$ and $s_2$ is surjective in dimension one, 
 $s_1$ is also surjective in dimension one.
Therefore $(F_1, s_1)$ is PT stable. 
\end{proof}

Similarly to Corollary~\ref{cor:K}, we have the following 
corollary of Theorem~\ref{thm:PTaction}. 
\begin{cor}\label{cor:K2}
The functors (\ref{descend:PT}) induce the right action of 
the K-theoretic Hall-algebra
of zero dimensional sheaves to the
direct sum of PT categories 
for a fixed $\beta$
\begin{align*}
\bigoplus_{n \in \mathbb{Z}}
K(\mathcal{DT}^{\mathbb{C}^{\ast}}(P_n(X, \beta))) \times 
\bigoplus_{m\ge 0}
K(\mathcal{DT}^{\mathbb{C}^{\ast}}(\mM_X^{\sigma \sss}(m[\mathrm{pt}])))
\to \bigoplus_{n\in \mathbb{Z}}K(\mathcal{DT}^{\mathbb{C}^{\ast}}(P_n(X, \beta))). 
\end{align*}
\end{cor}

\section{Hecke correspondences of categorical PT theories}\label{sec:Hecke}
In this section, we 
consider the simple operators in Corollary~\ref{cor:K2}.
They are induced by the stacks of Hecke correspondences, and 
we describe them as two kinds of projectivizations over some 
quasi-smooth derived stacks. 
Using these descriptions, we construct 
the annihilator operators of the above simple operators. 
\subsection{Simple operators}
Here we consider the operator (\ref{descend:PT})
for $m=1$
\begin{align*}
\phi_P \colon \mathcal{DT}^{\mathbb{C}^{\ast}}(P_n(X, \beta))
\times \mathcal{DT}^{\mathbb{C}^{\ast}}(\mM_X^{\sigma \sss}([\mathrm{pt}]))
\to \mathcal{DT}^{\mathbb{C}^{\ast}}(P_{n+1}(X, \beta)). 
\end{align*}
Note that 
 $\fM_S([\pt])$ is the derived moduli 
stack of skyscraper sheaves of points of $S$, 
which is written as
\begin{align*}
\fM_S([\pt])=[\fS/\mathbb{C}^{\ast}].
\end{align*}
Here $\fS$ is a quasi-smooth 
derived scheme whose classical truncation is 
$S$, and $\mathbb{C}^{\ast}$ acts on it trivially. 
Therefore we have the decomposition
\begin{align}\label{decom:S}
 \mathcal{DT}^{\mathbb{C}^{\ast}}(\mM_X^{\sigma \sss}([\mathrm{pt}]))=
\Dbc(\fM_S([\mathrm{pt}])) =\bigoplus_{k \in \mathbb{Z}}\Dbc(\fS)_{k}
\end{align}
where $\Dbc(\fS)_{k}$ consists of objects with $\mathbb{C}^{\ast}$-weights $k$. 
We have the following description of the derived scheme $\fS$. 

\begin{lem}\label{lem:trivial}
We have $\fS=\Spec_S S(\omega_S[1])$. 
\end{lem}
\begin{proof}
Since $\Ext_S^2(\oO_x, \oO_x)=\omega_S^{\vee}|_{x}$
for each point $x \in S$, 
the closed immersion $S \hookrightarrow \fS$ is 
a square zero extension such that we have the 
distinguished triangle
\begin{align*}
\omega_S[1] \to \oO_{\fS} \to \oO_S. 
\end{align*}
Therefore the square zero extension $S \hookrightarrow \fS$
is classified by the map $\delta$ in the following 
distinguished triangle (see~\cite[Proposition~5.4.2]{MR3701353})
\begin{align*}
\omega_S[1] \to \mathbb{L}_{\fS}|_{S} \to \Omega_S \stackrel{\delta}{\to}
\omega_S[2]. 
\end{align*}
Since the structure sheaf of the diagonal 
$\Delta \subset S \times S$
is the universal family 
on $\fS \times S$ restricted to $S \times S$, 
we have the isomorphism
\begin{align}\label{isom:L}
\mathbb{L}_{\fS}|_{S} \cong (\tau_{\ge 0} p_{1\ast} \hH om_{S \times S}(\oO_{\Delta}, \oO_{\Delta})[1])^{\vee}. 
\end{align}
Here $p_1 \colon S \times S \to S$ is the first projection. 
By the Hochschild-Kostant-Rosenberg 
 isomorphism, we have 
\begin{align*}
p_{1\ast} \hH om_{S \times S}(\oO_{\Delta}, \oO_{\Delta})[1]
\cong \oO_S[1] \oplus T_S \oplus \wedge^2 T_S[-1]. 
\end{align*}
It follows that (\ref{isom:L}) is isomorphic to 
$\Omega_S \oplus \omega_S[1]$, and $\delta=0$ follows. 
Therefore $S \hookrightarrow \fS$ is a trivial square 
zero extension, and the lemma follows. 
\end{proof}

Note that have the diagram
\begin{align}\label{dia:S/C}
S \stackrel{\iota}{\hookrightarrow} \fS 
\stackrel{\vartheta}{\leftarrow} [\fS/\mathbb{C}^{\ast}]
\end{align}
where the left arrow is induced by the classical truncation
and the right arrow is induced by 
the natural morphism $B\mathbb{C}^{\ast} \to \Spec \mathbb{C}$. 
We define 
\begin{align*}
\lambda_k \cneq \vartheta^{\ast}\iota_{\ast}(-) \otimes \oO(k) \colon  \Dbc(S) \to \Dbc(\fS)_{k}
\end{align*}
where $\oO(k)$ is a line bundle on $[\fS/\mathbb{C}^{\ast}]$ determined by 
the one dimensional representation of $\mathbb{C}^{\ast}$ with weight $k$. 
\begin{defi}\label{def:muk}
For each $k \in \mathbb{Z}$ and $\eE \in \Dbc(S)$, we define 
the functor $\mu_{\eE, k}^+$ by 
\begin{align*}
\mu_{\eE, k}^+(-)  \cneq \phi_P(\{ (-), \lambda_k(\eE)\})  \colon 
\mathcal{DT}^{\mathbb{C}^{\ast}}(P_n(X, \beta))
\to \mathcal{DT}^{\mathbb{C}^{\ast}}(P_{n+1}(X, \beta)). 
\end{align*}
\end{defi}
Note that the functor $\lambda_k$ induces the isomorphism 
on the K-theory
\begin{align*}
\lambda_k \colon K(S) \stackrel{\cong}{\to} K(\fS)_k. 
\end{align*}
Therefore 
under the decomposition (\ref{decom:S})
and the above isomorphism, 
the induced map of $\phi_P$ on K-theory 
\begin{align*}
\phi_P \colon K(\mathcal{DT}^{\mathbb{C}^{\ast}}(P_n(X, \beta)))
\times K(\mathcal{DT}^{\mathbb{C}^{\ast}}(\mM_X^{\sigma \sss}([\mathrm{pt}])))
\to K(\mathcal{DT}^{\mathbb{C}^{\ast}}(P_{n+1}(X, \beta)))
\end{align*}
is written as, for $\eE_k \in K(S)$, 
\begin{align*}
\phi_P \left(\left\{ (-), \sum_{k \in \mathbb{Z}} \eE_k\right\} \right)
=\sum_{k \in \mathbb{Z}} \mu_{\eE_k, k}^+(-).
\end{align*}
In the rest of this section, we construct 
functors in the 
opposite direction
\begin{align}\label{mu:opposite}
\mu_{\eE, k}^-(-) \colon 
\mathcal{DT}^{\mathbb{C}^{\ast}}(P_{n+1}(X, \beta))
\to \mathcal{DT}^{\mathbb{C}^{\ast}}(P_{n}(X, \beta)) 
\end{align}
as an analogy of annihilation operators 
for Heinsenberg algebra action on 
homologies of Hilbert schemes of points~\cite{MR1386846, MR1441880}.

\subsection{The stack of Hecke correspondences}
We describe the operator $\mu_{\eE, k}^+$ in terms of the 
stack of Hecke correspondences, which we will define below. 
For $v \in N_{\le 1}(S)$, we denote by 
\begin{align*}
\fM_S(v)^{\pure} \subset \fM_S(v),  \
\fM_S^{\dag}(v)^{\pure} \subset \fM_S^{\dag}(v)
\end{align*}
the derived open substacks consisting of pure one dimensional 
sheaves $F$, pairs $(F, \xi)$ such that $F$ is pure one dimensional, 
respectively. 
In the notation of the diagram (\ref{dia:MSdag}), 
we also define 
\begin{align*}
\fM_S^{\ext, \dag}(v_{\bullet})^{\pure} 
\cneq (\ev_2^{\dag})^{-1}(\fM_S^{\dag}(v_2)^{\pure}) 
\subset \fM_S^{\ext, \dag}(v_{\bullet}).
\end{align*}
Since any non-zero subsheaf of a pure one dimensional sheaf is also pure
one dimensional, 
for $v_{\bullet} \in N_{\le 1}(S)^{\times 3}$
the diagram (\ref{dia:MSdag}) restricts to the diagram 
\begin{align*}
\xymatrix{
\fM_S^{\ext, \dag}(v_{\bullet})^{\pure} \ar[r]^{\ev_2^{\dag}} 
 \ar[d]_{(\ev_1^{\dag}, \ev_3^{\dag})} & \fM_S^{\dag}(v_2)^{\pure} \\
\fM_S^{\dag}(v_1)^{\pure} \times \fM_S(v_3). & 
}
\end{align*}
We take 
$v_{\bullet} \in N_{\le 1}(S)^{\times 3}$ of the form
\begin{align*}
v_1=(\beta, n), \ v_2=(\beta, n+1), \ 
v_3=(0, 1)=[\pt]. 
\end{align*}
Here $\beta \in \mathrm{NS}(S)$
 is a non-zero effective curve class on $S$
and $n \in \mathbb{Z}$. 
Then we have the following diagram
\begin{align}\label{dia:Hecke}
\xymatrix{
\fH ecke(v_{\bullet}) \ar@{}[dr]|\square
\ar@/^30pt/[rr]^-{\pi_2} \ar[r]^-{\tau} 
\ar@/_60pt/[dd]_-{(\pi_1, \pi_3)} \ar[d] 
& \fM_S^{\ext, \dag}(v_{\bullet})^{\pure} \ar[r]^-{\ev^{\dag}_2} 
\ar[d]_-{(\ev_1^{\dag}, \ev_3^{\dag})} & \fM_S^{\dag}(v_2)^{\pure} \\
\fM_S^{\dag}(v_1)^{\pure} \times [S/\mathbb{C}^{\ast}] \ar[d] \ar[r] 
\ar@{}[dr]|\square
& \fM_S^{\dag}(v_1)^{\pure} \times [\fS/\mathbb{C}^{\ast}] 
\ar[d]_-{\id \times \vartheta} & \\
\fM_S^{\dag}(v_1)^{\pure} \times S 
\ar[r]_-{\id \times \iota} & \fM_S^{\dag}(v_1)^{\pure} \times \fS
}
\end{align}
Here $\fH ecke(v_{\bullet})$
is defined by the top left Cartesian square. 
Let us take derived open substacks
$\fM_S^{\dag}(v_i)^{\fin} \subset \fM_S^{\dag}(v_i)^{\pure}$
for $i=1, 2$ satisfying 
the condition (\ref{PT:qcompact}) and 
\begin{align}\label{hecke:fin}
\fH ecke(v_{\bullet})^{\fin} \cneq 
\pi_2^{-1}(\fM_S^{\dag}(v_2)^{\fin}) \subset 
\pi_1^{-1}(\fM_S^{\dag}(v_1)^{\fin}). 
\end{align}
Then we have the diagram
\begin{align}\label{dia:Hecke1.5}
\xymatrix{
\fH ecke(v_{\bullet})^{\fin} \ar[r]^-{\pi_2} \ar[d]_-{(\pi_1, \pi_3)} & \fM_S^{\dag}(v_2)^{\fin} \\
\fM_S^{\dag}(v_1)^{\fin} \times S. 
}
\end{align}
Here the vertical arrow is quasi-smooth and the horizontal arrow is 
proper. 
Also let $\lL$ be the line bundle on $\fH ecke(v_{\bullet})$ 
defined by 
\begin{align}\label{def:lL}
\lL \cneq \ev_3^{\dag \ast} \oO(1)|_{\fH ecke(v_{\bullet})}
 \in \Pic(\fH ecke(v_{\bullet})). 
\end{align}
Then for each $\eE \in \Dbc(S)$, we 
have the functor
\begin{align}\label{funct:fin}
\pi_{2\ast}(\pi_1^{\ast}(-)\otimes \pi_3^{\ast}\eE \otimes \lL^k) \colon 
\Dbc(\fM_S^{\dag}(\beta, n)^{\fin}) \to \Dbc(\fM_S^{\dag}(\eta, n+1)^{\fin}). 
\end{align}

\begin{lem}\label{lem:muHecke}
The functor (\ref{funct:fin}) descends to the functor $\mu_{\eE, k}^+$ in 
Definition~\ref{def:muk}, i.e. the following diagram commutes
\begin{align}\label{descend:PT2}
\xymatrix{
\Dbc(\fM_S^{\dag}(\beta, n)^{\fin}) \ar[r] \ar[d] & 
\Dbc(\fM_S^{\dag}(\beta, n+1)^{\fin}) \ar[d] \\
\mathcal{DT}^{\mathbb{C}^{\ast}}(P_n(X, \beta)) \ar[r]^-{\mu_{\eE, k}^{+}} & 
\mathcal{DT}^{\mathbb{C}^{\ast}}(P_{n+1}(X, \beta)). 
}
\end{align}
Here the top horizontal arrow is given by (\ref{funct:fin}) and the vertical 
arrows are quotient functors. 
\end{lem}
\begin{proof}
The lemma follows by using base change with respect to the diagram (\ref{dia:Hecke}). 
Namely in the notation of the diagram (\ref{dia:Hecke}), we have 
\begin{align*}
\ev_{2\ast}^{\dag}(\ev_1^{\dag}, \ev_3^{\dag})^{\ast}(- \times \vartheta^{\ast} \iota_{\ast}\eE \otimes \oO(k)) &\cong
\ev_{2\ast}^{\dag}\{(\ev_1^{\dag}, \ev_3^{\dag})^{\ast}(\id \times \vartheta)^{\ast}(\id \times \iota)_{\ast}(- \boxtimes \eE) \otimes (\ev_3^{\dag})^{\ast}\oO(k) \} \\
&\cong \ev_{2\ast}^{\dag}\{ \tau_{\ast}(\pi_1, \pi_3)^{\ast}(-\boxtimes \eE) \otimes (\ev_3^{\dag})^{\ast}\oO(k)\} \\
&\cong \pi_{2\ast}\{(\pi_1, \pi_3)^{\ast}(-\boxtimes \eE) \otimes \lL^k\}. 
\end{align*}
The above isomorphisms immediately implies the commutativity of the 
diagram (\ref{descend:PT2}). 
\end{proof}

\subsection{The descriptions of the derived stack $\fH ecke(v_{\bullet})$}
In this subsection, we give two descriptions of the derived
stack $\fH ecke(v_{\bullet})$
as projectivizations of perfect objects over quasi-smooth 
derived stacks. 
These descriptions immediately show that both of 
the morphisms $\pi_1$, $\pi_2$ are quasi-smooth and proper, 
and these facts will be required to construct the functor (\ref{mu:opposite}). 
The descriptions here will be also used to compute the commutators in 
the next section. 
We first prove some lemmas. 

\begin{lem}\label{lem:diaHecke}
The stack $\fH ecke(v_{\bullet})$
parametrizes diagrams
\begin{align}\label{dia:Hecke:seq}
\xymatrix{
& \oO_S \ar@{=}[r] \ar[d]_-{\xi_1} & \oO_S \ar[d]_-{\xi_2} & & \\
0 \ar[r] & F_1 \ar[r] & F_2 \ar[r]^-{j} & \oO_x \ar[r] & 0. 
}
\end{align} 
Here the bottom sequence is a non-split exact sequence 
such that $F_1$, $F_2$ are pure one dimensional coherent 
sheaves on $S$ with $[F_i]=v_i$,
and $x \in S$. 
The isomorphisms of the diagrams (\ref{dia:Hecke:seq}) are termwise 
isomorphisms 
which are identities on $\oO_S$ and $\oO_x$. 
\end{lem}
\begin{proof}
From the definition of $\fH ecke(v_{\bullet})$, 
it is enough to show that 
for a diagram (\ref{dia:Hecke:seq}), the sheaf $F_2$ is 
pure if and only if $F_1$ is pure and the bottom sequence is 
non-split. 
The only if direction is obvious. 
Suppose that $F_1$ is pure and the bottom sequence is non-split. 
If $F_2$ is not pure, 
then there is a point $y \in S$ and a non-zero map 
$i \colon \oO_y \to F_2$. 
As $F_1$ is pure, we must have $y=x$.
Moreover the composition
\begin{align*}
\oO_x \stackrel{i}{\to}
 F_2 \stackrel{j}{\twoheadrightarrow} \oO_x
\end{align*}
must be non-zero, as otherwise there is a non-zero map 
$\oO_x \to F_1$ which contradicts to the purity of $F_1$. 
Therefore the bottom sequence of (\ref{dia:Hecke:seq}) splits, so 
a contradiction. 
\end{proof}
Let $\mathbb{D}_S$ be the dualizing functor
\begin{align*}
\mathbb{D}_S=\hH om_{S}(-, \oO_S) \colon 
\Dbc(S)^{\mathrm{op}} \to \Dbc(S).
\end{align*}
Then $\mathbb{D}_S(\Coh(S))$
is another heart of a t-structure on $\Dbc(S)$. 
By the lemma below, 
we can also regard 
$\hH ecke(v_{\bullet})$
as the stack of short exact sequences in $\mathbb{D}_S(\Coh(S))[1]$.

\begin{lem}\label{lem:hecke:seq}
Giving a diagram (\ref{dia:Hecke:seq}) is equivalent to giving a non-split
exact sequence in 
$\mathbb{D}_S(\Coh(S))[1]$
of the form
\begin{align}\label{exact:D}
0 \to \oO_x[-1] \to I_1^{\bullet}[1] \to I_2^{\bullet}[1] \to 0. 
\end{align}
Here $I_i^{\bullet}=(\oO_S \stackrel{\xi_i}{\to} F_i) \in \Dbc(S)$. 
\end{lem}
\begin{proof}
By taking the cone and shift of the diagram (\ref{dia:Hecke:seq}),
we obtain the distinguished triangle 
$\oO_x[-1] \to I_1^{\bullet}[1] \to I_2^{\bullet}[1]$. 
Since we have 
\begin{align*}
\mathbb{D}_S(\oO_x)=\oO_x[-2], \ 
\mathbb{D}_S(F)=\eE xt_S^1(F, \oO_S)[-1], \ 
\mathbb{D}_S(\oO_S)=\oO_S
\end{align*}
where $F$ is a pure one dimensional sheaf, 
each term of 
the sequence (\ref{exact:D}) is an object in 
$\mathbb{D}_S(\Coh(S))[1]$.
Therefore (\ref{exact:D}) is an exact sequence in 
$\mathbb{D}_S(\Coh(S))[1]$. 
Since $F_1$ is pure, the sequence (\ref{exact:D}) is non-split. 

Conversely suppose that a non-split 
exact sequence (\ref{exact:D}) is given. 
We have the distinguished triangle
\begin{align}\label{seq:I3}
\RHom_S(I_2^{\bullet}[1], \oO_x) \to \RHom_S(F_2, \oO_x)
\to \RHom_S(\oO_S, \oO_x)=\mathbb{C}. 
\end{align}
By taking the associated long exact sequence of cohomologies, 
we see that 
the extension class $I_2^{\bullet}[1]
 \to \oO_x$ of the sequence (\ref{exact:D})
is represented by the commutative diagram
\begin{align*}
\xymatrix{
\oO_S \ar[r] \ar[d]_-{\xi_2} & 0 \ar[d] \\
F_2 \ar[r] & \oO_x. 
}
\end{align*}
As the map $I_2^{\bullet}[1] \to \oO_x$ is non-zero, the bottom 
arrow is surjection. By taking the termwise kernels, 
we obtain the diagram (\ref{dia:Hecke:seq}). 
 \end{proof}

For $i=1, 2$, we denote by
\begin{align*}
\fI^{\bullet}(v_i)=(\oO_{\fM_S^{\dag}(v_{i})^{\pure} \times S} \to 
\ffF(v_i)) \in \mathrm{Perf}(\fM_S^{\dag}(v_i)^{\pure} \times S)
\end{align*}
the object associated with the universal pair.  
We first describe the morphism $(\pi_1, \pi_3)$ in the diagram 
(\ref{dia:Hecke}). 

\begin{lem}\label{fiber:pi}
We have an equivalence of derived stacks
\begin{align}\label{equiv:Hecke}
\fH ecke(v_{\bullet}) \stackrel{\sim}{\to}
\mathbb{P}_{\fM_S^{\dag}(v_1)^{\pure} \times S}(\ffF(v_1)^{\vee} \boxtimes \omega_S[1])
\end{align}
such that the projection of the right hand side 
to $\fM_S^{\dag}(v_1)^{\pure} \times S$
is identified with $(\pi_1, \pi_3)$. 
In particular $(\pi_1, \pi_3)$ and $\pi_1$
are quasi-smooth and proper.  
\end{lem}
\begin{proof}
We consider the following diagram
\begin{align*}
\xymatrix{
\fM_S^{\dag}(v_1)^{\pure} \times S \ar@<-0.3ex>@{^{(}->}[r]^-{(\id, \Delta)}
 \ar[d] 
& \fM_S^{\dag}(v_1)^{\pure} \times S \times S \ar[r]^-{p_{12}} 
\ar[d]_-{p_{23}} 
\ar[rd]_-{p_{13}} &
\fM_S^{\dag}(v_1)^{\pure} \times S \\
S \ar@<-0.3ex>@{^{(}->}[r]^{\Delta} & S \times S & \fM_S^{\dag}(v_1)^{\pure} \times S. 
}
\end{align*}
Here $\Delta$ is the diagonal and 
$p_{ij}$ are projections onto the corresponding factors. 
By the description (\ref{MSext:spec}) and Lemma~\ref{lem:diaHecke}, 
we see that 
the derived stack $\fH ecke(v_{\bullet})$ together with 
$(\pi_1, \pi_3)$ are identified with 
the following
\begin{align}\label{stack:VM}
\left[\left(\mathbb{V}(
\hH om_{p_{13}}(p_{23}^{\ast}\oO_{\Delta}, p_{12}^{\ast}\ffF(v_1)[1])^{\vee})
\setminus 0_{\fM_S^{\dag}(v_1)^{\pure} \times S}\right)
/\mathbb{C}^{\ast}\right]
\to \fM_S^{\dag}(v_1)^{\pure} \times S. 
\end{align}
Here $\mathbb{C}^{\ast}$ acts on $\oO_{\Delta}$ with weight one. 
We have the isomorphisms
\begin{align}\notag
\hH om_{p_{13}}(p_{23}^{\ast}\oO_{\Delta}, p_{12}^{\ast}\ffF(v_1)[1])^{\vee}
&\cong \hH om_{p_{13}}(p_{12}^{\ast}\ffF(v_1), (\id, \Delta)_{\ast}(\oO\boxtimes \omega_S[1])) \\
&\cong  \notag
p_{13\ast}(\id, \Delta)_{\ast}\{((\id, \Delta)^{\ast}p_{12}^{\ast}
\ffF(v_1)^{\vee}) \boxtimes \omega_S[1] \} \\
\notag
&\cong \ffF(v_1)^{\vee} \boxtimes \omega_S[1]. 
\end{align}
Here the first isomorphism follows from the 
Grothendieck duality. 
Therefore we have the equivalence (\ref{equiv:Hecke}).

Let us take a point 
$(F_1, \xi_1)$
of $\fM_S^{\dag}(v_1)^{\pure}$ 
and a point $x \in S$.
By the equivalence (\ref{equiv:Hecke}), the fiber of 
$(\pi_1, \pi_3)$ at the point
$((F_1, \xi_1), x))$ is 
given by 
\begin{align*}
(\pi_1, \pi_3)^{-1}((F_1, \xi_1), x))
=\mathbb{P}(F_1^{\vee}|_{x} \otimes \omega_S[1]).
\end{align*}
Since $F_1$ is pure one dimensional, 
$F_1^{\vee}|_{x} \otimes \omega_S[1]$
has cohomological amplitude $[-1, 0]$. 
Therefore the morphism $(\pi_1, \pi_3)$ is 
quasi-smooth and proper. 
Then $\pi_1$ is also quasi-smooth and proper, 
as $S$ is smooth projective. 
\end{proof}

By replacing $\pi_1$ with $\pi_2$ in the diagram (\ref{dia:Hecke1.5}), 
we also obtain the diagram
\begin{align}\label{dia:Hecke2}
\xymatrix{
\fH ecke(v_{\bullet}) \ar@{}[dr]|\square
\ar@/^30pt/[rr]^-{\pi_1} \ar[r]^-{\tau} 
\ar@/_60pt/[dd]_-{(\pi_2, \pi_3)} \ar[d] 
& \fM_S^{\ext, \dag}(v_{\bullet})^{\pure} \ar[r]^-{\ev_1^{\dag}} 
\ar[d]_-{(\ev_2^{\dag}, \ev_3^{\dag})} & \fM_S^{\dag}(v_1)^{\pure} \\
\fM_S^{\dag}(v_2)^{\pure} \times [S/\mathbb{C}^{\ast}] \ar[d] \ar[r] 
\ar@{}[dr]|\square
& \fM_S^{\dag}(v_2)^{\pure} \times [\fS/\mathbb{C}^{\ast}] 
\ar[d]_-{\id \times \vartheta} & \\
\fM_S^{\dag}(v_2)^{\pure} \times S 
\ar[r]_-{\id \times \iota} & \fM_S^{\dag}(v_2)^{\pure} \times \fS
}
\end{align}
By Lemma~\ref{fiber:pi}, $\pi_1$ is quasi-smooth and proper. 
We also investigate the morphism 
$(\pi_2, \pi_3)$. 

\begin{lem}\label{lem:proper}
We have an equivalence of derived stacks
\begin{align}\label{equiv:Hecke2}
\fH ecke(v_{\bullet}) \stackrel{\sim}{\to}
\mathbb{P}_{\fM_S^{\dag}(v_2)^{\pure} \times S}
(\fI^{\bullet}(v_2)[1])
\end{align}
such that the projection of the right hand side 
to $\fM_S^{\dag}(v_2)^{\pure} \times S$
is identified with $(\pi_2, \pi_3)$. 
In particular $(\pi_2, \pi_3)$ and $\pi_3$
are quasi-smooth and proper. 
\end{lem}
\begin{proof} 
We consider the following diagram
\begin{align*}
\xymatrix{
\fM_S^{\dag}(v_2)^{\pure} \times S \ar@<-0.3ex>@{^{(}->}[r]^-{(\id, \Delta)}
 \ar[d] 
& \fM_S^{\dag}(v_2)^{\pure} \times S \times S \ar[r]^-{q_{12}} 
\ar[d]_-{q_{23}} 
\ar[rd]_-{q_{13}} &
\fM_S^{\dag}(v_2)^{\pure} \times S \\
S \ar@<-0.3ex>@{^{(}->}[r]^{\Delta} & S \times S & \fM_S^{\dag}(v_2)^{\pure} \times S. 
}
\end{align*}
Here 
$q_{ij}$ are projections onto the corresponding factors. 
By Lemma~\ref{lem:hecke:seq}, 
a similar argument of (\ref{stack:VM}) shows that 
the derived stack $\fH ecke(v_{\bullet})$ together with 
$(\pi_2, \pi_3)$ are identified with 
the following
\begin{align*}
\left[\left(\mathbb{V}(\hH om_{q_{13}}(q_{12}^{\ast}\fI^{\bullet}(v_2)[1], 
q_{23}^{\ast}\oO_{\Delta})^{\vee})
\setminus 0_{\fM_S^{\dag}(v_2)^{\pure} \times S}\right)
/\mathbb{C}^{\ast}\right]
\to \fM_S^{\dag}(v_2)^{\pure} \times S. 
\end{align*}
Here $\mathbb{C}^{\ast}$ acts on $\oO_{\Delta}$ with weight one. 
We have the isomorphisms
\begin{align*}
\hH om_{q_{13}}(q_{12}^{\ast}\fI^{\bullet}(v_2)[1], 
q_{23}^{\ast}\oO_{\Delta})^{\vee}
&\cong 
\hH om_{q_{13}}(q_{12}^{\ast}\fI^{\bullet}(v_2)[1], 
(\id, \Delta)_{\ast}\oO_{\fM_S^{\dag}(v_2)^{\pure} \times S})^{\vee} \\
&\cong \hH om(\fI^{\bullet}(v_2)[1], \oO_{\fM_S^{\dag}(v_2)^{\pure} \times S})^{\vee} \\
&\cong \fI^{\bullet}(v_2)[1]. 
\end{align*}
Therefore we have the equivalence (\ref{equiv:Hecke2}). 

Let us take a point 
$(F_2, \xi_2)$
of $\fM_S^{\dag}(v_2)^{\pure}$ 
and a point $x \in S$.
By the equivalence (\ref{equiv:Hecke2}), the fiber of 
$(\pi_2, \pi_3)$ at the point
$((F_2, \xi_2), x))$ is 
given by 
\begin{align*}
(\pi_2, \pi_3)^{-1}((F_2, \xi_2), x))
=\mathbb{P}(I_2^{\bullet}[1]|_{x}).
\end{align*}
We have the distinguished triangle
\begin{align*}
F_2|_{x} \to I_2^{\bullet}[1]|_{x} \to \oO_x[1]. 
\end{align*}
Since $F_2$ is pure one dimensional, 
$F_2|_{x}$ has cohomological amplitude $[-1, 0]$.
Therefore $I_2^{\bullet}[1]|_{x}$
has also cohomological amplitude $[-1, 0]$, hence 
$(\pi_2, \pi_3)$ is 
quasi-smooth and proper. 
Then $\pi_2$ is also quasi-smooth and proper, 
as $S$ is smooth projective. 
\end{proof}

\subsection{The stack of pure objects in $\aA_X$}
Here 
we give a refinement of Lemma~\ref{lem:PT}
which is required to construct the functors (\ref{mu:opposite}). 
For $v \in N_{\le 1}(S)$, we 
define the following open substack of 
$\mM_X^{\dag}(v)$
\begin{align*}
\mM_X^{\dag}(v)^{\pure} \cneq 
t_0(\Omega_{\fM_S^{\dag}(v)^{\pure}}(v)[-1]) \subset \mM_X^{\dag}(v).
\end{align*}
Here we have used the isomorphism (\ref{isom:dag}). 
Note that through the equivalence 
$\aA_X^{\le 1} \stackrel{\sim}{\to} \bB_S^{\le 1}$ in Theorem~\ref{thm:D026}, 
the stack 
 $\mM_X^{\dag}(v)^{\pure}$ parametrizes diagrams (\ref{dia:BS2})
such that $\vV=\oO_S$ and $F$ is a pure one dimensional sheaf with $[F]=v$. 
We have the following lemma. 

\begin{lem}\label{lem:evP}
For an exact sequence in $\aA_X$ of the form
\begin{align*}
0 \to E_1 \to E_2 \to \oO_x[-1] \to 0
\end{align*}
suppose that $E_i$ 
corresponds to points of $\mM_X^{\dag}(v_i)^{\pure}$. 
Then $E_1$ is a PT stable pair 
if and only if $E_2$ is a PT stable pair. 
\end{lem}
\begin{proof}
The if direction is proved in Lemma~\ref{lem:PT}, 
so we prove the only if direction. 
Suppose that $E_1=(\oO_{\overline{X}} \stackrel{s_1}{\to} F_1)$ for a PT stable pair 
$(F_1, s_1)$ on $X$. 
Then we have the following diagram
\begin{align*}
\xymatrix{
  & F_1[-1] \ar[d] &   \\
\oO_x[-2] \ar[r] \ar[rd]_-{0} & E_1 \ar[r] \ar[d] & E_2 \\
& \oO_{\overline{X}}.   & 
}
\end{align*}
Since $\Hom(\oO_x[-2], \oO_{\overline{X}})=0$, 
the map $E_1 \to \oO_{\overline{X}}$ factors through 
$E_1 \to E_2 \to \oO_{\overline{X}}$. 
By taking the cones, we obtain distinguished triangles
\begin{align*}
F_2[-1] \to E_2 \to \oO_{\overline{X}}, \ 
\oO_x[-2] \to F_1[-1] \to F_2[-1] 
\end{align*}
for some $F_2$. 
By the second sequence, 
the object $F_2$ is a one dimensional sheaf. 
Then by the first sequence, 
$E_2=(\oO_{\overline{X}} \stackrel{s_2}{\to} F_2)$
for a pair $(F_2, s_2)$. 
As $E_2$ is a point of 
$\mM_X^{\dag}(v_2)^{\pure}$, the sheaf $F_2$ 
is a pure one dimensional sheaf. 
Moreover $(F_2, s_2)$ is isomorphic to $(F_1, s_1)$
outside $x$, so $s_2$ is also surjective in 
dimension one. 
Therefore $(F_2, s_2)$ is a PT stable pair. 
\end{proof}

The above lemma is rephrased in terms of stack of exact sequences in 
$\aA_X$ in the following way. 
We take $v_{\bullet} \in N_{\le 1}(S)^{\times 3}$ and, 
using the notation of the diagram (\ref{dia:Xext:dag}), 
we set 
\begin{align*}
\mM_X^{\ext, \dag}(v_{\bullet})^{\pure} \cneq 
(\ev_2^{X, \dag})^{-1}(\mM_X^{\dag}(v_2)^{\pure})
\subset \mM_X^{\ext, \dag}(v_{\bullet}).
\end{align*}
Since any subsheaf of a pure one dimensional sheaf is pure one dimensional, 
the diagram (\ref{dia:Xext:dag}) restricts to the diagram 
\begin{align}\label{dia:Xext:dag2}
\xymatrix{
\mM_X^{\ext, \dag}(v_{\bullet})^{\pure} 
\ar[d]_-{(\ev_1^{X, \dag}, \ev_3^{X, \dag})} \ar[r]^-{\ev_2^{X, \dag}} 
& \mM_X^{\dag}(v_2)^{\pure} \\
\mM_X^{\dag}(v_1)^{\pure} \times \mM_X(v_3). &
}
\end{align}
We also define the following conical closed substack of 
$\mM_X^{\dag}(v)^{\pure}$
\begin{align*}
\zZ_{P\us}(v)^{\pure} \cneq \zZ_{P\us}(v) \cap \mM_X^{\dag}(v)^{\pure}
\subset \mM_X^{\dag}(v)^{\pure}.
\end{align*}
Then for $v_3=[\pt]$ in the diagram (\ref{dia:Xext:dag2}), 
the result of Lemma~\ref{lem:evP}
implies the following identity
\begin{align}\label{evX:id}
(\ev_1^{X, \dag}, \ev_3^{X, \dag})^{-1}(\zZ_{P\us}(v_1)^{\pure} \times \mM_X(v_3)) =
(\ev_2^{X, \dag})^{-1}(\zZ_{P\us}(v_2)^{\pure}). 
\end{align}

\subsection{The annihilator functors}

Finally in this section, we construct the 
annihilator functors (\ref{mu:opposite}). 
Recall that in (\ref{hecke:fin}), 
we took derived open substacks 
$\fM_S^{\dag}(v_i)^{\fin} \subset \fM_S^{\dag}(v_i)^{\pure}$
of finite type. 
Here we take another derived open substack 
$\fM_S^{\dag}(v_2)^{\fin'} \subset \fM_S^{\dag}(v_2)^{\pure}$
of finite type, 
which contains $\fM_S^{\dag}(v_2)^{\fin}$ and satisfies 
\begin{align*}
\fH ecke(v_{\bullet})^{\fin'} \cneq 
\pi_1^{-1}(\fM_S(v_1)^{\fin}) \subset \pi_2^{-1}(\fM_S(v_2)^{\fin'}). 
\end{align*}
Then the diagram (\ref{dia:Hecke2}) restricts to the diagram
\begin{align*}
\xymatrix{
\fH ecke(v_{\bullet})^{\fin'} \ar[r]^-{\pi_1} \ar[d]_-{(\pi_2, \pi_3)} & \fM_S^{\dag}(v_1)^{\fin} \\
\fM_S^{\dag}(v_2)^{\fin'} \times S. 
}
\end{align*}
By Lemma~\ref{fiber:pi} and Lemma~\ref{lem:proper},
the vertical arrow is quasi-smooth and the horizontal arrow is proper. 
Therefore for each 
$\eE \in \Dbc(S)$ and $k \in \mathbb{Z}$, 
we have the functor
\begin{align}\label{funct:-}
\pi_{1!}(\pi_2^{\ast}(-) \otimes \pi_3^{\ast} \eE \otimes \lL^k) \colon 
\Dbc(\fM_S^{\dag}(\beta, n+1)^{\fin'}) \to 
\Dbc(\fM_S^{\dag}(\beta, n)^{\fin}). 
\end{align}
\begin{lem}\label{lem:descend2}
The functor (\ref{funct:-})
sends 
$\cC_{\zZ_{P\us}(\beta, n+1)^{\fin'}}$ to $\cC_{\zZ_{P\us}(\beta, n)^{\fin}}$.
\end{lem}
\begin{proof}
By Lemma~\ref{fiber:pi}, we can compute 
$\omega_{\pi_1}[\vdim \pi_1]$ as
\begin{align*}
\omega_{\pi_1}[\vdim \pi_1]=
(\pi_1, \pi_3)^{\ast} \det \ffF(v_1) \otimes \pi_3^{\ast}\omega_S[1]. 
\end{align*}
Since $\pi_{1!}=\pi_{1\ast}(-\otimes \omega_{\pi_1}[\vdim \pi_1])$, 
the functor (\ref{funct:-}) is written as 
\begin{align*}
\pi_{1\ast}\{ \pi_2^{\ast}(-) \otimes \pi_3^{\ast}(\eE \otimes \omega_S) \otimes 
(\pi_1, \pi_3)^{\ast} \det \ffF(v_1)[1] \otimes \lL^k\}. 
\end{align*}
By Lemma~\ref{lem:trivial}, 
the morphism $\tau$ in the diagram (\ref{dia:Hecke}) is a trivial 
square zero extension. 
Therefore 
there is $\pP \in \mathrm{Perf}(\fM_S^{\ext, \dag}(v_{\bullet})^{\pure})$
whose restriction to $\fH ecke(v_{\bullet})$ is equivalent to 
$(\pi_1, \pi_3)^{\ast} \det \ffF(v_1)[1]$. 
Then by
the similar argument of Lemma~\ref{lem:muHecke}
 using base-change with respect to the diagram (\ref{dia:Hecke2}), 
the functor (\ref{funct:-})
is written as
\begin{align*}
\ev_{1\ast}^{\dag}\{(\ev_2^{\dag}, \ev_3^{\dag})^{\ast}( (-) \boxtimes \lambda_k(\eE \otimes \omega_S)) \otimes \pP \}. 
\end{align*} 
By Lemma~\ref{lem:evP}
and the argument of Corollary~\ref{cor:induce2}, 
the above functor sends  
$\cC_{\zZ_{P\us}(\beta, n+1)^{\fin'}}$ to $\cC_{\zZ_{P\us}(\beta, n)^{\fin}}$.
\end{proof}

\begin{defi}\label{defi:mu-}
We define the functor 
\begin{align}\label{mu:ann}
\mu_{\eE, k}^{-} \colon 
\mathcal{DT}^{\mathbb{C}^{\ast}}(P_{n+1}(X, \beta)) 
\to \mathcal{DT}^{\mathbb{C}^{\ast}}(P_{n}(X, \beta))
\end{align}
as a descendent of the functor (\ref{funct:-}), which 
exists uniquely by Lemma~\ref{lem:descend2}. 
\end{defi}

\subsection{Restrictions to perfect PT subcategories}\label{subsec:perfPT}
Here we introduce the notion of 
\textit{perfect PT categories}, 
and show that the operators 
$\mu_{\eE, k}^{\pm}$
restrict to perfect PT categories. 
Let us take a derived open substack 
$\fM_S^{\dag}(\beta, n)^{\fin} \subset \fM_S^{\dag}(\beta, n)^{\pure}$
as before. 
We give the following definition. 
\begin{defi}\label{def:PTperf}
We define the 
perfect PT category to be
\begin{align*}
\mathcal{DT}^{\mathbb{C}^{\ast}}_{\perf}(P_n(X, \beta))
\cneq \mathrm{Perf}(\fM_S^{\dag}(\beta, n)^{\fin})/\mathrm{Perf}(\fM_S^{\dag}(\beta, n)^{\fin}) \cap 
\cC_{\zZ_{P\us}(\beta, n)^{\fin}}. 
\end{align*} 
\end{defi}
We have the composition of functors
\begin{align}\label{com:perf}
\mathrm{Perf}(\fM_S^{\dag}(\beta, n)^{\fin})
\hookrightarrow \Dbc(\fM_S^{\dag}(\beta, n)^{\fin})
\twoheadrightarrow 
\mathcal{DT}^{\mathbb{C}^{\ast}}(P_n(X, \beta))
\end{align}
whose kernel is exactly 
$\mathrm{Perf}(\fM_S^{\dag}(\beta, n)^{\fin}) \cap 
\cC_{\zZ_{P\us}(\beta, n)^{\fin}}$. 
Therefore we have the canonical functor
\begin{align}\label{perfPT}
\mathcal{DT}^{\mathbb{C}^{\ast}}_{\perf}(P_n(X, \beta)) 
\to \mathcal{DT}^{\mathbb{C}^{\ast}}(P_n(X, \beta)).
\end{align}
The perfect PT categories 
are closely related to the moduli spaces of 
stable pairs on $S$, rather than those on 3-folds $X$
in the following way. 
Let
\begin{align}\label{PT:derived}
\fP_n(S, \beta) \subset \fM_S^{\dag}(\beta, n)^{\fin}
\end{align}
be the derived open substack of 
PT stable pairs on $S$, i.e. 
this is the substack of pairs $(F, \xi)$
such that $\xi$ is surjective in dimension one. 
\begin{lem}\label{lem:PTS}
We have an equivalence
\begin{align}\label{equiv:perf}
\mathcal{DT}^{\mathbb{C}^{\ast}}_{\perf}(P_n(X, \beta))
\stackrel{\sim}{\to}
\mathrm{Perf}(\fP_n(S, \beta)). 
\end{align}
\end{lem}
\begin{proof}
Note that we have
\begin{align*}
\mathrm{Perf}(\fM_S^{\dag}(\beta, n)^{\fin})=
\cC_{0_{\mM_S^{\dag}(\beta, n)^{\fin}}}
\subset \Dbc(\fM_S^{\dag}(\beta, n)^{\fin})
\end{align*}
where $0_{\mM_S^{\dag}(\beta, n)^{\fin}}$
is the zero section of 
$\pi_{\ast} \colon \mM_X^{\dag}(\beta, n) \to 
\mM_S^{\dag}(\beta, n)$
over $\mM_S^{\dag}(\beta, n)^{\fin}$
(see~\cite[Theorem~4.2.6]{MR3300415}). 
It follows that we have
\begin{align*}
&\mathrm{Perf}(\fM_S^{\dag}(\beta, n)^{\fin})
\cap \cC_{\zZ_{P\us}(\beta, n)^{\fin}} \\
&=\{ \eE \in \mathrm{Perf}(\fM_S^{\dag}(\beta, n)^{\fin}) :
\mathrm{Supp}(\eE) \subset \mM_S^{\dag}(\beta, n)^{\fin} \setminus P_n(S, \beta)\}. 
\end{align*}
Therefore the restriction functor with respect to the 
open immersion (\ref{PT:derived})
gives an equivalence (\ref{equiv:perf}). 
\end{proof}

As we see below, the 
operators $\mu_{\eE, k}^{\pm}$ restrict to the 
perfect PT categories. 
\begin{lem}\label{mu:perf}
The functors (\ref{funct:fin}), (\ref{funct:-}) restrict 
to functors
\begin{align*}
\nu_{\eE, k}^{\pm} \colon
\mathcal{DT}^{\mathbb{C}^{\ast}}_{\perf}(P_n(X, \beta))
\to  \mathcal{DT}^{\mathbb{C}^{\ast}}_{\perf}(P_{n\pm 1}(X, \beta))
\end{align*}
so that we have the commutative diagram
\begin{align}\label{commute:perf}
\xymatrix{
\mathcal{DT}^{\mathbb{C}^{\ast}}_{\perf}(P_n(X, \beta))
\ar[r] \ar[d]_-{\nu_{\eE, k}^{\pm}} &
\mathcal{DT}^{\mathbb{C}^{\ast}}(P_n(X, \beta)) \ar[d]^-{\mu_{\eE, k}^{\pm}} \\
\mathcal{DT}^{\mathbb{C}^{\ast}}_{\perf}(P_{n\pm 1}(X, \beta))
\ar[r] & 
\mathcal{DT}^{\mathbb{C}^{\ast}}(P_{n\pm 1}(X, \beta)).
}
\end{align}
Here the horizontal arrows are given by (\ref{perfPT}). 
\end{lem}
\begin{proof}
Since $\pi_1$, $\pi_2$ are quasi-smooth and proper, 
both of $\pi_{i\ast}$, $\pi_{i!}$ preserve
perfect objects (see~\cite[Lemma~2.2]{Tproper}). 
Therefore the lemma follows  by the 
definitions of the functors (\ref{funct:fin}), (\ref{funct:-}). 
The commutative diagram (\ref{commute:perf}) is obvious 
by the construction. 
\end{proof}

\section{Commutator relations in K-theory}
In this section, we compute the commutator relation 
of the operators $\mu_{\eE, k}^+$, $\mu_{\eE, k}^-$
for elements of K-groups of PT categories, coming from
perfect PT categories. 
Our strategy is to use residue arguments 
developed by Negut~\cite{Negut}. 
We will see that the commutator relation for $k=0$ gives an 
analogue of Weyl algebra action for homologies of 
Hilbert schemes of points on locally planar curves~\cite{MR3807309}. 
\subsection{Some notation in K-theory}
Here we prepare some notation in K-theory 
following~\cite{Negut}. 
For a derived stack $\fM$, we set 
\begin{align*}
K(\fM) =K(\Dbc(\fM)), \ 
K_{\mathrm{perf}}(\fM) =K(\mathrm{Perf}(\fM)). 
\end{align*}
Note that 
we have maps given by 
tensor products
\begin{align*}
\otimes \colon 
K_{\perf}(\fM) \times
K_{\perf}(\fM) \to K_{\perf}(\fM), \ 
\otimes \colon K_{\perf}(\fM) \times K(\fM) \to K(\fM)
\end{align*}
which make 
$K_{\mathrm{perf}}(\fM)$ 
a commutative ring and 
$K(\fM)$ a module over it. 
Suppose that $\fM$ is quasi-smooth 
and let $i \colon \mM \hookrightarrow \fM$
be the natural closed immersion for $\mM=t_0(\fM)$. 
Then the following induced map is an isomorphism
\begin{align}\label{isom:K}
i_{\ast} \colon K(\mM) \stackrel{\cong}{\to} K(\fM). 
\end{align}

 For a vector bundle $\pP \to \fM$, 
we set
\begin{align*}
\wedge^{\bullet}(\pP x) \cneq 
\sum_{i\ge 0} \wedge^i \pP \cdot (-x)^i \in K_{\perf}(\fM)[x]. 
\end{align*}
Also for an element $\pP=[\pP_0]-[\pP_1] \in K_{\mathrm{perf}}(\fM)$
for vector bundles $\pP_0$, $\pP_1$, we 
define
\begin{align*}
\wedge^{\bullet}(\pP x) \cneq 
\frac{\wedge^{\bullet}(\pP_0 x)}{\wedge^{\bullet}(\pP_1 x)} \in 
K_{\perf}(\fM)(x), \ 
\wedge^{\bullet} \left(\frac{x}{\pP} \right) \cneq 
\wedge^{\bullet}(\pP^{\vee} x) \in K_{\perf}(\fM)(x). 
\end{align*}
These are rational functions in $x$. 
For a rational function $f(x) \in K_{\perf}(\fM)(x)$, we 
denote by 
\begin{align*}
f(x)|_{x=\infty} \in K_{\perf}(\fM) \lgakko 1/x \rgakko, \ 
f(x)|_{x=0} \in K_{\perf}(\fM) \lgakko x \rgakko
\end{align*}
the expansions of the rational 
function $f(x)$
at $x=\infty$, $x=0$, respectively. 
For example if $\pP$ is a rank $r$ vector bundle, 
we have
\begin{align}\label{expand:P}
\frac{1}{\wedge^{\bullet}(\pP x)}\Big|_{x=\infty}
=(-1)^r \det \pP^{\vee} x^{-r} \mathrm{Sym}^{\bullet}(\pP^{\vee} x^{-1}), \ 
\frac{1}{\wedge^{\bullet}(\pP x)}\Big|_{x=0}
=\mathrm{Sym}^{\bullet}(\pP x). 
\end{align}
We define
\begin{align*}
f(x)|_{x=\infty -0} 
\cneq f(x)|_{x=\infty}-
f(x)|_{x=0} \in K(\fM)\{x\}. 
\end{align*}
We will use the following calculation.
\begin{lem}\label{lem:calculation}
For $\pP=[\pP_0]-[\pP_1]$ as above, 
suppose that $\rank(\pP)=0$. 
Then for any line bundle $\lL$ on $\fM$, 
we have 
\begin{align}\label{expand:P2}
\wedge^{\bullet}((\lL-1)\pP x))|_{x=\infty -0}=
O(x^{-2})+(1-\lL^{\vee})\pP^{\vee} x^{-1}+
(\lL-1)\pP x+O(x^2). 
\end{align} 
\end{lem}
\begin{proof}
By definition, we have the identity  
\begin{align*}
\wedge^{\bullet}((\lL-1)\pP x))
=\frac{\wedge^{\bullet}\lL \pP_0 x}{\wedge^{\bullet}\lL \pP_1 x} 
\cdot \frac{\wedge^{\bullet}\pP_1 x}{\wedge^{\bullet}\pP_0 x}. 
\end{align*}
Using the identities (\ref{expand:P})
and the assumption that $\rank(\pP)=0$, 
we have 
\begin{align*}
\wedge^{\bullet}((\lL-1)\pP x))|_{x=\infty -0}
=\wedge^{\bullet}(\lL^{\vee} \pP_0^{\vee}x^{-1})
&\wedge^{\bullet}(\pP_1^{\vee}x^{-1})
\mathrm{Sym}^{\bullet}(\lL^{\vee} \pP_1^{\vee}x^{-1})
\mathrm{Sym}^{\bullet}(\pP_0^{\vee}x^{-1}) \\
&-\wedge^{\bullet}(\lL \pP_0 x)
\wedge^{\bullet}(\pP_1x)
\mathrm{Sym}^{\bullet}(\lL \pP_1 x)
\mathrm{Sym}^{\bullet}(\pP_0 x).
\end{align*}
Therefore we obtain (\ref{expand:P2}). 
\end{proof}
We also set 
\begin{align*}
\delta(x) \cneq \left(\frac{1}{x-1}  \right)\Big|_{x=\infty -0}
=\sum_{k\in \mathbb{Z}}x^k. 
\end{align*}
We have the following relations
\begin{align*}
\delta\left(\frac{x}{y} \right)
f(x)|_{x=\infty}=\delta\left(\frac{x}{y} \right)
f(y)|_{y=\infty}, \ 
\delta\left(\frac{x}{y} \right)
f(x)|_{x=0}=\delta\left(\frac{x}{y} \right)
f(y)|_{y=0}. 
\end{align*}
Also for two rational functions $f(x)$, $g(y)$, 
we define
\begin{align*}
f(x) g(y)|_{(x, y)=\infty -0} \cneq 
f(x)|_{x=\infty} g(y)|_{y=\infty} -f(x)|_{x=0} g(y)|_{y=0}. 
\end{align*}
For a two variable rational function $h(x, y) \in K_{\perf}(\fM)(x, y)$, we denote by 
\begin{align*}
h(x, y) \Big|_{\begin{subarray}{c}
x=\infty -0 \\
y=\infty -0
\end{subarray}}
\in K_{\perf}(\fM)\{x, y\}
\end{align*}
first apply $|_{x=\infty -0}$ and 
then apply $|_{y=\infty-0}$. 
The following lemma can be checked by a direct 
calculation. 
\begin{lem}
For two rational functions $f(x)$, $g(y)$
and non-zero $\alpha$, we have 
\begin{align}\label{residue}
\frac{f(x)g(y)}{y/x-\alpha}\Big|_{\begin{subarray}{c}
x=\infty -0 \\
y=\infty-0
\end{subarray}}
-\frac{f(x)g(y)}{y/x-\alpha}\Big|_{\begin{subarray}{c}
y=\infty -0 \\
x=\infty-0
\end{subarray}}
=-\frac{1}{\alpha}\delta\left(\frac{y}{\alpha x}  \right)
\{f(x) g(y)|_{(x, y)=\infty -0}\}.
\end{align}
\end{lem}
For an object 
$\fE \in \mathrm{Perf}(\fM)$, 
suppose that 
$\fE|_x$ is 
of cohomological amplitude $[-1, 0]$
for any point $x \to \mM$. 
We have 
projectivizations
\begin{align*}
\pi \colon \mathbb{P}_{\fM}(\fE) \to \fM, \ 
\pi' \colon \mathbb{P}_{\fM}(\fE^{\vee}[1]) \to \fM
\end{align*}
which are quasi-smooth and proper. 
We will use the following lemma. 

\begin{lem}\label{lem:relation:K}
Suppose that $\mM=[Q/G]$ where $Q$ is a quasi-projective scheme 
and $G$ is a reductive algebraic group which acts on $Q$. 
Then 
for any $\alpha \in K(\fM)$, 
we have the following relations
in $K(\fM)\{z\}$
\begin{align}\notag
\pi_{\ast}\left[\delta\left( \frac{\oO_{\pi}(1)}{z} \right)  \right]
\otimes \alpha
=\wedge^{\bullet}\left(-\frac{\fE}{z}  \right)\Big|_{z=\infty -0}
\otimes \alpha, \ 
\pi'_{\ast}\left[\delta\left( \frac{\oO_{\pi'}(-1)}{z} \right)  \right]
\otimes \alpha
=\wedge^{\bullet}\left(\frac{z}{\fE}  \right)\Big|_{z=\infty -0}
\otimes \alpha.
\end{align}
Here elements before taking $\otimes \alpha$
are defined in $K_{\perf}(\fM)\{z\}$. 
\end{lem}
\begin{proof}
We set $\eE=\fE|_{\mM} \in \mathrm{Perf}(\mM)$, and 
take the similar projectivizations
on classical truncations
\begin{align*}
\overline{\pi} \colon \mathbb{P}_{\mM}(\eE) \to \mM, \
\overline{\pi}' \colon \mathbb{P}_{\mM}(\eE^{\vee}[1]) 
\to \mM.
\end{align*}
By the isomorphism (\ref{isom:K}), 
we can write $\alpha=i_{\ast}\alpha'$ for 
some $\alpha' \in K(\mM)$. 
Note that for any $\beta \in K_{\perf}(\fM)$ 
we have 
\begin{align*}
\beta \otimes i_{\ast}\alpha'
=i_{\ast}(i^{\ast}\beta \otimes \alpha').
\end{align*}
Therefore 
it is enough to show the following identities
in $K_{\perf}(\mM)\{z\}$
\begin{align}\label{id:eE}
\pi_{\ast}\left[\delta\left( \frac{\oO_{\overline{\pi}}(1)}{z} \right)  \right]
=\wedge^{\bullet}\left(-\frac{\eE}{z}  \right)\Big|_{z=\infty -0}, \ 
\overline{\pi}'_{\ast}
\left[\delta\left( \frac{\oO_{\pi'}(-1)}{z} \right)  \right]
=\wedge^{\bullet}\left(\frac{z}{\eE}  \right)\Big|_{z=\infty -0}.
\end{align}
By the assumption on $\mM$, 
 we can represent $\eE$ as a two term 
complex 
\begin{align*}
\eE=(\eE^{-1} \stackrel{d}{\to} \eE^0)
\end{align*}
where $\eE^{-1}$ and $\eE^0$
are $G$-equivariant vector 
bundles on $Q$ and $d$ is $G$-equivariant. 
Then the identities (\ref{id:eE})
follow from~\cite[Proposition~5.19]{Negut}. 
\end{proof}

\begin{rmk}\label{rmk:qproj}
It is well-known that 
any finite type 
derived open substack
of $\fM_S(v)$, $\fM^{\dag}_S(v)$
satisfies the assumption of Lemma~\ref{lem:relation:K}, i.e.
its classical truncation is of the form $[Q/G]$
as in Lemma~\ref{lem:relation:K}. 
For example, see~\cite[Proposition~4.1.1]{KaVa2}. 
\end{rmk}

\subsection{Actions on K-theory}
Now we return to the situation in Section~\ref{sec:Hecke}. 
For a fixed $\beta \in \mathrm{NS}(S)$, 
let $\fM_S^{\dag}(\beta)^{\pure}$ be defined by
\begin{align*}
\fM_S^{\dag}(\beta)^{\pure} \cneq 
\coprod_{n \in \mathbb{Z}}
\fM_S^{\dag}(\beta, n)^{\pure}. 
\end{align*}
Using the notation in the diagrams (\ref{dia:Hecke}),
(\ref{dia:Hecke2}),  
we define the following maps 
on the K-theory
\begin{align}\label{def:mu+-}
&\mu^{+}(z) \cneq (\pi_{2}, \pi_3)_{\ast}\left(\pi_1^{\ast}(-) \otimes 
\delta\left(\frac{\lL}{z}\right)\right)
 \colon K(\fM_S^{\dag}(\beta)^{\pure}) \to 
K(\fM_S^{\dag}(\beta)^{\pure} \times S)\{z\}, \\
\notag
&\mu^{-}(z) \cneq (\pi_1, \pi_3)_{!}\left(\pi_2^{\ast}(-) \otimes 
\delta\left(\frac{\lL}{z}\right)\right)
 \colon K(\fM_S^{\dag}(\beta)^{\pure}) \to 
K(\fM_S^{\dag}(\beta)^{\pure} \times S)\{z\}. 
\end{align}
We denote by 
\begin{align*}
\fI^{\bullet}(\beta)= (\oO_{S \times \fM_S^{\dag}(\beta)^{\pure}} \to 
\ffF(\beta)) \in \mathrm{Perf}(S\times \fM_S^{\dag}(\beta)^{\pure})
\end{align*}
the object associated with the universal pair. 
By Lemma~\ref{lem:proper}, the map 
$\mu^-(z)$ is written as
\begin{align*}
\mu^{-}(z) = (\pi_{1}, \pi_3)_{\ast}\left(\pi_2^{\ast}(-) \otimes 
\delta\left(\frac{\lL}{z}\right)\right)
\cdot (-\det \ffF(\beta)).
\end{align*}
We write $\mu^{\pm}(z)$ as 
\begin{align*}
\mu^{\pm}(z)=\sum_{k \in \mathbb{Z}} \mu_k^{\pm}z^{-k}, \ 
\mu_k^{\pm} \colon K(\fM_S^{\dag}(\beta)^{\pure}) \to K(\fM_S^{\dag}(\beta)^{\pure} \times S).
\end{align*}
Then by Lemma~\ref{lem:muHecke} 
and Definition~\ref{defi:mu-}, 
the functors $\mu_{\eE, k}^{\pm}$
satisfy the following relations
\begin{align}\label{eqn:mu}
p_{\fM\ast}(\mu_k^{+}(-) \otimes p_S^{\ast}\eE) =\mu_{\eE, k}^{+}(-), \
p_{\fM!}(\mu_k^{-}(-) \otimes p_S^{\ast}\eE) =\mu_{\eE, k}^{-}(-).
\end{align}
Here $p_{\fM}$ and $p_S$ are the projections from 
$\fM_S^{\dag}(\beta)^{\pure} \times S$
onto corresponding factors. 
We then set
\begin{align*}
[\mu^+(z), \mu^-(w)]
 \colon K(\fM_S^{\dag}(\beta)^{\pure}) \to 
K(\fM_S^{\dag}(\beta)^{\pure} \times S \times S)\{z, w\}
\end{align*}
by the following 
\begin{align*}
[\mu^+(z), \mu^-(w)] \cneq 
(\mu^+(z) \boxtimes \id_S)\circ \mu^-(w)-
(\mu^-(w) \boxtimes \id_S)\circ \mu^+(z). 
\end{align*}
The rest of this section is devoted to 
the computation of $[\mu^+(z), \mu^-(w)]$
following the argument of~\cite{Negut}.

\subsection{Compositions of Hecke actions}
In order to compute the composition 
$\mu^{-} \circ \mu^{+}$, we 
consider the following diagram
\begin{align}\label{dia:spade}
\xymatrix{
\fH ecke^{\spadesuit}(v_{\bullet}) \ar@{}[dr]|\square
\ar@/^30pt/[rr]^-{(\overline{\pi}_1^{\spadesuit}, p_1, p_2)} 
\ar[r]^-{q_2} 
\ar@/_60pt/[dd]_-{\pi_1^{\spadesuit}} \ar[d]_-{q_1} 
& \fH ecke(v_{\bullet}) \times S \ar[r]^-{(\pi_1, \pi_3, \id_S)} 
\ar[d]^-{(\pi_2, \id_S)} &  \fM_S^{\dag}(v_1)^{\pure} 
\times S \times S \\
\fH ecke(v_{\bullet}) \ar[d]_-{\pi_1} \ar[r]^-{(\pi_2, \pi_3)} 
& \fM_S^{\dag}(v_2)^{\pure} \times S & \\
\fM_S^{\dag}(v_1)^{\pure} & 
}
\end{align}
Here $\fH ecke^{\spadesuit}(v_{\bullet})$
is defined by the top left Cartesian square. 
From the construction, it parametrizes diagrams
\begin{align}\label{dia:Hecke3}
\xymatrix{
& F_1 \ar@<-0.3ex>@{^{(}->}[rd]^-{x} &  \\
\oO_S \ar[ru] \ar[rd] &   & F_2.  \\
& F_1' \ar@<-0.3ex>@{^{(}->}[ru]_-{y}&
}
\end{align}
Here 
$F_1, F_1'$ are points of 
$\mM_S(v_1)^{\pure}$, 
$F_2$ is a point of $\mM_S(v_2)^{\pure}$, 
and $F_1 \stackrel{x}{\hookrightarrow} F_2$ means 
that $F_2/F_1=\oO_x$ for $x \in S$. 

Similarly we take 
\begin{align*}
v_{\bullet}'=(v_1', v_2', v_3'), \ 
v_1'=(\beta, n-1), v_2'=v_1=(\beta, n), \ v_3'=(0, 1)
\end{align*}
and consider the diagram 
\begin{align}\label{dia:dia}
\xymatrix{
\fH ecke^{\diamondsuit}(v'_{\bullet}) \ar@{}[dr]|\square
\ar@/^30pt/[rr]^-{(\overline{\pi}_1^{\diamondsuit}, p_1', p_2')} 
\ar[r]^-{q_2'} 
\ar@/_60pt/[dd]_-{\pi_1^{\diamondsuit}} \ar[d]_-{q_1'} 
& \fH ecke(v_{\bullet}') \times S \ar[r]^-{(\pi_2', \pi_3', \id_S)} 
\ar[d]^-{(\pi_1', \id_S)} &  \fM_S^{\dag}(v_2')^{\pure} 
\times S \times S \\
\fH ecke(v_{\bullet}') \ar[d]_-{\pi_2'} \ar[r]^-{(\pi_1', \pi_3')} 
& \fM_S^{\dag}(v_1')^{\pure} \times S & \\
\fM_S^{\dag}(v_2')^{\pure} & 
}
\end{align}
Here $\fH ecke^{\diamondsuit}(v'_{\bullet})$
is defined by the top left Cartesian square. 
It parametrizes diagrams
\begin{align}\label{dia:Hecke4}
\xymatrix{
&  &  F_1 \\
\oO_S \ar[r] & F_0 
\ar@<-0.3ex>@{^{(}->}[ru]^-{y}
\ar@<-0.3ex>@{^{(}->}[rd]_-{x} &  \\
& & F_1'. 
}
\end{align}
Here 
$F_1, F_1'$ are points of 
$\mM_S(v_1)^{\pure}$, 
$F_0$ is a point of $\mM_S(v_1')^{\pure}$.
The following is an analogue of~\cite[Claim~3.8]{Negut}
for our situation. 
\begin{lem}\label{lem:dia:equiv}
Let $(F_1, F_1', x, y)$ be given, 
where $F_1, F_1'$ are points of 
$\fM_S(v_1)^{\pure}$ and 
$x, y \in S$. 
Suppose that either $x \neq y$ or 
$F_1$ is not isomorphic to $F_1'$. 
Then giving a diagram (\ref{dia:Hecke3}) is equivalent to 
giving a diagram (\ref{dia:Hecke4}). 
\end{lem}
\begin{proof}
Suppose that a diagram (\ref{dia:Hecke3}) is given. 
The assumption implies that 
$F_1 \neq F_1'$ inside $F_2$. 
Therefore by setting $F_0=F_1 \cap F_1'$
inside $F_2$, we obtain the diagram (\ref{dia:Hecke4}). 
Conversely suppose that a diagram (\ref{dia:Hecke4}) is given. 
We set $F_2$ by the exact sequence
\begin{align*}
0 \to F_0 \to F_1 \oplus F_1' \to F_2 \to 0. 
\end{align*}
We claim that, under
 the assumption of the lemma, 
the sheaf $F_2$ is a pure one dimensional sheaf. 
The claim is obvious if $x\neq y$, so 
we assume that $x=y$. 
If $F_2$ is not pure, 
there is an injection 
$\oO_x \hookrightarrow F_2$. 
By setting $F_2'=F_2/\oO_x$, 
we have the morphisms
\begin{align}\label{compose:F}
F_1 \stackrel{(\id, 0)}{\hookrightarrow}
 F_1 \oplus F_1' \twoheadrightarrow 
F_2 \twoheadrightarrow F_2'. 
\end{align}
The composition of the above morphisms
 is generically injective, hence injective 
as $F_1$ is pure. 
Since $\chi(F_1)=\chi(F_2')$, 
the morphism (\ref{compose:F}) is an isomorphism, 
$F_1 \cong F_2'$. Similarly 
$F_1' \cong F_2'$, which contradicts 
to the assumption that $F_1$ is not isomorphic 
to $F_1'$. Therefore $F_2$ is pure one dimensional, 
and we obtain the diagram (\ref{dia:Hecke3}). 
\end{proof}

Using diagrams (\ref{dia:spade}), (\ref{dia:dia}) and 
Lemma~\ref{lem:dia:equiv}, 
we have the following lemma, which 
is an analogue of~\cite[(3.32)]{Negut}.

\begin{lem}\label{lem:diagonal}
There exists $\gamma \in K(\fM_S^{\dag}(\beta)^{\pure} \times S)\{z, w\}$
such that we have 
the commutative diagram
\begin{align}\label{dia:gamma}
\xymatrix{
K_{\perf}(\fM_S^{\dag}(\beta)^{\pure}) \ar[r] \ar[rd]_-{\boxtimes (\mathrm{id}, \Delta)_{\ast}\gamma} & 
K(\fM_S^{\dag}(\beta)^{\pure}) \ar[d]^-{[\mu^+(z), \mu^-(w)]} \\
& K(\fM_S^{\dag}(\beta)^{\pure} \times S \times S)\{z, w\}. 
}
\end{align}
Here 
the horizontal arrow is the natural map, and 
$(\mathrm{id}, \Delta) \colon \fM_S^{\dag}(\beta)^{\pure} \times S
\to \fM_S^{\dag}(\beta)^{\pure} \times S \times S$
is the product with diagonal
$\Delta \colon S \hookrightarrow S \times S$. 
\end{lem}
\begin{proof}
From the diagrams (\ref{dia:spade}), (\ref{dia:dia}), 
we have 
\begin{align*}
&(\mu^-(w) \boxtimes \id_S)\circ \mu^+(z)
=(\overline{\pi}^{\spadesuit}_{1}, p_1, p_2)_{\ast}
\left((\pi_1^{\spadesuit})^{\ast}(-) \otimes 
\delta\left(\frac{q_1^{\ast}\lL}{z}\right)
\otimes 
\delta\left(\frac{q_2^{\ast}\lL}{w}
\right) \otimes 
\left(-(\overline{\pi}_1^{\spadesuit}, p_2)^{\ast}\det \ffF(\beta)  \right)
\right) , \\
&(\mu^+(z) \boxtimes \id_S)\circ \mu^-(w)
=(\overline{\pi}^{\diamondsuit}_{1}, p_1', p_2')_{\ast}
\left((\pi_1^{\diamondsuit})^{\ast}(-) \otimes 
\delta\left(\frac{q_1^{'\ast}\lL}{w}\right)
\otimes 
\delta\left(\frac{q_2^{'\ast}\lL}{z}\right)
\left(-(\pi_1^{\diamondsuit}, p_1')^{\ast}\det \ffF(\beta) \right)
\right).
\end{align*}
By Lemma~\ref{lem:dia:equiv}, two derived stacks
over $\fM_S^{\dag}(v_1)^{\pure} \times \fM_S^{\dag}(v_1)^{\pure} \times 
S \times S$
\begin{align*}
\xymatrix{ 
\fH ecke^{\spadesuit}(v'_{\bullet}) 
\ar[rd]_-{(\pi_1^{\spadesuit}, \overline{\pi}_1^{\spadesuit}, p_1, p_2)}   
& & \fH ecke^{\diamondsuit}(v'_{\bullet}) 
\ar[ld]^-{(\pi_1^{\diamondsuit}, \overline{\pi}_1^{\diamondsuit}, p'_1, p'_2)}
\\
& \fM_S^{\dag}(v_1)^{\pure} \times \fM_S^{\dag}(v_1)^{\pure} \times 
S \times S.  &
}
\end{align*}
are equivalent outside 
the diagonal 
\begin{align*}
(\Delta_{\fM} \times \Delta) \colon 
\fM_S^{\dag}(v_1)^{\pure} \times 
S \hookrightarrow
\fM_S^{\dag}(v_1)^{\pure} \times \fM_S^{\dag}(v_1)^{\pure} \times 
S \times S
\end{align*}
such that $q_1^{\ast}\lL$, $q_2^{'\ast}\lL$ are identified
and $(\overline{\pi}_1^{\spadesuit}, p_2)^{\ast}\det \ffF(\beta)$, 
$(\pi_1^{\diamondsuit}, p_1')^{\ast}\det \ffF(\beta)$
are identified. 
Therefore 
the difference 
\begin{align*}
&(\pi_1^{\spadesuit}, \overline{\pi}_1^{\spadesuit}, p_1, p_2)_{\ast}
\left\{\delta\left(\frac{q_1^{\ast}\lL}{z}  \right)\otimes 
\delta\left(\frac{q_2^{\ast}\lL}{w}\right) \otimes \left(-(\overline{\pi}_1^{\spadesuit}, p_2)^{\ast}\det \ffF(\beta)  \right) \right\} \\
&-(\overline{\pi}_1^{\diamondsuit}, \pi^{\diamondsuit}_1, p'_1, p'_2)_{\ast}
\left\{\delta\left(\frac{q_1^{'\ast}\lL}{z}  \right)\otimes 
\delta\left(\frac{q_2^{'\ast}\lL}{w}\right) \otimes \left(-(
\overline{\pi}_1^{\diamondsuit}, p_1')^{\ast}\det \ffF(\beta)  \right) \right\} \end{align*}
in $K(\fM_S^{\dag}(\beta)^{\pure} \times \fM_S^{\dag}(\beta)^{\pure} \times S \times S)\{z, w\}$
is written as 
$(\Delta_{\fM} \times \Delta)_{\ast} \gamma$
for some $\gamma \in K(\fM_S^{\dag}(\beta)^{\pure} \times S)\{z, w\}$.
Then the commutator $[\mu^+(z), \mu^-(w)](-)$ 
applied for perfect complexes coincides with 
$(-) \otimes (\id, \Delta)_{\ast}\gamma$, 
therefore the lemma holds. 
\end{proof}

Below we take derived open substacks 
$\fM_S^{\dag}(\beta, n)^{\fin} \subset \fM_S^{\dag}(\beta, n)^{\pure}$
for each $n$, 
satisfying the condition (\ref{PT:qcompact}). 
Then we set
\begin{align*}
\fM_S^{\dag}(\beta)^{\fin} \cneq 
\coprod_{n \in \mathbb{Z}} \fM_S^{\dag}(\beta, n)^{\fin}
\subset \fM_S^{\dag}(\beta)^{\pure}. 
\end{align*}
Let $h^{\pm}(z) \in K_{\perf}(\fM_S^{\dag}(\beta)^{\pure} \times S)\{z\}$
be defined by 
\begin{align*}
h^+(z)=\left(1-\frac{1}{z} \right)\wedge^{\bullet} \left(\frac{(q_S^{-1}-1)\ffF(\beta)}{z}\right)
\Big|_{z=\infty}, \
h^-(z)=\left(1-\frac{1}{z} \right)\wedge^{\bullet} \left(\frac{(q_S^{-1}-1)\ffF(\beta)}{z}\right)
\Big|_{z=0}. 
\end{align*}
Here $q_S$ is the class $[\omega_S] \in K(S)$, pulled back 
to $\fM_S^{\dag}(\beta)^{\pure} \times S$. 
The restrictions of $h^{\pm}(z)$
to $\fM_S^{\dag}(\beta)^{\fin} \times S$
are also denoted by $h^{\pm}(z)$. 
We have the following proposition, which is 
an analogue of~\cite[Proposition~3.6]{Negut}.  

\begin{prop}\label{prop:comm}
We have the following diagram
\begin{align*}
\xymatrix{
K_{\perf}(\fM_S^{\dag}(\beta)^{\pure} ) \ar[rd]^-{\boxtimes (\mathrm{id}, \Delta)_{\ast}\gamma}
 \ar[d]_-{\boxtimes \rho(z, w)} &  \\
K_{\perf}(\fM_S^{\dag}(\beta)^{\pure} \times S \times S)\{z, w\}
\ar[r] & K(\fM_S^{\dag}(\beta)^{\pure}\times S \times S)\{z, w\}
}
\end{align*}
which commutes after restricting 
the both compositions to $\fM_S^{\dag}(\beta)^{\fin} \times S \times S$. 
Here $\rho(z, w)$ is given by 
\begin{align}\label{def:rho}
\rho(z, w) \cneq 
(\mathrm{id}, \Delta)_{\ast}\frac{h^+(z)-h^-(w)}{q_S-1}
\delta\left(\frac{w}{z} \right)
\in K_{\perf}(\fM_S^{\dag}(\beta)^{\pure} \times S \times S)\{z, w\}.
\end{align}
\end{prop}
\begin{proof}
It is enough to show the following identity in
$K(\fM_S^{\dag}(\beta)^{\fin} \times S)\{z, w\}$
\begin{align}\label{gamma:form}
\gamma=\frac{\delta\left(w/z\right)}{q_S-1} \left[
\left(1-\frac{1}{z} \right)\wedge^{\bullet} \left(\frac{(q_S^{-1}-1)\ffF(\beta)}{z}\right)
\Big|_{z=\infty}
-\left(1-\frac{1}{w} \right)\wedge^{\bullet} \left(\frac{(q_S^{-1}-1)\ffF(\beta)}{w}\right)
\Big|_{w=0}
 \right]. 
\end{align}
We compute $\gamma$
by the identity
from Lemma~\ref{lem:diagonal}
\begin{align*}
[\mu^+(z), \mu^-(w)] \cdot 1=(\id, \Delta)_{\ast}\gamma. 
\end{align*}
Here $1$ is the class of the structure sheaf of $\fM_S^{\dag}(\beta)^{\pure}$. 
Let $\lL$ be the line bundle defined in (\ref{def:lL}). 
By Lemma~\ref{fiber:pi} and Lemma~\ref{lem:proper},
 we have 
\begin{align*}
\lL=\oO_{(\pi_1, \pi_3)}(-1)=\oO_{(\pi_2, \pi_3)}(1). 
\end{align*}
Therefore by Lemma~\ref{lem:proper}, 
Lemma~\ref{lem:relation:K} and Remark~\ref{rmk:qproj}, 
after restricting to 
$\fM_S^{\dag}(\beta)^{\fin} \times S$, 
we have 
\begin{align*}
\mu^+(z) \cdot 1&=\wedge^{\bullet}\left( -\frac{\fI^{\bullet}(\beta)[1]}{z} \right)\Big|_{z=\infty -0}.
\end{align*}
Below when we write some identity such as above, 
it always means the identity after restricting to 
$\fM_S^{\dag}(\beta)^{\fin}$. 
Similarly by Lemma~\ref{fiber:pi} and Lemma~\ref{lem:proper}, 
we have 
\begin{align*}
\mu^-(w) \cdot 1=\wedge^{\bullet} \left( \frac{wq_S}{\ffF(\beta)} 
\right)\Big|_{w=\infty-0} \cdot (-\det \ffF(\beta))=
\wedge^{\bullet} \left( \frac{\ffF(\beta)}{wq_S} 
\right)\Big|_{w=\infty-0} . 
\end{align*}
On the other hand using the diagram
\begin{align*}
\xymatrix{
\fH ecke(v_{\bullet}) \ar[r]^-{(\id_{\fH}, \pi_3)} & \fH ecke(v_{\bullet}) \times S 
\ar[d]_-{(\pi_1, \id_S)} \ar[r]^-{(\pi_2, \id_S)} &
\fM_S^{\dag}(v_2)^{\pure} \times S \\
& \fM_S^{\dag}(v_1)^{\pure} \times S
}
\end{align*}
we have the distinguished triangle 
\begin{align}\label{dist:I}
(\pi_1, \id_S)^{\ast}\fI^{\bullet}(v_1)[1] \to (\pi_2, \id_S)^{\ast}
\fI^{\bullet}(v_2)[1] \to 
(\id_{\fH}, \pi_3)_{\ast}\lL. 
\end{align}
Then using (\ref{dist:I}), we have 
\begin{align}\notag
\mu^-(w) \mu^+(z) \cdot 1 &=
\mu^-(w) \wedge^{\bullet}\left(-\frac{\fI^{\bullet}(\beta)[1]} {z}
\right)\Big|_{z=\infty-0} \\
\notag
&=\wedge^{\bullet}\left(-\frac{\fI^{\bullet(2)}(\beta)[1]}{z} \right)
\wedge^{\bullet}\left( -\frac{w}{z}\oO_{\Delta} \right) \mu^-(w) 
\cdot 1 \Big|_{z=\infty-0}  \\
\label{mu+-1}
&=\wedge^{\bullet}\left(\frac{\ffF^{(1)}(\beta)}{wq_S^{(1)}}\right)
\wedge^{\bullet}\left(-\frac{\fI^{\bullet(2)}(\beta)[1]}{z} \right)
\wedge^{\bullet}\left( -\frac{w}{z}\oO_{\Delta} \right)\Big|_{\begin{subarray}{c}
z=\infty-0 \\
w=\infty-0
\end{subarray}}. 
\end{align}
Here the subscript $(1)$, $(2)$ means 
pulling back to 
$\fM_S^{\dag}(\beta)^{\pure} \times S \times S$
by the first and second projection, the first and third projection, 
respectively. 
A similar computation shows that 
\begin{align}\label{mu+-2}
\mu^+(z) \mu^-(w) \cdot 1
=\wedge^{\bullet}\left(\frac{\ffF^{(1)}(\beta)}{wq_S^{(1)}}\right)
\wedge^{\bullet}\left(-\frac{\fI^{\bullet(2)}(\beta)[1]}{z} \right)
\wedge^{\bullet}\left( -\frac{w}{z}\oO_{\Delta} \right)\Big|_{\begin{subarray}{c}
w=\infty-0 \\
z=\infty-0
\end{subarray}}.
\end{align}
The difference between (\ref{mu+-1}) and (\ref{mu+-2}) is that, 
in the former we first 
expand by $z$ and then expand by $w$, while 
in the latter we do in the opposite order. 
Therefore by (\ref{residue}), 
the difference comes from the poles in 
$w/z$, 
\begin{align*}
[\mu^+(z), \mu^-(w)] \cdot 1=
\sum_{\alpha \in \mathrm{Pol}(\xi^S)} \mathop{\mathrm{Res}}\limits_{x=\alpha} 
\xi^S(x) 
\cdot \frac{1}{\alpha}
\delta\left( \frac{w}{\alpha z}\right) \left\{\wedge^{\bullet}\left(\frac{\ffF^{(1)}(\beta)}{wq_S^{(1)}}\right)
\wedge^{\bullet}\left(-\frac{\fI^{\bullet(2)}(\beta)[1]}{z} \right) \right\}\Big|_{(z, w)=\infty-0}.
\end{align*}
Here $\xi^S(x) \in K(S \times S)(x)$ is defined by
\begin{align*}
\xi^S(x)=\wedge^{\bullet}(-x \cdot \oO_{\Delta})
=1+\frac{[\oO_{\Delta}] \cdot x}{(1-x)(1-xq_S)}
\end{align*}
where the second identity is due 
to~\cite[Equation~(3.5), (3.6)]{Negut}. 
Also $\mathrm{Pol}(\xi^S(x))$ is the set of poles of $\xi^S(x)$,
i.e. $\{1, q_S^{-1}\}$. 
Therefore we have 
\begin{align}\notag
[\mu^+(z), \mu^-(w)] \cdot 1
=&\delta\left( \frac{w}{z}\right) \wedge^{\bullet}\left(\frac{\ffF(\beta)}{zq_S}\right)
\wedge^{\bullet}\left(-\frac{\fI^{\bullet}(\beta)[1]}{z} \right)
\frac{[\oO_{\Delta}]}{q_S-1}
\Big|_{z=\infty-0} \\
\label{cu:pm}
&+\delta\left( \frac{wq_S}{z}\right) \wedge^{\bullet}\left(\frac{\ffF(\beta)}{z}\right)
\wedge^{\bullet}\left(-\frac{\fI^{\bullet}(\beta)[1]}{z} \right)
\frac{q_S[\oO_{\Delta}]}{q_S-1}
\Big|_{z=\infty-0}. 
\end{align}
Here we have identified
$\ffF^{(1)}(\beta)=\ffF(\beta)$, 
$\fI^{\bullet(2)}(\beta)=\fI^{\bullet}(\beta)$ because of the factor $[\oO_{\Delta}]$.  
By the distinguished triangle
\begin{align}\label{dist:Ibeta}
\ffF(\beta) \to \fI^{\bullet}(\beta)[1] \to \oO_{\fM_S^{\dag}(\beta)^{\pure} \times S}[1]
\end{align}
the second term of (\ref{cu:pm}) is 
\begin{align*}
\delta\left( \frac{wq_S}{z}\right) \wedge^{\bullet} 
\left(\frac{[\oO_{\fM_S^{\dag}(\beta)^{\pure} \times S \times S}]}{z} \right)
\frac{q_S[\oO_{\Delta}]}{q_S-1}
\Big|_{z=\infty-0}=0. 
\end{align*}
By computing the first term of (\ref{cu:pm})
using (\ref{dist:Ibeta}), 
we have 
\begin{align*}
[\mu^+(z), \mu^-(w)] \cdot 1=
\delta\left( \frac{w}{z}\right)
\frac{[\oO_{\Delta}]}{q_S-1}
\left(1-\frac{1}{z} \right)\wedge^{\bullet}\left(\frac{(q_S^{-1}-1)\ffF(\beta)}{z} \right)
\Big|_{\infty -0}. 
\end{align*}
Therefore we obtain (\ref{gamma:form}). 
\end{proof}


For $\eE \in \Dbc(S)$, 
let $\mu_{\eE, k}^{\pm}$ be the functors 
defined in Definition~\ref{def:muk}, Definition~\ref{defi:mu-}.
We use the same notation for the induced maps
on K-theories
\begin{align*}
\mu_{\eE, k}^{\pm} \colon K(\dD \tT^{\mathbb{C}^{\ast}}(P_n(X, \beta)))
\to K(\dD \tT^{\mathbb{C}^{\ast}}(P_{n\pm 1}(X, \beta))).
\end{align*}
Then we set
\begin{align*}
\mu_{\eE}^{\pm}(z)
\cneq \sum_{k\in\mathbb{Z}} \frac{\mu_{\eE, k}^{\pm}}{z^{k}} \colon \bigoplus_{n \in \mathbb{Z}}
K(\dD \tT^{\mathbb{C}^{\ast}}(P_n(X, \beta)))
\to \bigoplus_{n\in \mathbb{Z}}
K(\dD \tT^{\mathbb{C}^{\ast}}(P_{n}(X, \beta)))\{z\}. 
\end{align*}
We have the following
result. 

\begin{thm}\label{thm:relation}
We have the following commutative diagram
\begin{align}\label{com:eqn:cor}
\xymatrix{
\bigoplus_{n\in \mathbb{Z}}K(\mathcal{DT}_{\perf}^{\mathbb{C}^{\ast}}(P_n(X, \beta))) \ar[r] \ar[d]_-{\otimes \rho_{\eE_1, \eE_2}(z, w)} & \bigoplus_{n\in \mathbb{Z}}K(\mathcal{DT}^{\mathbb{C}^{\ast}}(P_n(X, \beta))) \ar[d]^-{[\mu_{\eE_1}^+(z), \mu_{\eE_2}^-(w)]} \\
\bigoplus_{n\in \mathbb{Z}}K(\mathcal{DT}_{\perf}^{\mathbb{C}^{\ast}}(P_n(X, \beta)))\{z, w\} \ar[r] & 
\bigoplus_{n\in \mathbb{Z}}K(\mathcal{DT}^{\mathbb{C}^{\ast}}(P_n(X, \beta)))\{z, w\}. 
}
\end{align}
Here 
the horizontal arrows are natural maps, and 
$\rho_{\eE_1, \eE_2}(z, w)$ is given by 
\begin{align}\notag
\rho_{\eE_1, \eE_2}(z, w)=
p_{\fP\ast}\left(
(\eE_1 \otimes \eE_2 \otimes \omega_S) \boxtimes 
\frac{h^+(z)-h^-(w)}{q_S-1}
\delta\left(\frac{w}{z}\right) \right)
\in \bigoplus_{n\in \mathbb{Z}}K(\mathcal{DT}_{\perf}^{\mathbb{C}^{\ast}}(P_n(X, \beta)))\{z, w\}
\end{align}
where $p_{\pP} \colon \fP_n(S, \beta) \times S \to \fP_n(S, \beta)$
is the projection
and we have used the equivalence (\ref{equiv:perf}). 
\end{thm}
\begin{proof}
Let $q_{\fM}$, $r_{S1}$, $r_{S2}$ be the projections 
from $\fM_S^{\dag}(\beta)^{\pure} \times S \times S$ onto 
the corresponding factors. 
By (\ref{eqn:mu}) and noting that 
$p_{\fM!}(-)=p_{\fM\ast}(- \boxtimes \omega_S[2])$
for the projection $p_{\fM} \colon \fM_S^{\dag}(\beta)^{\pure} \times S \to 
\fM_S^{\dag}(\beta)^{\pure}$, 
we have 
\begin{align}\label{mu:comu}
&\mu^+_{\eE_1}(z) \circ \mu_{\eE_2}^-(w)(-)
=q_{\fM\ast}\{ (\mu^+(z) \circ \mu^-(w)(-))
\otimes r_{S1}^{\ast}(\eE_2 \otimes \omega_S) \otimes
r_{S2}^{\ast}(\eE_1) \}, \\
\notag
&\mu^-_{\eE_2}(w) \circ \mu^+_{\eE_1}(z)(-)
=q_{\fM\ast}\{ (\mu^-(w) \circ \mu^+(z)(-))
\otimes r_{S1}^{\ast}(\eE_1) \otimes
r_{S2}^{\ast}(\eE_2  \otimes \omega_S) \}.
\end{align}
Then the commutative diagram
 (\ref{com:eqn:cor}) follows from Lemma~\ref{lem:diagonal} 
and 
Proposition~\ref{prop:comm}.
\end{proof}

By taking the commutator for $k=0$, we 
obtain the following corollary. 
\begin{cor}\label{cor:com:k=0}
We have the following commutative diagram
\begin{align*}
\xymatrix{
\bigoplus_{n\in \mathbb{Z}}K(\mathcal{DT}_{\perf}^{\mathbb{C}^{\ast}}(P_n(X, \beta))) \ar[r] \ar[d]_-{\otimes \rho_{\eE_1, \eE_2}} & \bigoplus_{n\in \mathbb{Z}}K(\mathcal{DT}^{\mathbb{C}^{\ast}}(P_n(X, \beta))) \ar[d]^-{[\mu_{\eE_1, k=0}^+, \mu_{\eE_2, k=0}^-]} \\
\bigoplus_{n\in \mathbb{Z}}K(\mathcal{DT}_{\perf}^{\mathbb{C}^{\ast}}(P_n(X, \beta))) \ar[r] & 
\bigoplus_{n\in \mathbb{Z}}K(\mathcal{DT}^{\mathbb{C}^{\ast}}(P_n(X, \beta))). 
}
\end{align*}
Here $\rho_{\eE_1, \eE_2}$ is given by 
\begin{align}\label{com:eqn:2}
\rho_{\eE_1, \eE_2}=
p_{\fP\ast}((\eE_1 \otimes \eE_2 \otimes \omega_S) 
\boxtimes \ffF(\beta)^{\vee})
\in \bigoplus_{n\in \mathbb{Z}}K(\mathcal{DT}_{\perf}^{\mathbb{C}^{\ast}}(P_n(X, \beta))).
\end{align}
\end{cor}
\begin{proof}
Using the expansion (\ref{expand:P2}), we see that 
\begin{align*}
\wedge^{\bullet}\left( \frac{(q_S^{-1}-1)\ffF(\beta)}{z} \right)\Big|_{z=\infty -0}
=O(z^{-1})+
(1-q_S)[\ffF(\beta)^{\vee}]z+ O(z^2).
\end{align*}
Therefore the constant term of 
$\rho(z, w)$ in (\ref{def:rho}) 
is given by 
$(\id, \Delta)_{\ast}\ffF(\beta)^{\vee}$.
By taking the constant term of the diagram (\ref{com:eqn:cor}), 
we obtain the desired commutative diagram (\ref{com:eqn:2}). 
\end{proof}

\begin{exam}\label{exam:commute}
In Corollary~\ref{thm:relation}, 
let us take $\eE_1=\oO_S$ and $\eE_2=\oO_H$
for an ample divisor $H \subset S$.
Suppose that $H$ intersects with any curve
$C \subset S$ with $[C]=\beta$ transversely. 
Then by 
setting $\mu_S^{\pm}=\mu_{\oO_S, 0}^{\pm}$, 
$\mu_H^{\pm}=\mu_{\oO_H, 0}^{\pm}$, 
the diagram (\ref{com:eqn:2}) implies that 
\begin{align}\label{relation:muSH}
[\mu_H^-, \mu_S^+](-)=
[\mu_S^-, \mu_H^+](-)= (-)\otimes \vV(\beta)
\end{align}
for a vector bundle $\vV(\beta)$ on $\fM_S^{\dag}(\beta)^{\pure}$ of 
rank $\beta \cdot H$
and $(-)$ is a K-group element coming from 
the perfect PT categories. 
We also have obvious vanishings 
$[\mu_{S}^+, \mu_S^+]=[\mu_H^+, \mu_H^-]=0$. 
However $[\mu_S^+, \mu_S^-]$, $[\mu_H^+, \mu_H^-]$
do not necessary vanish (though the latter is a nilpotent operator). 
\end{exam}

\begin{rmk}\label{rmk:planar}
In~\cite{MR3807309}, Rennemo 
proved the following. 
Let $C$ be an irreducible curve with at worst planar singularities, 
and 
$C^{[n]}$ the Hilbert scheme of $n$-points on $C$. 
Then there exist linear maps
\begin{align*}
\mu^{\pm}_{C} \colon 
H_{\ast}(C^{[n]})
\to H_{\ast +1 \pm 1}(C^{[n+1]}), \ 
\mu^{\pm}_{\pt} \colon H_{\ast}(C^{[n]}) \to 
H_{\ast -1 \pm 1}(C^{[n+1]})
\end{align*}
using the moduli space 
parameterizing $(Z, Z')$, 
$Z \in C^{[n]}$, $Z' \in C^{[n+1]}$ with $Z \subset Z'$.
They  
satisfy the following relations of Weyl algebras
\begin{align}\label{relation:planar}
[\mu^{-}_{\pt}, \mu^+_{C}]=[\mu^-_{C}, \mu^+_{\pt}]=\id
\end{align}
and all other pairs of operators commute. 

Since the moduli spaces of stable pairs 
much generalize $C^{[n]}$, 
the result of Corollary~\ref{thm:relation}
may be regarded as a categorification of 
the above result by Rennemo.  
Indeed the relation (\ref{relation:muSH}) 
is regarded as a categorification of (\ref{relation:planar}).
However 
contrary to the vanishing of commutators of $[\mu_{C}^+, \mu_{C}^-]$, 
$[\mu_{\pt}^+, \mu_{\pt}^-] $, 
the commutators $[\mu_S^+, \mu_S^-]$, $[\mu_H^+, \mu_H^-]$ 
in Example~\ref{exam:commute}
do not 
necessary vanish. 
\end{rmk}


\section{Some other actions of DT categories}
\label{sec:other}
In this section, we construct some other 
actions of DT categories
in the following cases: 
the left action of 
DT categories of zero dimensional sheaves 
to MNOP categories, 
the right/left actions of 
DT categories of one dimensional semistable 
sheaves to DT categories for stable 
D0-D2-D6 bound states. 
Almost all the arguments are similar to those 
in Section~\ref{sec:catPT}, 
so we omit details in several places. 
\subsection{Another stacks of extensions}
For $v_{\bullet}=(v_1, v_2, v_3) \in N_{\le 1}(S)^{\times 3}$, we 
define the derived stack 
$\fM_S^{\ext, \ddag}(v_{\bullet})$ by the 
Cartesian square
\begin{align*}
\xymatrix{
\fM_S^{\ext, \ddag}(v_{\bullet}) \ar[r]^-{\ev_2^{\ddag}} \ar[d] 
\ar@{}[rd]|\square
& \fM_S^{\dag}(v_2) \ar[d] \\
\fM_S^{\ext}(v_{\bullet}) \ar[r]^-{\ev_2} & \fM_S(v_2). 
}
\end{align*}
For $T \in dAff$, the $T$-valued points of $\fM_S^{\dag}$ form the 
$\infty$-groupoid of diagrams
\begin{align}\label{diagram:ddag}
\xymatrix{
 & \oO_{S \times T}\ar[d]^-{\xi}  &  \\
\ffF_1 \ar[r] & \ffF_2 \ar[r] & \ffF_3.
}
\end{align}
Here the bottom sequence is a $T$-valued point of 
$\fM_S^{\ext}(v_{\bullet})$. 
We have the following diagram
\begin{align}\label{diagram:ddag2}
\xymatrix{
\fM_S^{\ext, \ddag}(v_{\bullet}) \ar[r]^{\ev_2^{\ddag}} 
\ar[d]_{(\ev_1^{\ddag}, \ev_3^{\ddag})} & \fM_S^{\dag}(v_2) \\
\fM_S(v_1) \times \fM_S^{\dag}(v_3). & 
}
\end{align}
Here $\ev_1^{\ddag}$ sends a diagram (\ref{diagram:ddag}) to $\ffF_1$, 
and $\ev_3^{\ddag}$ sends a diagram (\ref{diagram:ddag})
to the composition 
$\oO_{S \times T} \stackrel{\xi}{\to} \ffF_2 \to \ffF_3$. 
Note that $\ev_2^{\ddag}$ is proper by definition. 

\begin{lem}\label{lem:ddag}
The morphism $(\ev_1^{\ddag}, \ev_3^{\ddag})$ is quasi-smooth. 
In particular, the derived stack 
$\fM_S^{\ext, \ddag}(v_{\bullet})$ is quasi-smooth. 
\end{lem}
\begin{proof}
Similarly to (\ref{MSext:spec}),
we have 
\begin{align*}
\fM_S^{\ext, \ddag}(v_{\bullet})=
\mathbb{V}(\hH om_{p_{\fM \times \fM^{\dag}}}(\fI^{\bullet}(v_3), \ffF(v_1))^{\vee})
\to \fM(v_1) \times \fM^{\dag}(v_3). 
\end{align*} 
Here $p_{\fM \times \fM^{\dag}}$ is the projection 
from $S \times \fM_S(v_1) \times \fM_S^{\dag}(v_3)$ to 
$\fM_S(v_1) \times \fM_S^{\dag}(v_3)$. 
From the distinguished triangle
\begin{align*}
p_{\fM \ast}\ffF(v_1)
\boxtimes \oO_{\fM_S^{\dag}(v_3)} \to 
\hH om_{p_{\fM \times \fM^{\dag}}}(\fI^{\bullet}(v_3), \ffF(v_1))
\to \hH om_{p_{\fM \times \fM^{\dag}}}(\ffF(v_3), \ffF(v_1))[1]
\end{align*}
the middle term is perfect 
whose restriction to any 
point $x \to \mM_S(v_1) \times \mM_S^{\dag}(v_3)$
is 
of cohomological amplitude $[-1, 1]$. 
Therefore $(\ev_1^{\ddag}, \ev_3^{\ddag})$ is quasi-smooth. 
The last statement holds as 
$\fM_S(v_1) \times \fM_S^{\dag}(v_3)$ is also quasi-smooth. 
\end{proof}

We take $v_{\bullet}=(v_1, v_2, v_3) \in N_{\le 1}(S)^{\times 3}$
and define the classical 
stack
\begin{align*}
\mM_X^{\ext, \ddag}(v_{\bullet}) \colon Aff^{op} \to Groupoid
\end{align*}
by sending $T \in Aff$ to the groupoid 
of distinguished triangles
$\eE_1[-1] \to \eE_2 \stackrel{j}{\to} \eE_3$
together with commutative diagrams
\begin{align}\notag
\xymatrix{
\eE_2 \dotimes \oO_{S_{\infty} \times T} \ar[r]^-{j} 
\ar[d]_-{\lambda_2}^-{\cong} & 
\eE_3 \dotimes \oO_{S_{\infty} \times T} \ar[d]^-{\lambda_3}_-{\cong} \\
\oO_{S_{\infty} \times T} \ar@{=}[r] & \oO_{S_{\infty} \times T}. 
}
\end{align}
Here $(\eE_i, \lambda_i)$ for $i=2, 3$ are $T$-valued points of 
$\mM_X^{\dag}(v_i)$
and $\eE_1$ is a $T$-valued point of $\mM_X(v_1)$. 
We also have the evaluation morphisms
\begin{align}\label{dia:Xext:dagA}
\xymatrix{
\mM_X^{\ext, \ddag}(v_{\bullet}) 
\ar[r]^-{\ev_2^{X, \ddag}} \ar[d]_-{(\ev_1^{X, \ddag}, \ev_3^{X, \ddag})} 
& \mM_X^{\dag}(v_2) \\
\mM_X(v_1) \times \mM_X^{\dag}(v_3) & 
}
\end{align}
where $\ev_i^{X, \dag}$ sends
 $\eE_{\bullet}$ to $\eE_i$. 

Let $\ev^{\ddag}$ be the morphism
\begin{align*}
\ev^{\ddag}=(\ev_1^{\ddag}, \ev_2^{\ddag}, \ev_3^{\ddag}) \colon 
\fM_S^{\ddag}(v_{\bullet}) \to 
\fM_S(v_1) \times \fM_S^{\dag}(v_2) \times \fM_S^{\dag}(v_3). 
\end{align*}
Similarly to Proposition~\ref{prop:ev+}, 
we can show the isomorphism over $\mM_S^{\ext, \ddag}(v_{\bullet})$
\begin{align}\label{isom:Mddag}
\mM_X^{\ext, \ddag}(v_{\bullet})
\stackrel{\cong}{\to}
t_0(\Omega_{\ev^{\ddag}}[-2]). 
\end{align}

\subsection{Actions on MNOP categories}
For $(\beta, n) \in N_{\le 1}(S)$, 
we denote by 
\begin{align}\notag
I_n(X, \beta)
\end{align}
the moduli space 
compactly supported closed subschemes 
$Z \subset X$ 
satisfying $[\pi_{\ast}\oO_Z]=(\beta, n)$. 
The moduli space $I_n(X, \beta)$ is a Hilbert scheme of 
one or zero dimensional subschemes of $X$, 
so it is a quasi-projective scheme. 

We have the open immersion
\begin{align*}
I_n(X, \beta) \subset \mM_X^{\dag}(\beta, n)
\end{align*}
sending a closed subscheme $Z \subset X$ to the
ideal sheaf $I_{Z} \subset \oO_{\overline{X}}$. 
We define the following 
conical closed substack
\begin{align*}
\zZ_{I\us}(\beta, n) \cneq \mM_X^{\dag}(\beta, n) \setminus I_n(X, \beta)
\subset \mM_X^{\dag}(\beta, n). 
\end{align*}
Since $I_n(X, \beta)$ is a quasi-projective scheme, 
there is a 
derived open substack 
$\mathfrak{M}_S^{\dag}(\beta, n)^{\fin}
\subset \mathfrak{M}_S^{\dag}(\beta, n)$ 
of finite type
such that 
\begin{align}\label{I:qcompact}
I_n(X, \beta) \subset 
t_0(\Omega_{\mathfrak{M}_S^{\dag}(\beta, n)^{\fin}}[-1])
\subset \mM_X^{\dag}(\beta, n).  
\end{align}

\begin{defi}\label{def:catPT2}\emph{(\cite[Definition~6.6]{TocatDT})}
The $\mathbb{C}^{\ast}$-equivariant 
categorical MNOP theory is defined by
\begin{align*}
\dD \tT^{\mathbb{C}^{\ast}}(I_n(X, \beta)) \cneq 
D^b_{\rm{coh}}(\mathfrak{M}_S^{\dag}(\beta, n)^{\fin})/ 
\cC_{\zZ_{I\us}(\beta, n)^{\fin}}.
\end{align*}
\end{defi}

Let us take $v_{\bullet} \in N_{\le 1}(S)^{\times 3}$ to be 
\begin{align*}
v_1=(0, m)=m[\pt], \ v_2=(\beta, n+m), \ 
v_3=(\beta, n).
\end{align*}
We also take 
derived open substacks 
of finite type 
\begin{align*}
\fM_S(v_1)^{\fin}=\fM_S(m[\pt]), \ 
\fM_S^{\dag}(v_2)^{\fin} \subset \fM_S^{\dag}(v_2), \ 
\fM_S^{\dag}(v_3)^{\fin} \subset \fM_S^{\dag}(v_3)
\end{align*}
satisfying (\ref{I:qcompact}) for $v_2$, $v_3$ and 
the condition 
\begin{align}\label{cond:ev2}
\fM_S^{\ddag, \ext}(v_{\bullet})^{\fin} \cneq 
(\ev^{\ddag}_2)^{-1}(\fM^{\dag}_S(v_2)^{\fin}) 
\subset (\ev^{\ddag}_1, \ev^{\ddag}_3)^{-1}(\fM_S(v_1)^{\fin} 
\times \fM^{\ddag}_S(v_3)^{\fin}). 
\end{align}
Then the diagram (\ref{diagram:ddag2}) restricts to the diagram
\begin{align}\notag
\xymatrix{
\fM_S^{\ext, \ddag}(v_{\bullet})^{\fin} \ar[r]^-{\ev^{\ddag}_2}
 \ar[d]_-{(\ev_1^{\ddag}, \ev_3^{\ddag})} & \fM_S^{\dag}(v_2)^{\fin} \\
\fM_S(v_1)^{\fin} \times \fM_S^{\dag}(v_3)^{\fin}. & 
}
\end{align}
The vertical arrow is quasi-smooth and 
the horizontal arrow is proper. 
Therefore we have the induced functor
\begin{align}\label{cat:COHAI2.5}
\ev^{\ddag}_{2\ast}(\ev^{\ddag}_1, \ev^{\ddag}_3)^{\ast}
 \colon \Dbc(\fM_S(v_1)^{\fin}) \times \Dbc(\fM_S^{\dag}(v_3)^{\fin})
\to \Dbc(\fM_S^{\dag}(v_2)^{\fin}). 
\end{align}
Similarly to Theorem~\ref{thm:PTaction}, we have the 
following. 
\begin{thm}\label{thm:Iaction}
The functor (\ref{cat:COHAI2.5})
descends to the functor
\begin{align}\label{descend:I}
\dD \tT^{\mathbb{C}^{\ast}}(\mM_X^{\sigma \sss}(m[\mathrm{pt}])) \times 
\dD \tT^{\mathbb{C}^{\ast}}(I_n(X, \beta))
\to \dD \tT^{\mathbb{C}^{\ast}}(I_{n+m}(X, \beta)). 
\end{align}
\end{thm}
\begin{proof}
By Lemma~\ref{lem:PT} below, we have 
the following inclusion
\begin{align*}
(\ev_1^{X, \ddag}, \ev_3^{X, \ddag})^{-1}(
\mM_X(v_1) \times 
\zZ_{I\us}(v_3))
\subset (\ev_2^{X, \ddag})^{-1}(\zZ_{I\us}(v_2)). 
\end{align*}
Therefore the result follows from 
Corollary~\ref{cor:induce2}. 
\end{proof}

Here we have used the following lemma. 
\begin{lem}\label{lem:I}
For an exact sequence $0 \to E_1 \to E_2 \to E_3 \to 0$
in $\aA_X$, 
suppose that 
$E_1=Q[-1]$ for a zero dimensional 
sheaf $Q$ and $E_2=I_{Z_2}$ for a compactly supported
one or zero dimensional subscheme $Z_2 \subset X$. 
Then $E_3=I_{Z_3}$ for a compactly supported
one or zero dimensional subscheme $Z_3 \subset X$.
\end{lem}
\begin{proof}
We have the distinguished triangle 
$E_2 \to E_3 \to E_1[1]$, so 
we have $E_3 \in \aA_X \cap \Coh(\overline{X})$. 
By the definition of $\aA_X$, 
any object in $\aA_X \cap \Coh(\overline{X})$ is torsion 
free, so $E_3$ is of the form $E_3=I_{Z_3}$
for a closed subscheme $Z_3 \subset \overline{X}$. 
Since $E_3 \in \aA_X$, it must be trivial along the divisor 
at the infinity, so $Z_3 \subset X$. 
\end{proof}
Similarly to Corollary~\ref{cor:K2}, 
we have the following corollary. 
\begin{cor}\label{cor:K2A}
The functors (\ref{descend:I}) induce the left action of 
the K-theoretic Hall-algebra
of zero dimensional sheaves to the
direct sum of MNOP categories for a fixed $\beta$
\begin{align*}
\bigoplus_{m\ge 0}
K(\mathcal{DT}^{\mathbb{C}^{\ast}}(\mM_X^{\sigma \sss}(m[\mathrm{pt}])))
\times 
\bigoplus_{n \in \mathbb{Z}}
K(\mathcal{DT}^{\mathbb{C}^{\ast}}(I_n(X, \beta))) 
\to \bigoplus_{n\in \mathbb{Z}}K(\mathcal{DT}^{\mathbb{C}^{\ast}}(I_n(X, \beta))). 
\end{align*}
\end{cor}

\subsection{Categorical DT theories for stable D0-D2-D6 bound states}
In this subsection, we recall 
categorical DT theories for stable 
D0-D2-D6 bound states. 
They depend on a choice of a stability 
parameter $t\in \mathbb{R}$, and 
the wall-crossing phenomena with respect to 
the above stability parameter is relevant for 
the rationality of generating series of PT invariants 
and GV formula. See~\cite{MR2892766} for details. 

Below, we fix an element $\sigma=iH$ in (\ref{sigma:BH})
for $B=0$ and an ample divisor $H$. 
We define the map $\mu$ by 
\begin{align*}
\mu \colon N_{\le 1}(S) \to \mathbb{Q} \cup \{\infty\}, \ 
(\beta, n) \mapsto \frac{n}{\beta \cdot H}. 
\end{align*}
Here $\mu(\beta, n)=\infty$ if the denominator is zero. 
For a non-zero $F \in \Coh_{\le 1}(X)$, 
it is $\sigma$-(semi)stable if and only if 
we have the inequality 
$\mu(F') \le (<) \mu(F)$ for any 
non-zero subsheaf $F' \subsetneq F$. 

For each $t \in \mathbb{R}$, we also 
define the following map 
\begin{align*}
\mu_t^{\dag} \colon 
\mathbb{Z} \oplus N_{\le 1}(S)
 \to \mathbb{R} \cup \{\infty\}, \ 
(r, \beta, n) \mapsto \left\{ \begin{array}{cl}
t, & r\neq 0, \\
\mu(\beta, n), & r=0.
\end{array} \right. 
\end{align*}
For $E \in \aA_X$, we set
$\mu_t^{\dag}(E)=\mu_t^{\dag}(\cl(E))$. 

\begin{defi}\label{def:stability}
An object $E \in \aA_X$ is $\mu_t^{\dag}$-(semi)
stable if for any exact sequence 
$0 \to E' \to E \to E'' \to 0$ in 
$\aA_X$ we have the inequality 
$\mu_t^{\dag}(E')<(\le) \mu_t^{\dag}(E'')$. 
\end{defi}
We have the substacks
\begin{align*}
P_n(X, \beta)_t 
\subset \mM_X^{\dag}(\beta, n)
\end{align*}
corresponding to $\mu_t^{\dag}$-stable objects. 
The result of~\cite[Proposition~3.17]{MR2683216} shows that 
the above substack is an algebraic space of finite type. 
Moreover there is a finite set of walls
$W \subset \mathbb{Q}$ such that
$P_n(X, \beta)_t$ is constant 
if $t$ lies on 
on a connected component of $\mathbb{R} \setminus W$.
We say that $t \in \mathbb{R}$ lies in a chamber if $t \notin W$. 

We define the following 
conical closed substack
\begin{align*}
\zZ_{\mu_t^{\dag}\us}(\beta, n) \cneq \mM_X^{\dag}(\beta, n) \setminus 
P_n(X, \beta)_t
\subset \mM_X^{\dag}(\beta, n). 
\end{align*}
Since $P_n(X, \beta)_t$ is of finite type, 
there is a 
derived open substack 
$\mathfrak{M}_S^{\dag}(\beta, n)^{\fin}
\subset \mathfrak{M}_S^{\dag}(\beta, n)$ 
of finite type
such that 
\begin{align}\label{Pt:qcompact}
P_n(X, \beta)_t \subset 
t_0(\Omega_{\mathfrak{M}_S^{\dag}(\beta, n)^{\fin}}[-1])
\subset \mM_X^{\dag}(\beta, n).  
\end{align}

\begin{defi}\label{def:cat026}\emph{(\cite[Definition~6.6]{TocatDT})}
The $\mathbb{C}^{\ast}$-equivariant 
DT theory for stable D0-D2-D6 bound states 
is defined by
\begin{align*}
\dD \tT^{\mathbb{C}^{\ast}}(P_n(X, \beta)_t) \cneq 
D^b_{\rm{coh}}(\mathfrak{M}_S^{\dag}(\beta, n)^{\fin})/ 
\cC_{\zZ_{\mu_t^{\dag}\us}(\beta, n)^{\fin}}.
\end{align*}
\end{defi}

\subsection{Actions on categorical DT theories for stable D0-D2-D6 bound states}
Here we construct actions of 
DT categories of one dimensional semistable sheaves
to the DT categories in Definition~\ref{def:cat026}. 
We fix $t=t_0 \in \mathbb{R}$, and 
set $t_{\pm}=t_0 \pm \varepsilon$ for $0<\varepsilon \ll 1$. 
Then for 
$\overline{\chi}=m+t_0 \in \mathbb{Q}[m]$, 
we have
\begin{align*}
N(S)_{\overline{\chi}}
=\{(\beta, n) \in N_{\le 1}(S) : \mu(\beta, n)=t_0\} \cup \{0\}.
\end{align*} 
We take $v_{\bullet}=(v_1, v_2, v_3) \in N_{\le 1}(S)$
by 
\begin{align*}
v_1=(\beta, n), \ v_2=(\beta+\beta', n+n'), \ 
v_3=(\beta', n') \in N(S)_{\overline{\chi}}. 
\end{align*}
We then take derived open substacks 
\begin{align*}
\fM_S^{\dag}(v_1)^{\fin} \subset
\fM_S^{\dag}(v_1), \ 
\fM_S^{\dag}(v_2)^{\fin} \subset \fM_S^{\dag}(v_2), \ 
\fM_S(v_3)^{\fin} \subset \fM_S(v_3)
\end{align*}
satisfying the conditions (\ref{Pt:qcompact})
for $v_1$, $v_2$, the condition  
(\ref{take:open}) for $v_3$ and (\ref{cond:ev1.5}). 
Then we have the following. 
\begin{thm}\label{thm:Pt1}
The functor (\ref{cat:COHA2.5})
descends to the functor
\begin{align*}
\dD \tT^{\mathbb{C}^{\ast}}(P_n(X, \beta)_{t_{-}}) 
\times \dD \tT^{\mathbb{C}^{\ast}}(\mM_{n'}^{\sigma \sss}(X, \beta'))
\to \dD \tT^{\mathbb{C}^{\ast}}(P_{n+n'}(X, \beta+\beta')_{t_{-}}). 
\end{align*}
\end{thm}
\begin{proof}
Similarly to Theorem~\ref{thm:PTaction}, the result follows from 
Lemma~\ref{exact:AP} (i). 
\end{proof}
Here we have used the following lemma, 
which is an analogue of Lemma~\ref{lem:PT}
and Lemma~\ref{lem:I}. 
\begin{lem}\label{exact:AP}
Let $0 \to E_1 \to E_2 \to E_3 \to 0$ be an 
exact sequence in $\aA_X$. 

(i) Suppose that $E_3=F_3[-1]$ for $F_3 \in \Coh_{\le 1}(X)$
with $\mu(F_3)=t_0$. 
Then if $E_2$ is $\mu_{t_-}^{\dag}$-stable, 
then $F_3$ is $\mu$-semistable and $E_1$ is $\mu_{t_-}^{\dag}$-stable. 

(ii) Suppose that $E_1=F_1[-1]$ for $F_1 \in \Coh_{\le 1}(X)$
with $\mu(F_1)=t_0$. 
Then if $E_2$ is $\mu_{t_+}^{\dag}$-stable, 
then $F_1$ is $\mu$-semistable and $E_2$ is $\mu_{t_+}^{\dag}$-stable. 
\end{lem}
\begin{proof}
We only prove (i), as the proof of (ii) is similar. 
First we show that $F_3$ is $\mu$-semistable. 
Let $F_3 \twoheadrightarrow F$ be a non-trivial 
surjection in $\Coh_{\le 1}(X)$. 
Then we have the surjection $E_2 \twoheadrightarrow F[-1]$ in 
$\aA_X$. 
By the $\mu_{t_-}$-stability of $E_2$, 
we have $\mu(F)>t_-$, hence $\mu(F) \ge t_0=\mu(F_3)$. 
Therefore $F_3$ is $\mu$-semistable. 

Next we show that $E_1$ is $\mu_{t_-}^{\dag}$-stable. 
Let $F[-1] \hookrightarrow E_1$ be an injection in $\aA_X$
for $F \in \Coh_{\le 1}(X)$. Then 
we have an injection $F[-1] \hookrightarrow E_2$ 
hence $\mu(F)<t_-$ by the $\mu_{t_-}$-stability of $E_2$. 
Let $E_1 \twoheadrightarrow F[-1]$ be a surjection in $\aA_X$. 
Then we have exact sequences in $\aA_X$
and $\Coh_{\le 1}(X)$
\begin{align*}
0 \to E_1' \to E_2 \to F'[-1] \to 0, \ 
0 \to F \to F' \to F_3 \to 0. 
\end{align*}
By the $\mu_{t_-}^{\dag}$-stability of $E_2$, we have 
$\mu(F')\ge t_0$. 
Then as $\mu(F_3)=t_0$, we have 
$\mu(F) \ge t_0>t_-$. 
Therefore $E_1$ is $\mu_{t-}$-stable. 
\end{proof}
We see that after crossing the wall at 
$t=t_0$, we obtain the left action of DT categories 
of one dimensional semistable sheaves. 
We take 
$v_{\bullet}=(v_1, v_2, v_3) \in N_{\le 1}(S)$
as
\begin{align*}
v_1=(\beta', n') \in N(S)_{\overline{\chi}}, 
\ v_2=(\beta+\beta', n+n'), \ 
v_3=(\beta, n). 
\end{align*}
We then take derived open substacks 
\begin{align*}
\fM_S(v_1)^{\fin} \subset
\fM_S(v_1), \ 
\fM_S^{\dag}(v_2)^{\fin} \subset \fM_S^{\dag}(v_2), \ 
\fM_S^{\dag}(v_3)^{\fin} \subset \fM_S^{\dag}(v_3)
\end{align*}
satisfying the conditions (\ref{Pt:qcompact}) for $v_2$, $v_3$,
the condition 
(\ref{take:open}) for $v_1$ and (\ref{cond:ev2}). 
Then we have the following. 
\begin{thm}\label{thm:Pt2}
The functor (\ref{cat:COHAI2.5})
descends to the functor
\begin{align}\notag
\dD \tT^{\mathbb{C}^{\ast}}(\mM_{n'}^{\sigma \sss}(X, \beta'))
\times 
\dD \tT^{\mathbb{C}^{\ast}}(P_n(X, \beta)_{t_{+}}) 
\to \dD \tT^{\mathbb{C}^{\ast}}(P_{n+n'}(X, \beta+\beta')_{t_{+}}). 
\end{align}
\end{thm}
\begin{proof}
Similarly to Theorem~\ref{thm:Iaction}, the result follows from 
Lemma~\ref{exact:AP} (ii). 
\end{proof}

As a corollary of Theorem~\ref{thm:Pt1} and Theorem~\ref{thm:Pt2}, 
we have the following:
\begin{cor}\label{cor:right/left}
For each $t_0 \in \mathbb{R}$, 
the K-theoretic 
Hall-algebra of one dimensional 
semistable sheaves 
with slope $t_0$
\begin{align*}
\bigoplus_{\mu(\beta, n)=t_0}
K(\mathcal{DT}^{\mathbb{C}^{\ast}}
(\mM_{X}^{\sigma\mathchar`-\rm{ss}}(\beta, n)))
\end{align*}
acts on 
\begin{align}\label{act:lr}
\bigoplus_{(\beta, n) \in N_{\le 1}(S)}
K(\dD \tT^{\mathbb{C}^{\ast}}(P_n(X, \beta)_{t_{-}})), \ 
\bigoplus_{(\beta, n) \in N_{\le 1}(S)}
K(\dD \tT^{\mathbb{C}^{\ast}}(P_n(X, \beta)_{t_{+}}))
\end{align}
from right, left, respectively. 
\end{cor}

\begin{exam}\label{exam:conifold}
Let $S \to \mathbb{C}^2$ be the blow-up at the origin, 
and $C \subset S$ is the exceptional curve.  
Then
\begin{align*}
X=\mathrm{Tot}_S(\omega_S)=
\mathrm{Tot}_{\mathbb{P}^1}(\oO_{\mathbb{P}^1}(-1)^{\oplus 2}) 
\end{align*}
is the resolved conifold. 
In this case, the set of walls 
is given by $W=\mathbb{Z}_{>0} \subset \mathbb{Q}$
since any one dimensional stable sheaf on $X$ is 
of the form $\oO_C(m)$.  
Then we set 
 \begin{align*}
P_n(X, d)_m \cneq P_n(X, d[C])_{t}, \ 
t \in (m, m+1)
\end{align*}
which makes sense by the 
above description of walls. 
By Corollary~\ref{cor:right/left}
for $t_0=m, m+1$, 
the direct sum for each $m \in \mathbb{Z}_{>0}$
\begin{align*}
\bigoplus_{(n, d)\in \mathbb{Z}^2}
K(\mathcal{DT}^{\mathbb{C}^{\ast}}(P_n(X, d)_m))
\end{align*}
admits right/left actions of the algebra (\ref{K:simple}). 
\end{exam}

\newcommand{\etalchar}[1]{$^{#1}$}
\providecommand{\bysame}{\leavevmode\hbox to3em{\hrulefill}\thinspace}
\providecommand{\MR}{\relax\ifhmode\unskip\space\fi MR }
\providecommand{\MRhref}[2]{%
  \href{http://www.ams.org/mathscinet-getitem?mr=#1}{#2}
}
\providecommand{\href}[2]{#2}


\begin{thebibliography}{BBBBJ15}

\bibitem[AG15]{MR3300415}
D.~Arinkin and D.~Gaitsgory, \emph{Singular support of coherent sheaves and the
  geometric {L}anglands conjecture}, Selecta Math. (N.S.) \textbf{21} (2015),
  no.~1, 1--199. \MR{3300415}

\bibitem[BBBBJ15]{MR3352237}
O.~B.~Bassat, C.~Brav, V.~Bussi, and D.~Joyce, \emph{A
  `{D}arboux theorem' for shifted symplectic structures on derived {A}rtin
  stacks, with applications}, Geom. Topol. \textbf{19} (2015), no.~3,
  1287--1359. \MR{3352237}

\bibitem[BBD{\etalchar{+}}15]{MR3353002}
C.~Brav, V.~Bussi, D.~Dupont, D.~Joyce, and B.~Szendr\H~oi, \emph{Symmetries
  and stabilization for sheaves of vanishing cycles}, J. Singul. \textbf{11}
  (2015), 85--151, With an appendix by J\"org Sch\"urmann. \MR{3353002}

\bibitem[Beh09]{Beh}
K.~Behrend, \emph{Donaldson-{T}homas type invariants via microlocal geometry},
  Ann.~of Math \textbf{170} (2009), 1307--1338.

\bibitem[BIK08]{MR2489634}
D.~Benson, S.~B. Iyengar, and H.~Krause, \emph{Local cohomology
  and support for triangulated categories}, Ann. Sci. \'{E}c. Norm. Sup\'{e}r.
  (4) \textbf{41} (2008), no.~4, 573--619. \MR{2489634}

\bibitem[BJM]{BDM}
V.~Bussi, D.~Joyce, and S.~Meinhardt, \emph{On motivic vanishing cycles of
  critical loci}, preprint, arXiv:1305.6428.

\bibitem[Bri11]{MR2813335}
T.~Bridgeland, \emph{Hall algebras and curve-counting invariants}, J. Amer.
  Math. Soc. \textbf{24} (2011), no.~4, 969--998. \MR{2813335}

\bibitem[Dav17]{MR3667216}
B.~Davison, \emph{The critical {C}o{HA} of a quiver with potential}, Q. J.
  Math. \textbf{68} (2017), no.~2, 635--703. \MR{3667216}

\bibitem[DG13]{MR3037900}
V.~Drinfeld and D.~Gaitsgory, \emph{On some finiteness questions for
  algebraic stacks}, Geom. Funct. Anal. \textbf{23} (2013), no.~1, 149--294.
  \MR{3037900}

\bibitem[GR17]{MR3701353}
D.~Gaitsgory and N.~Rozenblyum, \emph{A study in derived algebraic
  geometry. {V}ol. {II}. {D}eformations, {L}ie theory and formal geometry},
  Mathematical Surveys and Monographs, vol. 221, American Mathematical Society,
  Providence, RI, 2017. \MR{3701353}

\bibitem[Gro96]{MR1386846}
I.~Grojnowski, \emph{Instantons and affine algebras. {I}. {T}he {H}ilbert
  scheme and vertex operators}, Math. Res. Lett. \textbf{3} (1996), no.~2,
  275--291. \MR{1386846}

\bibitem[HRS96]{HRS}
D.~Happel, I.~Reiten, and S.~O. Smal$\o$, \emph{Tilting in abelian categories
  and quasitilted algebras}, Mem.~Amer.~Math.~Soc, vol. 120, 1996.

\bibitem[Joy]{Jslide}
D.~Joyce, \emph{Shifted symplectic geometry, {C}alabi-{Y}au moduli spaces, and
  generalizations of {D}onaldson-{T}homas theory: our current and future
  research}, Talks given Oxford, October 2013, at a workshop for EPSRC
  Programme Grant research group,
  https://people.maths.ox.ac.uk/joyce/PGhandout.pdf.

\bibitem[JT17]{MR3607000}
Y.~Jiang and R.~P.~Thomas, \emph{Virtual signed {E}uler
  characteristics}, J. Algebraic Geom. \textbf{26} (2017), no.~2, 379--397.
  \MR{3607000}

\bibitem[KS]{K-S}
M.~Kontsevich and Y.~Soibelman, \emph{Stability structures, motivic
  {D}onaldson-{T}homas invariants and cluster transformations}, preprint,
  arXiv:0811.2435.

\bibitem[KS11]{MR2851153}
M.~Kontsevich and Y.~Soibelman, \emph{Cohomological {H}all algebra,
  exponential {H}odge structures and motivic {D}onaldson-{T}homas invariants},
  Commun. Number Theory Phys. \textbf{5} (2011), no.~2, 231--352. \MR{2851153}

\bibitem[KV]{KaVa2}
M.~Kapranov and E.~Vasserot, \emph{The cohomological {H}all algebra of a
  surface and factorization cohomology}, arXiv:1901.07641.

\bibitem[Nak97]{MR1441880}
H.~Nakajima, \emph{Heisenberg algebra and {H}ilbert schemes of points on
  projective surfaces}, Ann. of Math. (2) \textbf{145} (1997), no.~2, 379--388.
  \MR{1441880}

\bibitem[Neg19]{Negut}
A.~Negu\c{t}, \emph{Shuffle algebras associated to surfaces}, Selecta Math.
  (N.S.) \textbf{25} (2019), no.~3, Art. 36, 57. \MR{3950703}

\bibitem[P{\u{a}}d]{Tudor}
T.~P{\u{a}}durairu, \emph{K-theoretic {H}all algebras for quivers with
  potential}, arXiv:1911.05526.

\bibitem[PS]{PoSa}
M.~Porta and F.~Sala, \emph{Two dimensional categorified 
{H}all algebras}, arXiv:1903.07253.


\bibitem[PT09b]{MR2545686}
R.~Pandharipande and R.~P. Thomas, \emph{Curve counting via stable pairs in the
  derived category}, Invent. Math. \textbf{178} (2009), no.~2, 407--447.
  \MR{2545686}

\bibitem[PT14]{MR3221298}
\bysame, \emph{13/2 ways of counting curves}, Moduli spaces, London Math. Soc.
  Lecture Note Ser., vol. 411, Cambridge Univ. Press, Cambridge, 2014,
  pp.~282--333. \MR{3221298}

\bibitem[Ren18]{MR3807309}
J.~V.~Rennemo, \emph{Homology of {H}ilbert schemes of points on a
  locally planar curve}, J. Eur. Math. Soc. (JEMS) \textbf{20} (2018), no.~7,
  1629--1654. \MR{3807309}

\bibitem[RS17]{MR3727563}
J.~Ren and Y.~Soibelman, \emph{Cohomological {H}all algebras, semicanonical
  bases and {D}onaldson-{T}homas invariants for 2-dimensional {C}alabi-{Y}au
  categories (with an appendix by {B}en {D}avison)}, Algebra, geometry, and
  physics in the 21st century, Progr. Math., vol. 324, Birkh\"{a}user/Springer,
  Cham, 2017, pp.~261--293. \MR{3727563}

\bibitem[Tho00]{MR1818182}
R.~P. Thomas, \emph{A holomorphic {C}asson invariant for {C}alabi-{Y}au
  3-folds, and bundles on {$K3$} fibrations}, J. Differential Geom. \textbf{54}
  (2000), no.~2, 367--438. \MR{1818182}


\bibitem[TV07]{MR2493386}
B.~To\"{e}n and M.~Vaqui\'{e}, \emph{Moduli of objects in
  dg-categories}, Ann. Sci. \'{E}cole Norm. Sup. (4) \textbf{40} (2007), no.~3,
  387--444. \MR{2493386}

\bibitem[To\"{e}11]{MR2762557}
B.~To\"{e}n, \emph{Lectures on dg-categories}, Topics in algebraic and
  topological {$K$}-theory, Lecture Notes in Math., vol. 2008, Springer,
  Berlin, 2011, pp.~243--302. \MR{2762557}

\bibitem[To\"{e}12]{Tproper}
\bysame, \emph{Proper local complete intersection morphisms preserve perfect
  complexes}, arXiv:1210.2827 (2012).

\bibitem[To\"{e}14a]{MR3285853}
\bysame, \emph{Derived algebraic geometry}, EMS Surv. Math. Sci. \textbf{1}
  (2014), no.~2, 153--240. \MR{3285853}

\bibitem[To\"{e}14b]{MR3728637}
\bysame, \emph{Derived algebraic geometry and deformation quantization},
  Proceedings of the {I}nternational {C}ongress of {M}athematicians---{S}eoul
  2014. {V}ol. {II}, Kyung Moon Sa, Seoul, 2014, pp.~769--792. \MR{3728637}



\bibitem[Tod]{TocatDT}
Y.~Toda, \emph{On categorical {D}onaldson-{T}homas theory for local surfaces},
  arXiv:1907.09076.

\bibitem[Tod10a]{MR2669709}
\bysame, \emph{Curve counting theories via stable objects {I}. {DT}/{PT}
  correspondence}, J. Amer. Math. Soc. \textbf{23} (2010), no.~4, 1119--1157.
  \MR{2669709}

\bibitem[Tod10b]{MR2683216}
\bysame, \emph{Generating functions of stable pair invariants via
  wall-crossings in derived categories}, New developments in algebraic
  geometry, integrable systems and mirror symmetry ({RIMS}, {K}yoto, 2008),
  Adv. Stud. Pure Math., vol.~59, Math. Soc. Japan, Tokyo, 2010, pp.~389--434.
  \MR{2683216}

\bibitem[Tod12]{MR2892766}
\bysame, \emph{Stability conditions and curve counting invariants on
  {C}alabi-{Y}au 3-folds}, Kyoto J. Math. \textbf{52} (2012), no.~1, 1--50.
  \MR{2892766}



\bibitem[YZ]{YangZhao}
Y.~Yang and G.~Zhao, \emph{On two cohomological {H}all algebras},
  arXiv:1604.01477.

\bibitem[Zha]{Zhao}
Y.~Zhao, \emph{On the {K}-theoretic {H}all algebra of a surface},
  arXiv:1901.00831.

\end{thebibliography}

\vspace{5mm}

Kavli Institute for the Physics and 
Mathematics of the Universe (WPI), University of Tokyo,
5-1-5 Kashiwanoha, Kashiwa, 277-8583, Japan.

\textit{E-mail address}: yukinobu.toda@ipmu.jp

\end{document}